\documentclass[10pt, a4paper, reqno]{article}

\usepackage{fullpage}

\usepackage{mathtools}

\usepackage{enumitem}
\usepackage{hyperref}
\usepackage{amsthm,amsfonts,amssymb,euscript,mathrsfs}

\usepackage{color}
\usepackage{xcolor}

\theoremstyle{plain}

\newtheorem{lemma}{Lemma}[section]
\newtheorem{thm}{Theorem}[section]
\newtheorem{prop}{Proposition}[section]

\newtheorem{mydef}{Definition}[section]
\newtheorem{remark}{Remark}[section]
\numberwithin{equation}{section}

\newcommand{\non}{\nonumber}

\usepackage{stmaryrd}
\date{}

\renewcommand{\d}{\mathrm{d}}
\newcommand{\R}{\mathbb{R}}
\newcommand{\Z}{\mathbb{Z}}
\newcommand{\ffi}{\varphi}

\newcommand{\e}{\varepsilon}
\newcommand{\dr}{\partial}

\newcommand{\dive}{\mathrm{div}}
\renewcommand{\div}{\mathrm{div}}

\newcommand{\Ll}{\mathscr{L}}
\newcommand{\tr}{\mathrm{tr}}
\newcommand{\T}{\mathbf{T}}

\renewcommand{\L}{\mathrm{L}}

\newcommand{\Lb}{\underline{L}}
\newcommand{\D}{\mathbf{D}}

\newcommand{\half}{\frac{1}{2}}

\newcommand{\h}{\mathfrak{h}}
\newcommand{\F}{\mathfrak{F}}
\newcommand{\g}{\mathbf{g}}
\renewcommand{\gg}{\mathfrak{g}}
\newcommand{\A}{\mathbf{A}}
\newcommand{\B}{\mathbf{B}}
\newcommand{\C}{\mathbf{C}}

\newcommand{\tBox}{\Tilde{\Box}}

\newcommand{\la}{\lambda}
\renewcommand{\a}{\alpha}
\renewcommand{\b}{\beta}
\newcommand{\si}{\sigma}
\newcommand{\HH}{\mathcal{H}}
\newcommand{\MM}{\mathcal{M}}
\newcommand{\LL}{\mathcal{L}}
\newcommand{\RR}{\mathcal{R}}

\newcommand{\nab}{\nabla}
\newcommand{\Si}{\Sigma}
\newcommand{\ga}{\gamma}
\newcommand{\Ga}{\Gamma}
\newcommand{\ka}{\kappa}
\renewcommand{\th}{\theta}
\newcommand{\De}{\Delta}
\newcommand{\vDe}{\dot{\De}}
\newcommand{\tffi}{\widetilde{\varphi}}
\newcommand{\tX}{\widetilde{X}}
\newcommand{\PP}{\mathcal{P}}
\renewcommand{\i}{\mathbf{i}}
\renewcommand{\j}{\mathbf{j}}
\newcommand{\de}{\delta}
\newcommand{\osc}{\mathrm{osc}}

\newcommand{\Pol}[2]{\mathscr{P}\left[ #1 \middle| #2\right]}
\newcommand{\GO}[1]{O\left( #1 \right)}
\renewcommand{\l}{\left\|}
\renewcommand{\r}{\right\|}
\newcommand{\enstq}[2]{\left\{#1~\middle|~#2\right\}} 
\newcommand{\saut}{\par\leavevmode\par}
\newcommand{\pth}[1]{\left( #1 \right)}
\newcommand{\cro}[1]{\left\{ #1 \right\}}

\DeclareFontFamily{U}{mathx}{\hyphenchar\font45}
\DeclareFontShape{U}{mathx}{m}{n}{
      <5> <6> <7> <8> <9> <10>
      <10.95> <12> <14.4> <17.28> <20.74> <24.88>
      mathx10
      }{}
\DeclareSymbolFont{mathx}{U}{mathx}{m}{n}
\DeclareFontSubstitution{U}{mathx}{m}{n}
\DeclareMathAccent{\widecheck}{0}{mathx}{"71}

\newcommand{\wc}[1]{\widecheck{ #1 }}

\title{The reverse Burnett conjecture for null dusts}
\author{Arthur Touati \thanks{CNRS \& Institut de Mathématiques de Bordeaux, Bordeaux, France (\href{mailto:arthur.touati@math.u-bordeaux.fr}{arthur.touati@math.u-bordeaux.fr})}}

\begin{document}

\maketitle

\begin{abstract}
Given a regular solution $\g_0$ of the Einstein-null dusts system without restriction on the number of dusts, we construct families of solutions $(\g_\la)_{\la\in(0,1]}$ of the Einstein vacuum equations such that $\g_\la-\g_0$ and $\dr(\g_\la-\g_0)$ converges respectively strongly and weakly to 0 when $\la\to0$. Our construction, based on a multiphase geometric optics ansatz, thus extends the validity of the reverse Burnett conjecture without symmetry to a large class of massless kinetic spacetimes. In order to deal with the finite but arbitrary number of direction of oscillations we work in a generalised wave gauge and control precisely the self-interaction of each wave but also the interaction of waves propagating in different null directions, relying crucially on the non-linear structure of the Einstein vacuum equations. We also provide the construction of oscillating initial data solving the vacuum constraint equations and which are consistent with the spacetime ansatz.
\end{abstract}

\tableofcontents

\section{Introduction}

In this article we construct oscillatory solutions of the Einstein vacuum equations
\begin{align}
R_{\mu\nu}(\g)=0\label{EVE},
\end{align}
where $\g$ is a Lorentzian metric on a 4-dimensional manifold $\mathcal{M}$ and $R_{\mu\nu}(\g)$ is the Ricci tensor of $\g$. These solutions, denoted $\g_\la$, depend on a small wavelength $\la>0$ and oscillate at frequency $\la^{-1}$. We are particularly interested in their behaviour in the high-frequency limit $\la\to 0$ and our construction, based on a geometric optics expansion of the solutions, seems to be the first to allow a finite but arbitrary number of directions of oscillation with no symmetry assumption. It builds on previous works by the author \cite{Touati2023a,Touati2023b}, where only one direction of oscillation was allowed. It is motivated by Burnett's conjecture in general relativity and by the study of high-frequency gravitational waves initiated by Choquet-Bruhat. We discuss further these two aspects in Sections \ref{section burnett} and \ref{section GO} below where we also give two rough versions of our main result (see Theorem \ref{theo evol} for the precise version) from these two perspectives.

\subsection{The Burnett conjecture}\label{section burnett}

The first motivation of our work is the description of the closure of the set of solutions to \eqref{EVE} under appropriate weak limits. This question was first raised by Burnett in \cite{Burnett1989}. In this article, he considers sequences $(\g_\la)_{\la\in(0,1]}$ of solutions to \eqref{EVE} on a given manifold $\mathcal{M}$ such that the following convergence holds when $\la\to 0$:
\begin{equation}\label{CV}
\begin{aligned}
\g_\la & \rightarrow \g_0 \quad\text{uniformly on compact sets},
\\ \dr\g_\la & \rightharpoonup \dr \g_0 \quad\text{weakly in $L^2_{loc}$},
\end{aligned}
\end{equation}
where $\dr$ denotes a partial derivative in any coordinates system on $\mathcal{M}$ and where $\g_0$ is a metric on $\mathcal{M}$. The question of interest here is the following: what matter model can the limit metric $\g_0$ describe? More precisely, what are the possible effective stress-energy tensors $R_{\mu\nu}(\g_0) - \half R(\g_0)(\g_0)_{\mu\nu}$? If the convergence of the derivatives $\dr\g_\la$ towards $\dr\g_0$ is strong then $\g_0$ also solves \eqref{EVE}, but if this convergence is only weak then products of derivatives of the metric in the Ricci tensor might produce backreaction, i.e a non-trivial effective stress-energy tensor. Burnett proposed a two-sided conjecture to answer the above questions:
\begin{itemize}
\item \textbf{The direct Burnett conjecture.} The effective stress-energy tensor is the one of a massless Vlasov field on $\mathcal{M}$, i.e there exists a density $f_0:\mathcal{M}\times T\mathcal{M}\longrightarrow \R$ such that $(\g_0,f_0)$ solves the massless Einstein-Vlasov system on $\mathcal{M}$
\begin{equation}\label{EV}
\left\{
\begin{aligned}
R_{\mu\nu}(\g_0) - \half R(\g_0)(\g_0)_{\mu\nu} & = \int_{\g_0(p,p)=0}f_0^2 p_\mu p_\nu,
\\ p^\alpha \dr_\alpha f_0 - p^\alpha p^\beta \Gamma(g_0)^\rho_{\alpha\beta} \dr_{p^\rho}f_0 & =0.
\end{aligned}
\right.
\end{equation}
\item \textbf{The reverse Burnett conjecture.} For every solution $(\mathcal{M},\g_0,f_0)$ of \eqref{EV}, there exists a sequence $(\g_\la)_{\la\in(0,1]}$ of solutions of \eqref{EVE} on $\mathcal{M}$ with the convergence properties \eqref{CV}.
\end{itemize}
Put together, these two parts of Burnett's conjecture affirm that the closure of the set of solutions to \eqref{EVE} for the topology of the convergence \eqref{CV} identifies with the set of solutions to \eqref{EV}. 

\saut
The direct conjecture has been first considered in \cite{Green2011} and then proved in $\mathbb{U}(1)$ symmetry, i.e when $\mathcal{M}$ admits a spacelike Killing field, in \cite{Huneau2019} and \cite{Guerra2021}, where the effective stress-energy tensor is identified by means of microlocal defect measures. The reverse conjecture, which is this article's topic, is opposite in spirit to the direct one since it requires solving \eqref{EVE} with a given target $\g_0$. However, not only does it give a more complete and satisfying understanding of backreaction as a generic phenomenon, it is also highly connected to the direct conjecture. Indeed, the lack of strong convergence in \eqref{CV} does not allow boundedness of the sequence $(\g_\la)_{\la\in(0,1]}$ in $H^2_{loc}$ so that the celebrated bounded $L^2$ curvature conjecture proved in \cite{Klainerman2015} does not apply and the mere existence of sequences $(\g_\la)_{\la\in(0,1]}$ of solutions to \eqref{EVE} displaying the pathological behaviour \eqref{CV} is not guaranteed by the general theory. Solving the reverse conjecture precisely amounts to producing examples of such sequences and thus prevents the theorems on the direct conjecture to be empty.

\saut
The reverse conjecture has been proved first in $\mathbb{U}(1)$ symmetry in \cite{Huneau2018a} where the authors consider targets $\g_0$ solving a discretized version of \eqref{EV}, namely the Einstein-null dusts system
\begin{equation}\label{null dusts}
\left\{
\begin{aligned}
R_{\mu\nu}(\g_0) & = \sum_\A F_\A^2 \dr_\mu u_\A \dr_\nu u_\A,
\\ \g_0^{-1}(\d u_\A , \d u_\A) &  = 0,
\\ -2 L_\A F_\A + (\Box_{\g_0}u_\A)F_\A & = 0,
\end{aligned}
\right.
\end{equation}
where $\A$ runs through a finite set $\mathcal{A}$ of arbitrary cardinal denoted $|\mathcal{A}|$ in the sequel. In the setting of $\mathbb{U}(1)$ symmetry, the same authors prove in \cite{Huneau2024} that one can take the limit $|\mathcal{A}|\to+\infty$ and reach generic solutions to \eqref{EV} as targets for backreaction. In both these works, the $\mathbb{U}(1)$ symmetry plays an important role since it allows for the construction of an elliptic gauge, in which \eqref{EVE} reduces to a wave-map system coupled to semi-linear elliptic equations for the metric.

\saut
Outside of symmetry, the two sides of Burnett's conjecture have been tackled in double null gauge in \cite{Luk2020}, where the authors also shed new lights on the well-posedness theory for null dust shell and the formation of trapped surfaces. Their proof relies on a low-regularity result for \eqref{EVE} in double null gauge proved in \cite{Luk2017}, where derivatives of the metric in the two null directions are allowed to be only in $L^2_{loc}$. Due to the structure of \eqref{EVE} in double null gauge, this has to be compensated by higher regularity in the angular directions. Therefore, \cite{Luk2020} proves both sides of Burnett's conjecture with limits $\g_0$ solving \eqref{null dusts} with $|\mathcal{A}|=2$ and thus leaves open the question of whether there exists a setting for \eqref{EVE} capable of handling more than two null dusts. The present article's main goal is to show that wave gauges provide such a setting. From this perspective, we obtain the following:

\begin{thm}\label{theo rough burnett}[Rough version of the result from the point of view of Burnett's conjecture]
The reverse Burnett conjecture holds true when $\g_0$ is a regular solution of \eqref{null dusts} in wave gauge with arbitrary $|\mathcal{A}|<+\infty$.
\end{thm}

The intuitive reason behind the fact that wave gauges are good candidates to handle the superposition of more than two null dusts is that in these gauges the Einstein vacuum equations \eqref{EVE} rewrite as a system of quasi-linear wave equations which doesn't single out any null direction, as opposed to double null gauges. We mention some limitations of the present work and future directions of research:
\begin{itemize}
\item As we will see, in this article we construct sequences $(\g_\la)_{\la\in(0,1]}$ where the desired lack of strong convergence is purely due to oscillations. One of the strengths of \cite{Luk2020} is to include concentration effects as the cause of weak convergence. It is an interesting open problem to prove Burnett's reverse conjecture outside of symmetry with $|\mathcal{A}|\geq 3$ and with concentration-based backreaction.
\item More importantly, this article completely leaves open the question of sending the number of null dusts to infinity and thus proving the reverse Burnett conjecture with the highest level of generality. In particular, our local existence result in generalised wave gauge is not uniform in $|\mathcal{A}|$ and obtaining such uniformity seems to require ideas beyond the present work.
\item The regularity mentioned in Theorem \ref{theo rough burnett} refers to both spacetime regularity of the metric and the densities $F_\A$ and also regularity of the foliations induced by the eikonal functions $u_\A$ in \eqref{null dusts} (see Section \ref{section BG} for the precise definition of the class of allowed background spacetime). Proving the reverse Burnett conjecture with relaxed assumptions on the background is a very interesting problem, and we mention for instance the case of measure-valued null dusts as considered in \cite{Luk2020} or the presence of caustics in the null foliations (see the reviews \cite{Huneau2024a,Touati2025} for more open problems related to both the direct and reverse Burnett conjecture).
\end{itemize}

\subsection{Multiphase geometric optics}\label{section GO}

The second motivation of our work is the multiphase geometric optics approximation for the Einstein vacuum equations \eqref{EVE}. Geometric optics seeks a description of how waves propagate as solutions to a given system of PDEs. In general relativity, one can wonder how gravitational waves propagate as solutions to \eqref{EVE}. As explained in depth in Chapter 35 of \cite{Misner1973}, one would also like to go beyond the linearized gravity setting which by definition cannot describe the energy of gravitational waves, a quadratic quantity by nature, and thus misses the global impact of gravitational waves on a background spacetime. The framework of geometric optics, as presented in \cite{Rauch2012} or \cite{Metivier2009}, provides such a setting, as we will now describe.

\saut
In \cite{ChoquetBruhat1969}, Choquet-Bruhat is the first to construct WKB approximate solutions of \eqref{EVE}. More precisely, she constructs a family of metrics of the form
\begin{align}\label{rough CB}
\g_\la(x) & = \g_0(x) + \la \g^{(1)}(x,\th)_{|_{\th=\frac{\phi(x)}{\la}}} + \la^2 \g^{(2)}(x,\th)_{|_{\th=\frac{\phi(x)}{\la}}} 
\end{align} 
where the $\g^{(i)}$'s are periodic in the $\th$ variable and such that $R_{\mu\nu}(\g_\la) = \GO{\la}$. The typical geometric optics phenomena are recovered: $\phi$ must solve the background eikonal equation and $\g^{(1)}$ is transported along the rays. Less common is the fact that there should exist a scalar function $\tau>0$ such that $R_{\mu\nu}(\g_0) = \tau \dr_\mu \phi \dr_\nu \phi$, thus making \cite{ChoquetBruhat1969} one of the earliest examples of backreaction. As pointed out in \cite{Metivier2009}, it remained to prove or disprove the stability of the geometric optics approximation, i.e to answer the following: does the approximate solution \eqref{rough CB} stay close to an exact one on a uniform time scale in $\la$? A positive answer has been given in the articles \cite{Touati2023a,Touati2023b}. As the approximate construction of \cite{ChoquetBruhat1969}, these articles deal with the singlephase geometric optics approximation, when only one direction of oscillation is allowed. We extend these results to the following:

\begin{thm}\label{theo rough GO}[Rough version of the result from the point of view of geometric optics]
Under a strong coherence assumption, the multiphase geometric optics approximation is stable for \eqref{EVE}.
\end{thm}

Note that this result cannot follow from general results on multiphase geometric optics such as \cite{Joly1993} mainly because the hyperbolicity of \eqref{EVE} comes at the cost of a gauge choice, due to the invariance by diffeomorphism of the Ricci tensor. Here, the gauge is part of the geometric optics construction and in particular the metrics $\g_\la$ will be such that $\Box_{\g_\la}x^\a$, i.e the term defining the wave gauge, is oscillating and of order $\la$. Note that geometric optics for semi-linear gauge invariant equations have been studied in \cite{Jeanne2002}, here the situation is different since \eqref{EVE} are quasi-linear.

\subsection{A discussion of transparency}\label{section transparence}

The two motivations described above are obviously very much connected. Both are concerned with describing the non-linear interactions through \eqref{EVE} of small scale inhomogeneities in the metric, and geometric optics is used as a strategy to attack Burnett's reverse conjecture. The non-linear structure of the Ricci tensor plays a central role in our proof, and we would like to bring to the reader's attention a truly astounding aspect of \eqref{EVE}, illustrated by Theorems \ref{theo rough burnett} and \ref{theo rough GO}. 

\saut
As it was already the case in the approximate construction of \cite{ChoquetBruhat1969}, the transport equation along the rays of the optical function for the first profile in the WKB ansatz constructed in the present article is linear. At first glance, this is very surprising since the Einstein vacuum equations have a complicated non-linear structure. In the geometric optics literature, the fact that the transport equations for the first profile in a WKB ansatz turns out to be linear despite the waves interacting non-linearly is called transparency, see \cite{Joly2000}. Note that transparency can also be seen from Burnett's conjecture's perspective, since the transport equation for the density in \eqref{EV} is linear.  However, Burnett's conjecture describes a truly non-linear effect since the effective stress-energy tensor, or alternatively the energy of the gravitational wave $\tau$ in Choquet-Bruhat's approximate construction, precisely originates in the quadratic self-interaction of derivatives of the metric. It is truly astounding that this interaction produces a global quadratic effect while being linearly propagated. 

\saut
As pointed out in \cite{Lannes2013}, transparency is directly linked to the null condition introduced in \cite{Christodoulou1986,Klainerman1986} for the study of global existence for small data for non-linear wave equations. Elaborating on \cite{ChoquetBruhat1969} in her article \cite{ChoquetBruhat2000}, Choquet-Bruhat uses geometric optics to show that \eqref{EVE} cannot truly satisfy the null condition, otherwise we would have $\tau=0$ above or alternatively Burnett's conjecture would reduce to an uninteresting statement about vacuum spacetimes necessarily approaching vacuum spacetimes. Later, Lindblad and Rodnianski introduced in \cite{Lindblad2003} a weakened version of the null condition, rightly called the weak null condition. They show that \eqref{EVE} in wave gauge satisfy this condition and uses this to prove the stability of Minkowski in wave gauge in \cite{Lindblad2010}, a result already proved in the seminal \cite{Christodoulou1993}. However, as noticed in Section 3.1.6 in \cite{Touati2023a}, there exist systems with the weak null condition but without the transparency property. The Einstein vacuum equations \eqref{EVE} thus seem to be very much unique in the sense that both a linear and a non-linear behaviour can be simultaneously exhibited.

\subsection{Notations and tools}\label{section notation}

In this section we introduce various notations and tools which will be used throughout the article.

\paragraph{Geometric notations.} 
\begin{itemize}
\item Our construction takes place on the fixed manifold $\mathcal{M}\vcentcolon=[0,1]\times\R^3$, endowed with the standard coordinates $(t,x^1,x^2,x^3)$. We define $\Si_t\vcentcolon=\{t\}\times \R^3$ for $t\in[0,1]$. On $\mathcal{M}$ we denote by $\mathbf{m}$ the Minkowski metric while on each $\Si_t$ the Euclidean metric is denoted by $e$. Greek indices will refer to the coordinates $(t,x^1,x^2,x^3)$ and will thus run from 0 to 3. Latin indices will refer to the coordinates $(x^1,x^2,x^3)$ and will thus run from 1 to 3. Repeated indices (with one up and one down) will be always summed over. If $r\geq 0$ then $B_r$ denotes the closed ball in $\R^3$ centered at 0 and of radius $r$ in the Euclidean metric, i.e $B_r=\{|x|\leq r \}$. 
\item If $T$ and $S$ are symmetric 2-tensors and if $\g$ is a Lorentzian metric on $\mathcal{M}$ then we define
\begin{align*}
|T\cdot S|_{\g} & = \g^{\a\b}\g^{\mu\nu}T_{\a\mu}S_{\b\nu},
\\ |T|_{\g} & = \sqrt{ \g^{\a\b}\g^{\mu\nu}T_{\a\mu}T_{\b\nu}},
\\ \tr_\g T & = \g^{\a\b}T_{\a\b}.
\end{align*}
These notations have their natural 1-tensor equivalent: $|X\cdot Y|_\g = \g^{\a\b}X_\a Y_\b$. We don't distinguish between the covariant or contravariant forms of tensors, and if $T$ is a general 2-tensor then its symmetric and anti-symmetric part are denoted
\begin{align*}
T_{(\a\b)} & = T_{\a\b} + T_{\b\a},
\\ T_{[\a \b]} & = T_{\a\b} - T_{\b\a}.
\end{align*}
All these notations extend naturally to tensors defined only on $\Si_0$ endowed with a Riemannian metric. 
\item For $f$ a scalar function on $\mathcal{M}$, $\nab f$ will denote the Euclidean gradient of $f$ or, alternatively and with a slight abuse of notations, any spatial derivatives $\dr_i f$ while $\dr f$ will denote any spacetime derivatives $\dr_\a f$. We associate to a Lorentzian metric $\g$ on $\mathcal{M}$ its wave operator $\Box_\g$ which acts on scalar function as $\Box_\g f  = \g^{\mu\nu}\pth{ \dr_\mu \dr_\nu f - \Ga(\g)^{\rho}_{\mu\nu}\dr_\rho f}$, where $\Ga(\g)^{\rho}_{\mu\nu}$ denotes the usual Christoffel symbols of the metric $\g$. If $h$ is a Riemannian metric the same definition in spacelike coordinates defined its Laplace-Beltrami operator $\De_h$. If $v$ is a scalar function on $\Si_0$ then we define its gradient with respect to the metric $h$ by the vector field $\nab_h v = h^{ij}\dr_i v \dr_j$.
\end{itemize}

\paragraph{Analytic tools.}
\begin{itemize}
\item On each slice $\Si_t$ we define the usual Lebesgues and Sobolev spaces $L^p$ and $W^{k,p}$ with respect to the Euclidean element of integration. We also define the weighted Sobolev spaces $W^{k,p}_\de$ to be the completion of smooth and compactly supported functions on $\Si_t$ for the norm
\begin{align}
\l f \r_{W^{k,p}_\de} = \sum_{0\leq |m| \leq k} \l \pth{1+|x|^2}^{\frac{\de+|m|}{2}} \nab^m f  \r_{L^p}\label{norm WSS}
\end{align}
with the standard special case $H^k_\de\vcentcolon = W^{k,2}_\de$ and $L^2_\de=H^0_\de$. By replacing $L^p$ in \eqref{norm WSS} by $L^\infty$ we define the weighted Hölder spaces $C^k_\de$. These norms are extended to tensors by summing over their components in the coordinates $(t,x^1,x^2,x^3)$.
\item From \cite{ChoquetBruhat2009}, these spaces satisfy the continuous embeddings
\begin{align*}
W^{s_1,p}_{\de_1}\times W^{s_2,p}_{\de_2} \subset W^{s,p}_\de
\end{align*}
where $s\leq \min(s_1,s_2)$, $s<s_1+s_2-\frac{3}{p}$ and $\de<\de_1+\de_2+\frac{3}{p}$ and
\begin{align*}
W^{s,p}_\de \subset C^m_{\de+\frac{3}{p}}
\end{align*}
where $m<s-\frac{3}{p}$. From \cite{ChoquetBruhat2009} we also get that $\De:H^2_\de\longrightarrow L^2_{\de+2}$ is an isomorphism if $-\frac{3}{2}<\de<-\half$ and where $\De$ is the flat Laplacian, sometimes denoted $\De_e$.
\item Estimates with no mention of the time will always refer to estimates holding uniformly in time over $[0,1]$, i.e $\l f\r_X \leq C$ will denote 
\begin{align*}
\sup_{t\in[0,1]}\l f\r_{X(\Si_t)} \leq C,
\end{align*}
where $X$ is any of the function spaces defined here.
\end{itemize}

\paragraph{High-frequency and schematic notations.} Since our construction is based on a geometric optics expansion of the metric, we will encounter objects defined or expressed via an expansion in terms of $\la$ with oscillating coefficients, where by object we mean tensors of any types including scalar functions, vector fields, 1-forms and higher order tensors. We introduce some notations to describe and manipulate these expansions.
\begin{itemize}
\item If an object $S$ admits an expansion in terms of powers (non-negative or negative) of $\la$, we denote by $S^{(i)}$ the coefficient of $\la^i$ in this expansion. If $j\in\Z$, we define 
\begin{align*}
S^{(\geq j)} = \sum_{k\geq j}\la^{k-j}S^{(k)},
\end{align*}
so that
\begin{align*}
S = \sum_{k\leq j-1} \la^k S^{(k)} + \la^j S^{(\geq j)}.
\end{align*}
\item A coefficient $S^{(i)}$ in the expansion of an object $S$ might oscillate at the frequency $\la^{-1}$, i.e be a linear combination of terms of the form
\begin{align*}
T\pth{\frac{z}{\la}}S^{(i,T,z)}
\end{align*}
for some $S^{(i,T,z)}$ independent of $\la$ (the notation $S^{(i,T,z)}$ will not be used systematically, it just serves our purpose here), $T\in\{\cos,\sin\}$ and $z$ a phase, i.e a scalar function on $\mathcal{M}$.
\item In order to manipulate complicated non-linear expressions where high-frequency expansions might be differentiated and multiplied, we introduce a schematic notation.  In what follows, components of tensors and partial derivatives are defined with respect to a fixed coordinates system that will be properly introduced in Section \ref{section BG} below. Let $S$ and $S'$ be two quantities depending (in a tensorial way or not) on coordinate indexes.
\begin{itemize}
\item We denote by $\{S\}$ any linear combination of components of $S$.
\item We denote by $\{SS'\}$ any linear combination of products of components of $S$ and $S'$.
\item We denote by $\{\dr S\}$ any linear combination of partial derivatives of any components of $S$.
\end{itemize}
Moreover, if $f_\a$ is a quantity depending (in a tensorial way or not) on the coordinate index $\a$ and of the form $\{S\}$ (resp. $\{SS'\}$ or $\{\dr S\}$), we might also write $f_\a=\{S\}_\a$ (resp. $f_\a=\{SS'\}_\a$ or $f_\a=\{\dr S\}_\a$). This obviously extends to any number of indices. For instance the Christoffel symbol $\Ga(\g)^\rho_{\mu\nu}$ of a metric $\g$ might be rewritten schematically as $\cro{\g^{-1}\dr \g}$ or as $\cro{\g^{-1}\dr \g}^\rho_{\mu\nu}$.
\item We extend this schematic notation to include undescribed oscillations:  $\cro{S}^\osc$ will denote any quantity of the form
\begin{align*}
\sum_{\substack{i\geq 0 \\ z \text{ scalar function}\\T\in\{1,\cos,\sin\}}} \la^i T\pth{\frac{z}{\la}} \cro{S}.
\end{align*}
This obviously extends to the bracket notation with indexes such as $\{\cdot\}_\a$ introduced above.
\end{itemize}

\paragraph{Some important tensor operators.} We conclude this section by introducing several operators acting on symmetric 2-tensors and related to the action of the Ricci tensor on oscillating tensors.

\begin{mydef}\label{def pola}
Let $v$ be a scalar function on $\mathcal{M}$ and $S$ a symmetric 2-tensor.
\begin{itemize}
\item[(i)] We define the polarization tensor of $S$ with respect to $v$ by
\begin{align*}
\Pol{S}{v}_\a \vcentcolon = \g_0^{\mu\nu}S_{\a\mu}\dr_\nu v  - \half \tr_{\g_0}S \dr_\a v .
\end{align*}
\item[(ii)] We define the operator $\mathcal{P}_v$ by setting
\begin{align*}
\mathcal{P}_v(S)_{\a\b} \vcentcolon = - \g_0^{-1}(\d v,\d v) S_{\a\b} + \dr_{(\a}v \Pol{S}{v}_{\b)}.
\end{align*}
\end{itemize}
\end{mydef}

Note that if $\g_0^{-1}(\d u,\d u)=0$, then in any null frame $\pth{L,\Lb,e^{(1)},e^{(2)}}$ associated to $u$ (see the next section for the definition of a null frame) we have
\begin{align}
\Pol{S}{u}_{L} & = - S_{L L} , \label{Pol L}
\\ \Pol{S}{u}_{e^{(\i)}} & = - S_{L e^{(\i)}} , \label{Pol e}
\\ \Pol{S}{u}_{\Lb} & = - S_{e^{(1)} e^{(1)}} - S_{e^{(2)} e^{(2)}} . \label{Pol Lb}
\end{align}

\begin{remark}
In our previous work \cite{Touati2023a}, $\Pol{S}{v}$ was denoted by $\mathrm{Pol}(S)$ (there was no need to include the phase function in the notation since only one was considered). This latter notation can be misleading since a 2-tensor satisfying $\mathrm{Pol}(S)=0$ is physically polarized, while the notation suggests that it has 'zero' polarization. We thank Thibault Damour for this remark, according to which we changed our notation. The term "polarization" originates in the analogy with linearized gravity, where gravitational waves admit two possible polarizations corresponding here to the two degrees of freedom of $S_{XY}$ for $X,Y\in\{L,e^{(1)},e^{(2)}\}$ in the case $\Pol{S}{u}=0$ with $\g_0^{-1}(\d u,\d u)=0$ (see also Remark \ref{remark seed} below).
\end{remark}

The operator $\mathcal{P}_v$ will play a very important role in our construction in the case $\g_0^{-1}(\d v, \d v)\neq 0$ since it will describe the leading term in the Ricci tensor when acting on a tensor oscillating with the phase function $v$. The next lemma shows how one can invert this operator modulo its kernel (which we won't need). Its straightforward proof is left to the reader, it relies on the identity
\begin{align}\label{Pol dv S}
\Pol{\dr_{(\cdot}vQ_{\cdot)}}{v}=\g_0^{-1}(\d v,\d v) Q,
\end{align}
which holds for any 1-form $Q$.

\begin{lemma}\label{lem Pv}
Let $v$ be a scalar function on $\mathcal{M}$ with $\g_0^{-1}(\d v, \d v)\neq0$. We have
\begin{align*}
\mathrm{ran}\mathcal{P}_v = \enstq{A\text{ symmetric 2-tensor}}{ \Pol{A}{v}=0}.
\end{align*}
Moreover, if $ \Pol{A}{v}=0$ then
\begin{align*}
\PP_v\pth{-\frac{1}{\g_0^{-1}(\d v,\d v)}A}=A.
\end{align*}
\end{lemma}

\subsection{The background null dusts spacetime}\label{section BG}

We assume that a regular solution of the Einstein-null dusts system is given on the manifold $\mathcal{M}$. More precisely, we consider given on $\mathcal{M}$ a Lorentzian metric $\g_0$ and two families of scalar functions $\pth{ u_\A}_{\A\in\mathcal{A}}$ and $\pth{ F_\A}_{\A\in\mathcal{A}}$, where the index $\A$ runs through a finite set $\mathcal{A}$, solving the Einstein-null dusts system
\begin{equation}\label{BG system}
\left\{
\begin{aligned}
R_{\mu\nu}(\g_0) & = \sum_\A F_\A^2 \dr_\mu u_\A \dr_\nu u_\A,
\\ \g_0^{-1}(\d u_\A , \d u_\A) &  = 0,
\\ -2 L_\A F_\A + (\Box_{\g_0}u_\A)F_\A & = 0,
\end{aligned}
\right.
\end{equation}
where $L_\A \vcentcolon = - \g_0^{\alpha\beta}\dr_\alpha u_\A \dr_\beta$ is the spacetime gradient and is assumed to be future-directed. Thanks to the eikonal equation in \eqref{BG system}, $L_\A$ is null and geodesic i.e $\g_0(L_\A,L_\A)=0$ and
\begin{align}
\D_{L_\A}L_\A = 0 \label{geodesic equation},
\end{align}
where $\D$ denotes the Levi-Civita connection associated to $\g_0$. 

\paragraph{Gauge condition.} We assume that $\g_0$ satisfies the wave gauge condition
\begin{align}
\g_0^{\mu\nu}\Gamma(\g_0)^\rho_{\mu\nu} =0\label{wave condition BG},
\end{align}
in the coordinates $(t,x^1,x^2,x^3)$ of $\mathcal{M}$. This allows us to rewrite the first equation of \eqref{BG system} as
\begin{align}
-\tBox_{\g_0}(\g_0)_{\alpha\beta} + P_{\alpha\beta}(\g_0)(\dr \g_0,\dr \g_0) = 2 \sum_\A F_\A^2 \dr_\mu u_\A \dr_\nu u_\A. \label{eq BG coordinates}
\end{align}
where
\begin{align*}
P_{\alpha\beta}(\g_0)(\dr \g_0 ,\dr \g_0) & = \g_0^{\mu\rho}\g_0^{\nu\sigma} \bigg( \dr_{(\alpha}(\g_0)_{\rho\sigma}\dr_\mu (\g_0)_{\beta)\nu}  - \frac{1}{2} \dr_\alpha (\g_0)_{\rho\sigma} \dr_\beta (\g_0)_{\mu\nu}  
\\&\hspace{3cm} - \dr_\rho (\g_0)_{\alpha\nu}\dr_\sigma (\g_0)_{\beta\mu} +  \dr_\rho (\g_0)_{\sigma\alpha} \dr_\mu (\g_0)_{\nu\beta} \bigg).
\end{align*}

\paragraph{Regularity and decay.} We will measure the regularity by an integer $N\geq 10$ and the decay by a real constant $\de$ satisfying $-\frac{3}{2}<\de<-\half$. We also introduce a smallness constant $\e>0$.
\begin{itemize}
\item We assume that the metric $\g_0$ is close to Minkowski in the following sense 
\begin{align}\label{estim g0}
\l \g_0 - \mathbf{m} \r_{H^{N+1}_{\de}} + \l \dr_t \g_0 \r_{H^N_{\de+1}} + \l \dr_t^2\g_0 \r_{H^{N-1}_{\de+2}} \leq \e,
\end{align}
where $\mathbf{m}$ is the Minkowski metric on $\mathcal{M}$.
\item We assume that there exist constant non-zero vector fields $\mathfrak{z}_\A$ in $\R^3$ such that for all $\A\in\mathcal{A}$ we have
\begin{align}\label{estim uA}
\l \nab u_\A - \mathfrak{z}_\A \r_{H^N_{\de+1}} \leq \e.
\end{align}
We associate to each $u_\A$ a transport operator acting on tensors of all type
\begin{align*}
\Ll_\A \vcentcolon = -2\D_{L_\A} + \Box_{\g_0} u_\A.
\end{align*}
\item We assume that all densities $F_{\A|_{\Si_0}}$ are initially supported in a ball $B_R\subset \Si_0$ for some fixed $R>0$. The transport equation in \eqref{BG system} then implies that each $F_\A$ is supported in $J^+_0\pth{ B_R}$, i.e the causal future associated to $\g_0$. Thanks to \eqref{estim g0}-\eqref{estim uA} there exists a constant $C_{\mathrm{supp}}>0$ such that $J^+_0\pth{ B_R}\subset \enstq{(t,x)\in\mathcal{M}}{|x|\leq C_{\mathrm{supp}}R}$. We also assume that for all $\A\in\mathcal{A}$ we have
\begin{align}\label{estim FA}
\l F_\A \r_{H^N} \leq \e.
\end{align}
\end{itemize} 

\paragraph{Strong coherence.} We introduce some sets of phases. First, the null harmonics that will appear in the metric $\g_\la$:
\begin{align*}
 \mathcal{N}_k & \vcentcolon = \enstq{ k u_\A}{\A\in\mathcal{A}}.
\\ \mathcal{N} & \vcentcolon = \mathcal{N}_1  \cup \mathcal{N}_2 \cup\mathcal{N}_3.
\end{align*}
Second, the mixed harmonics that will appear in the metric $\g_\la$:
\begin{align*}
\mathcal{I}_2 & \vcentcolon = \enstq{u_\A\pm u_\B}{\A\neq \B \text{ and } \pm\in \{-1,+1\}} ,
\\ \mathcal{I}_3 & \vcentcolon = \enstq{u_\A\pm 2 u_\B}{\A\neq \B \text{ and } \pm\in \{-1,+1\}}
\\&\hspace{1cm} \cup \enstq{u_\A\pm_1 u_\B \pm_2 u_\C}{\A, \B, \C \text{ all distincts and } \pm_1,\pm_2\in \{-1,+1\}} ,
\\ \mathcal{I} & \vcentcolon = \mathcal{I}_2 \cup \mathcal{I}_3.
\end{align*}
We also define $\mathcal{W} \vcentcolon =\mathcal{N}\cup\mathcal{I}$. Finally, the mixed harmonics that will appear using the contracted Bianchi identities in Section \ref{section propa 2}:
\begin{align*}
\mathcal{I}_4 & \vcentcolon =\enstq{u_\A\pm 3 u_\B}{\A\neq \B \text{ and } \pm\in \{-1,+1\}} 
\\&\hspace{1cm} \cup \enstq{u_\A\pm_1 u_\B \pm_2 2 u_\C}{\A, \B, \C \text{ all distincts and } \pm_1,\pm_2\in \{-1,+1\}}
\\&\hspace{1cm} \cup \enstq{u_\A\pm_1 u_\B \pm_2 u_\C \pm_3 u_\D}{\A, \B, \C ,\D \text{ all distincts and } \pm_1,\pm_2,\pm_3\in \{-1,+1\}},
\\ \mathcal{I}_5 & \vcentcolon = \enstq{u_\A\pm 4u_\B}{\A\neq \B \text{ and } \pm\in \{-1,+1\}} 
\\&\hspace{1cm}\cup \enstq{3u_\A\pm 2u_\B}{\A\neq \B \text{ and } \pm\in \{-1,+1\}} 
\\&\hspace{1cm} \cup \enstq{u_\A\pm_1 2u_\B \pm_2 2u_\C}{\A, \B, \C \text{ all distincts and } \pm_1,\pm_2\in \{-1,+1\}}
\\&\hspace{1cm} \cup \enstq{u_\A\pm_1 u_\B \pm_2 3u_\C}{\A, \B, \C \text{ all distincts and } \pm_1,\pm_2\in \{-1,+1\}}
\\&\hspace{1cm} \cup \enstq{u_\A\pm_1 u_\B \pm_2 u_\C \pm_3 2u_\D}{\A, \B, \C ,\D \text{ all distincts and } \pm_1,\pm_2,\pm_3\in \{-1,+1\}},
\end{align*}
We also define $\mathcal{Z}\vcentcolon = \mathcal{W} \cup \mathcal{I}_4\cup\mathcal{I}_5  $. Our geometric optics construction will rely on the following strong coherence assumption: there exists a constant $c_{\mathrm{coherence}}>0$ such that 
\begin{align}
\min_{v\in\mathcal{I}} |\g_0^{-1}(\d v,\d v)| \geq c_{\mathrm{coherence}}. \label{coherence}
\end{align}
We also assume that there exists a constant $c_{\mathrm{spatial}}>0$ such that
\begin{align}\label{spatial}
\min_{z\in\mathcal{Z}}\pth{ \inf_{\R^3}|\nab z| } > c_{\mathrm{spatial}}.
\end{align}

\begin{remark}\label{remark adapted}
The assumptions \eqref{coherence}-\eqref{spatial} can be proved to hold under the assumption that the phases $u_\A$ are initially angularly separated, as was first noticed in \cite{Huneau2018a}. More precisely, if there exists $\eta\in(0,1)$ such that for $\A\neq \B$ we have
\begin{align*}
\frac{\nabla u_\A \cdot \nabla u_\B}{|\nabla u_\A| |\nabla u_\B|} \leq 1 - \eta
\end{align*}
on $\Si_0$ and if \eqref{spatial} holds on $\Si_0$, then Huneau and Luk show that if $\e$ is small enough then there exists positive constants $c_\A$ such that if we replace the phases $u_\A$ by $u_\A'\vcentcolon = c_\A u_\A$ then \eqref{coherence}-\eqref{spatial} hold for some constants $c_{\mathrm{coherence}}$ and $c_{\mathrm{spatial}}$ (note that $(\g_0,c_\A^{-1} F_\A,  u_\A')$ would still solve \eqref{BG system}). Note that strictly speaking \cite{Huneau2018a} consider less mixed phases in $\mathcal{Z}$ than the present article but their argument extends easily. Finally, note that ensuring angular separation and \eqref{spatial} on $\Si_0$ can be done by choosing some angularly separated $\mathfrak{z}_\A$ and then choosing $\e$ small enough in \eqref{estim uA}. 
\end{remark}

We mention a important consequence of \eqref{spatial}: together with a stationary phase argument we can show that for all $T:\R\longrightarrow\R$ smooth, $2\pi$-periodic with $\int_0^{2\pi}T=0$ and for all $z\in\mathcal{Z}$ the sequence of functions $\pth{T\pth{\frac{z}{\la}}}_{\la\in(0,1]}$ converges weakly to 0 in $L^2(K)$ when $\la\to 0$ and where $K$ is any compact subset of $\R^3$.

\paragraph{Initial data.} The background spacetime $(\mathcal{M},\g_0,u_\A,F_\A)$ induces an initial data set \[\pth{\Si_0,g_0,k_0,u_{\A|_{\Si_0}},F_{\A|_{\Si_0}}}\] on $\Si_0$ solving the null dusts constraint equations
\begin{align}
R(g_0) - |k_0|^2_{g_0} + (\tr_{g_0} k_0)^2 & = 2\sum_\A (\dr_t u_\A)^2 F_\A^2 , \label{BG hamiltonian}
\\ (\div_{g_0} k_0)_i - \dr_i \tr_{g_0} k_0 & = -\sum_\A \dr_t u_\A F_\A^2  \dr_i u_\A, \label{BG momentum}
\end{align}
and satisfying
\begin{align}\label{estim IDS}
\l g_0 - e \r_{H^{N+1}_{\de}} + \l k_0 \r_{H^N_{\de+1}}  \lesssim \e.
\end{align}
To simplify the resolution of the constraint equations on $\Si_0$ and the definition of initial data for the geometric optics hierarchy of equations, we make the generic assumption that $\dr_t$ is the future-directed unit normal to $\Si_0$ for $\g_0$. This implies that the second fundamental form $k_0$ of $\Si_0$ in $(\mathcal{M},\g_0)$ is given in coordinates by $(k_0)_{ij}=-\half \dr_t(g_0)_{ij}$ and that the spatial components of \eqref{wave condition BG} rewrite
\begin{align}
 \dr_t(\g_0)_{0\ell}   & = g_0^{ij}\pth{\dr_{i}(g_0)_{j\ell} - \half \dr_\ell (g_0)_{ij} },\label{dt g 0i}
\end{align}
on $\Si_0$. This simplifying assumption also has consequences regarding optical functions $u_\A$. Since they solve the eikonal equations and that their spacetime gradient is assumed to be future-directed we have on $\Si_0$ and for all $\A$
\begin{align}\label{eikonal init}
\dr_t u_\A = |\nab_{g_0}u_\A|_{g_0}.
\end{align}
On $\Si_0$ we define the following vector fields
\begin{equation}\label{def vector fields init}
\begin{aligned}
N_\A & \vcentcolon = \frac{\nab_{g_0}u_\A}{ |\nab_{g_0}u_\A|_{g_0}},
\\ N^{(\pm)}_{\A\B} & \vcentcolon = \frac{\nab_{g_0}(u_\A\pm u_\B)}{ |\nab_{g_0}(u_\A\pm u_\B)|_{g_0}},
\end{aligned}
\end{equation}
where $\A\neq \B$ in the second definition and where we used \eqref{estim g0} and \eqref{spatial} to divide by $|\nab_{g_0}u_\A|_{g_0}$ or $|\nab_{g_0}(u_\A\pm u_\B)|_{g_0}$. Note that \eqref{eikonal init} implies that 
\begin{align*}
L_{\A|_{\Si_0}} & = |\nab_{g_0}u_\A|_{g_0} (\dr_t - N_\A).
\end{align*}
Moreover, if we project the geodesic equation \eqref{geodesic equation} onto $\Si_0$, i.e compute $\g_0(\D_{L_\A}L_\A,\dr_\ell)$, we obtain
\begin{align}\label{geodesic spatial}
 -   N_\A  (N_\A)_\ell    +\half N_\A^k N_\A^a \dr_\ell (g_0)_{ka} + \frac{1}{|\nab_{g_0} u_\A|_{g_0}} \dr_\ell |\nab_{g_0} u_\A|_{g_0}   =  \frac{1}{|\nab_{g_0} u_\A|_{g_0}} (N_\A)_\ell N_\A |\nab_{g_0} u_\A|_{g_0}  .
\end{align}
This identity will be used in the proof of Lemma \ref{lem ka11}.

\paragraph{Null frames.} To each optical function $u_\A$ we associate a null frame $\pth{L_\A,\Lb_\A,e^{(1)}_\A,e^{(2)}_\A}$ on $\mathcal{M}$, which thus satisfies
\begin{align*}
\g_0(L_\A,L_\A) & = \g_0(\Lb_\A,\Lb_\A) = \g_0\pth{L_\A,e^{(\i)}_\A} = \g_0\pth{\Lb_\A,e^{(\i)}_\A}=0
\end{align*}
and
\begin{align*}
\g_0(L_\A,\Lb_\A) & = -2, \qquad \g_0\pth{e^{(\i)}_\A,e^{(\j)}_\A} = \de_{\i\j}.
\end{align*}
Moreover, we ask that $\pth{e^{(1)}_\A,e^{(2)}_\A}$ is an orthonormal frame for $\g_0$ of $TP_{\A,t,u}$ where 
\begin{align*}
P_{\A,t,u} = \Si_t \cap \enstq{(\tau,x)\in\mathcal{M}}{u_\A(\tau,x)=u},
\end{align*}
which, thanks to \eqref{estim uA}, has the topology of a plane in $\R^3$. More details on the definition of such null frames can be found in \cite{Szeftel2018}. Finally, note that $\pth{N_\A,e^{(1)}_\A,e^{(2)}_\A}$ is an orthonormal frame for $g_0$ of $T\Si_0$ and that $\Lb_{\A|_{\Si_0}}=|\nab_{g_0}u_\A|_{g_0} (\dr_t + N_\A)$.

\begin{remark}
In this article, we don't prove that such a background spacetime $(\mathcal{M},\g_0,u_\A,F_\A)$ can be constructed, this follows from \cite{ChoquetBruhat2006} adapted to the null dusts case.
\end{remark}

\subsection{Statement of the result}

The following theorem is the main result of this article.

\begin{thm}\label{theo evol}
Let $(\mathcal{M},\g_0, u_\A,F_\A)$ be a solution of the Einstein-null dusts system as described in Section \ref{section BG}. There exists $\e_0=\e_0(N,\de,R)>0$ such that if $0<\e\leq \e_0$ then there exists a family $(\g_\la)_{\la\in(0,\e_0]}$ of solutions to the Einstein vacuum equations on $\mathcal{M}$ of the form
\begin{align}\label{ansatz theo}
\g_\la & = \g_0 + \la \sum_\A \cos\pth{\frac{u_\A}{\la}} F^{(1)}_\A + \la^2\pth{ \sum_{\substack{w\in\mathcal{N}_1\cup\mathcal{N}_2\cup\mathcal{I}_2\\T\in\{\cos,\sin\}}} T\pth{\frac{w}{\la}}F^{(2,w,T)}_\la +  \widetilde{\g_\la}}.
\end{align}
Moreover,
\begin{itemize}
\item[(i)] the metrics $\g_\la$ solve \eqref{EVE} in generalised wave gauge and in the coordinates $(t,x^1,x^2,x^3)$ we have
\begin{equation}\label{gauge theo}
\begin{aligned}
\g_\la^{\mu\nu}\Ga(\g_\la)^\rho_{\mu\nu} & \rightarrow 0 \quad\text{uniformly on compact sets},
\\ \dr_\a\pth{ \g_\la^{\mu\nu}\Ga(\g_\la)^\rho_{\mu\nu}} & \rightharpoonup 0 \quad\text{weakly in $L^2_{loc}$},
\end{aligned}
\end{equation}
when $\la\to 0$,
\item[(ii)] the tensors $F^{(1)}_\A$ and $F^{(2,w,T)}_\la$ are supported in $J^+_0\pth{ B_R}$ and $F^{(1)}_\A$ satisfies 
\begin{align}
-2\D_{L_\A} F^{(1)}_\A + (\Box_{\g_0}u_\A)F^{(1)}_\A & = 0, \label{transport theo}
\\ \left| F^{(1)}_\A \right|^2_{\g_0} & = 8 F_\A^2,\label{energy theo}
\end{align}
\item[(iii)] there exists $C=C(N,\de,R)>0$ such that 
\begin{align}
\l F^{(1)}_\A \r_{H^N} & \leq C\e,\label{estim theo 1}
\\ \l F^{(2,w,T)}_\la \r_{L^2} + \max_{r\in\llbracket 0,5\rrbracket}\la^r\l\dr\nab^r F^{(2,w,T)}_\la \r_{L^2} & \leq C\e,\label{estim theo 2}
\\ \l \widetilde{\g_\la} \r_{L^2_{\de}} + \max_{r\in\llbracket 0, 4 \rrbracket} \la^{r} \l\dr \nab^r \widetilde{\g_\la} \r_{L^2_{\de+1+r}} & \leq C \e. \label{estim theo 3}
\end{align}
\end{itemize}
\end{thm}

\noindent Some comments are in order.
\begin{itemize}
\item \textbf{The reverse Burnett conjecture.} The estimates stated in Theorem \ref{theo evol} are consistent with Burnett's weak convergence as in \eqref{CV}. Indeed, the assumption $\de>-\frac{3}{2}$, the Sobolev embedding $H^2_{-\frac{3}{2}}\subset L^\infty$ and \eqref{estim theo 2}-\eqref{estim theo 3} imply 
\begin{align*}
\l F^{(2,w,T)}_\la \r_{L^\infty} + \l \widetilde{\g_\la} \r_{L^\infty}\lesssim \la^{-1},
\end{align*}
and thus that $\l \g_\la - \g_0 \r_{L^\infty}\lesssim \la$. Moreover, we have
\begin{align*}
\dr(\g_\la-\g_0) & = -\sum_\A \sin\pth{\frac{u_\A}{\la}}\dr u_\A F^{(1)}_\A 
\\&\quad +  \la \pth{ \sum_\A \cos\pth{\frac{u_\A}{\la}} \dr F^{(1)}_\A + \sum_{\substack{w\in\mathcal{N}_1\cup\mathcal{N}_2\cup\mathcal{I}_2\\T\in\{\cos,\sin\}}} T'\pth{\frac{w}{\la}}\dr w F^{(2,w,T)}_\la }
\\&\quad +  \la^2 \pth{ \sum_{\substack{w\in\mathcal{N}_1\cup\mathcal{N}_2\cup\mathcal{I}_2\\T\in\{\cos,\sin\}}} T\pth{\frac{w}{\la}}\dr F^{(2,w,T)}_\la +  \dr \widetilde{\g_\la}}.
\end{align*}
Thanks to \eqref{spatial} and \eqref{estim theo 1}, $\sin\pth{\frac{u_\A}{\la}}\dr u_\A F^{(1)}_\A$ converges weakly to 0 in $L^2_{loc}$ and \eqref{estim theo 2}-\eqref{estim theo 3} again implies that the remaining terms in $\dr(\g_\la-\g_0)$ converge actually strongly to 0 in $L^2_{loc}$. Therefore, Theorem \ref{theo evol} indeed implies Theorem \ref{theo rough burnett}.
\item \textbf{Parametrization of the $(\g_\la)_{\la\in(0,1]}$.} For a given background spacetime $(\mathcal{M},\g_0, u_\A,F_\A)$ as described in Section \ref{section BG}, we actually construct several families $(\g_\la)_{\la\in(0,\e_0]}$ as in Theorem \ref{theo evol}. These families are parametrized by a set of seeds living on the initial hypersurface $\Si_0$ and defined in Definition \ref{def seed} below. Each seed is parametrized by two real numbers, which correspond to the coefficients of two possible polarizations for the leading oscillating term in \eqref{ansatz theo}, as in the linearized gravity setting in TT gauge (see Remark \ref{remark seed}).
\item \textbf{The generalised wave gauge.} The exact expression of the generalised wave gauge term $\g_\la^{\mu\nu}\Ga(\g_\la)^\rho_{\mu\nu}$ is somehow irrelevant since we don't really prescribe what it should contain but rather what it should not contain, namely first order derivatives of some of the metric components solving (or coupled with) a wave equation. A novelty compared to \cite{Touati2023a} is that we also use this gauge term to control the interaction of waves propagating in different directions, see the discussion on $\gg^{(3,e)}_{v,T}$ in Section \ref{section gauge term}. Overall, the gauge term helps us recover true hyperbolicity and ellipticity of the Ricci tensor.
\item \textbf{The constraint equations.} As in every resolution of \eqref{EVE} from a spacelike hypersurface, one should solve first the constraint equations on the initial hypersurface. In our previous work on the singlephase case, we separated the resolution of the constraint equations on $\Si_0$ from the resolution of \eqref{EVE} on $\mathcal{M}$, resulting in the two articles \cite{Touati2023a} and \cite{Touati2023b}. For the multiphase case, we decided to associate the elliptic and hyperbolic aspects of the construction in one self-consistent article. In particular, we want to highlight the surprising connections between the elliptic and hyperbolic procedures due to the oscillatory aspect of the metric. Some of these connections were already present in the singlephase case, but some are specific to the multiphase case, see the discussion after Proposition \ref{prop constraint main}.
\item \textbf{Gauge-independent transparency.} We conclude these comments by insisting on the astounding structures in the Einstein vacuum equations leading in particular to \eqref{transport theo}-\eqref{energy theo}. These relations satisfied by the first profiles in the ansatz \eqref{ansatz theo} for $\g_\la$ illustrate perfectly the discussion of Section \ref{section transparence}: the self-interaction of gravitational waves produces a quadratic macroscopic effect \eqref{energy theo} while each wave is linearly transported \eqref{transport theo}. Moreover the interaction of gravitational waves propagating in different null directions don't produce a macroscopic effect, in particular since the first mixed harmonics $v\in\mathcal{I}$ only appear at order $\la^2$ in \eqref{ansatz theo}. Finally, note that both the linear propagation and non-linear backreaction are phenomena that are independent of the choice of coordinates since \eqref{transport theo} and \eqref{BG system} are tensorial equations.
\end{itemize}
The rest of this article is devoted to the proof of Theorem \ref{theo evol}.


\section{The high-frequency ansatz}\label{section expansion Ricci}

In this section we present the high-frequency ansatz for the metric $\g_\la$. Its expression is given in Section \ref{section expansion metric}, its Ricci tensor is computed in Section \ref{section ricci} and finally in Section \ref{section hierarchy} a hierarchy of equations and conditions is deduced from the requirement that $\g_\la$ solves the Einstein vacuum equations \eqref{EVE}.

\subsection{Expansion of the metric}\label{section expansion metric}

For $\la>0$, we consider the following metric 
\begin{equation}\label{expansion metric}
\begin{aligned}
\g_\la & = \g_0 + \la \sum_\A  \cos\pth{\frac{u_\A}{\la}} F^{(1)}_\A  +   \lambda^2 \sum_\A \pth{ \sin\pth{\frac{u_\A}{\la}} \pth{ \F_\A + F^{(2,1)}_\A } + \cos\pth{\frac{2u_\A}{\la}} F^{(2,2)}_\A }
\\&\quad + \la^2 \sum_{\A\neq \B,\pm}   \cos\pth{\frac{u_\A\pm u_\B}{\la}} F^{(2,\pm)}_{\A\B} + \la^2 \h_\la 
\\&\quad + \la^3 \sum_{\substack{u\in\mathcal{N}\\T\in\{\cos,\sin\}}}T\pth{\frac{u}{\la}} \gg^{(3,h)}_{u,T} + \la^3 \sum_{\substack{v\in\mathcal{I}\\T\in\{\cos,\sin\}}} T\pth{\frac{v}{\la}} \gg^{(3,e)}_{v,T}.
\end{aligned}
\end{equation}
In this expression, 
\begin{itemize}
\item $F^{(1)}_\A$, $F^{(2,1)}_\A$, $F^{(2,2)}_\A$, $F^{(2,\pm)}_{\A\B}$, are symmetric 2-tensors which don't depend on $\la$ and we already assume $F^{(2,\pm)}_{\A\B}  = F^{(2,\pm)}_{\B\A}$,
\item the remainder $\h_\la$ and the amplitude $\F_\A$ are symmetric 2-tensors and do depend on $\la$. The higher order symmetric 2-tensors $\gg^{(3,h)}_{u,T}$ and $\gg^{(3,e)}_{v,T}$ will depend on $\F_\A$ and thus on $\la$.
\end{itemize}

The ansatz \eqref{expansion metric} differs from the singlephase ansatz in \cite{Touati2023a} only because of the terms $F^{(2,\pm)}_{\A\B}$ and $\gg^{(3,e)}_{v,T}$. They will crucially be used to control the creation of mixed harmonics, i.e terms oscillating with phases $v\in\mathcal{I}$, at order $\la^0$ and $\la^1$ in the Ricci tensor of $\g_\la$ respectively.

\subsection{Expansion of the Ricci tensor}\label{section ricci}

In this section, we compute the Ricci tensor of the metric $\g_\lambda$. Since we want $\g_\la$ to solve \eqref{EVE} in generalised wave gauge we use the following decomposition of the Ricci tensor of Lorentzian metric $\g$:
\begin{align}\label{Ricci tensor GWC}
2R_{\a\b}(\g) = - \tBox_\g \g_{\a\b} + P_{\a\b}(\g)(\dr\g,\dr\g) + \g_{\rho(\a}\dr_{\b)}H^\rho + H^\rho\dr_\rho\g_{\a\b},
\end{align}
where
\begin{align}
\tBox_\g & \vcentcolon = \g^{\mu\nu}\dr_\mu \dr_\nu, \label{wave heur}
\\ H^\rho & \vcentcolon = \g^{\mu\nu}\Ga(\g)^\rho_{\mu\nu},\label{gauge heur}
\\ P_{\a\b}(\g)(\dr\g,\dr\g) & \vcentcolon = \g^{\mu\rho}\g^{\nu\si} \pth{ \dr_{(\a}\g_{\rho\si}\dr_\mu \g_{\b)\nu}  - \frac{1}{2} \dr_\a \g_{\rho\si} \dr_\b \g_{\mu\nu}  - \dr_\rho \g_{\a\nu}\dr_\si \g_{\b\mu} +  \dr_\rho \g_{\si\a} \dr_\mu\g_{\nu\b} }. \label{P dg dg}
\end{align}
We expand the wave part $\tBox_{\g_\lambda}(\g_\lambda)_{\alpha\beta}$ in Section \ref{section wave part}, the quadratic non-linearity $P_{\alpha\beta}(\g_\la)(\dr \g_\lambda ,\dr \g_\lambda) $ in Section \ref{section quadratic} and the gauge term $H^\rho$ in Section \ref{section gauge term}. We combine the results in Section \ref{section ricci final} to obtain the expansion of $R_{\a\b}(\g_\lambda)$. The computations will be performed under the following assumption on $F^{(1)}_\A$:
\begin{equation}\label{assumption}
\begin{aligned}
\Pol{F^{(1)}_\A}{u_\A} & =0,     
\\ F^{(1)}_{L_\A \Lb_\A} & = 0.
\end{aligned}
\end{equation}
Note that thanks to \eqref{Pol L}-\eqref{Pol Lb} the assumption \eqref{assumption} is equivalent to $\tr_{\g_0}F^{(1)}_\A =0$ and $(F^{(1)}_\A)_{L_\A \a } =0$. They are consistent with the singlephase case of \cite{Touati2023a} and the approximate construction of \cite{ChoquetBruhat1969}. More precisely, the first assumption in \eqref{assumption} ensures that $R_{\a\b}(\g_\lambda)=\GO{1}$. The second assumption is convenient to simplify the computations and is not necessary. They will be proved to be consistent with the hierarchy derived from \eqref{EVE}, as in \cite{Touati2023a}.

\begin{remark}
We introduce two conventions on the use of the bracket schematic notations introduced in Section \ref{section notation}:
\begin{itemize}
\item In a bracket we will never write down the quantities depending on the background spacetime introduced in Section \ref{section BG}, since we are not solving for it. This includes metric or inverse metric coefficients, phases $z\in\mathcal{Z}$, vector fields from the null frames.
\item In a bracket we will never specify the indices referring to the set $\mathcal{A}$. For instance, $\cro{F^{(1)}}$ simply denotes any of the $\cro{F^{(1)}_\A}$ for $\A\in\mathcal{A}$. Similarly $\cro{\gg^{(3,h)}}$ denotes any of the $\cro{\gg^{(3,h)}_{u,T}}$ for $u\in\mathcal{N}$ and $T\in\{\cos,\sin\}$ (same for $\gg^{(3,e)}$). 
\end{itemize}
These conventions also apply to the notation $\cro{\cdot}^\osc$ and will be used throughout the article.
\end{remark}

\begin{remark}\label{remark admissible forbidden}
To clarify the mechanisms at stake, we introduce a notion of \emph{admissible} and \emph{forbidden} terms, as in \cite{Touati2023a}. The admissibility of a term depends on its order in terms of $\la$ in the Ricci tensor of $\g_\la$. More precisely, for $k\in\{0,1\}$, a term at order $\la^k$ in the Ricci tensor is called admissible if it oscillates with frequency $\frac{\ell u_\A}{\la}$ for some $\A\in\mathcal{A}$ and $\ell\in\llbracket 1,k+1\rrbracket$. A non-admissible term is said to be forbidden. Admissible terms are the easiest to treat since they will be dealt with by transport equations and won't require the help of gauge freedom, as opposed to forbidden terms. 
\end{remark}

\subsubsection{The wave part}\label{section wave part}

In this section we expand the wave term in \eqref{Ricci tensor GWC}, i.e $\tBox_{\g_\lambda}(\g_\lambda)_{\alpha\beta}$. The result is contained in the next proposition, whose proof is postponed to Appendix \ref{appendix proof lemma W}.

\begin{lemma}\label{lemma W}
Under the assumptions \eqref{assumption}, the wave part of the Ricci tensor of $\g_\la$ admits the expansion
\begin{equation}\label{exp W}
\begin{aligned}
\tBox_{\g_\la}(\g_\la)_{\a\b} & =  \tBox_{\g_0}(\g_0)_{\a\b} + \sum_\A \sin\pth{\frac{u_\A}{\la}} (W^{(0,1)}_\A)_{\a\b}  +  \sum_{\A\neq\B,\pm} \cos\pth{\frac{u_\A\pm u_\B}{\la}}  (W^{(0,\pm)}_{\A\B})_{\a\b}
\\&\quad + \la \sum_\A\pth{\cos\pth{\frac{u_\A}{\la}} (W^{(1,1)}_\A)_{\a\b} + \sin\pth{\frac{2u_\A}{\la}} (W^{(1,2)}_\A)_{\a\b}}
\\&\quad + \la\sum_\A\cos\pth{\frac{u_\A}{\lambda}} \cos\pth{ \frac{2u_\A}{\lambda} } (F^{(2,2)}_\A)_{L_\A L_\A} (F^{(1)}_\A)_{\alpha\beta}
\\&\quad +  \frac{\la}{2}\sum_\A \sin\pth{\frac{2u_\A}{\la}} (\F_\A)_{L_\A L_\A} (F^{(1)}_\A)_{\a\b} + \la \sum_{\substack{v\in\mathcal{I}\\T\in\{\cos,\sin\}}} T\pth{\frac{v}{\la}}(W^{(1)}_{v,T})_{\a\b} 
\\&\quad + \la^2 W^{(\geq 2)}_{\a\b}.
\end{aligned}
\end{equation}
The terms of order 0 are given by
\begin{align}
(W^{(0,1)}_\A)_{\a\b} & = 2 L_\A (F^{(1)}_\A)_{\a\b} - (\tBox_{\g_0} u_\A)(F^{(1)}_\A)_{\a\b} ,\label{W 0 1}
\\ (W^{(0,\pm)}_{\A\B})_{\a\b} & = - \g_0^{-1}(\d (u_\A\pm u_\B), \d(u_\A\pm u_\B))  (F^{(2,\pm)}_{\A\B})_{\alpha\beta} \label{W 0 pm}
\\&\quad + \frac{1}{4}   \pth{ (F^{(1)}_\B)_{L_\A L_\A}   (F^{(1)}_\A)_{\alpha\beta} + (F^{(1)}_\A)_{L_\B L_\B}   (F^{(1)}_\B)_{\alpha\beta} }.\non
\end{align}
The terms of order 1 are given by
\begin{align}
(W^{(1,1)}_\A)_{\a\b} & =   (-2 L_\A + \tBox_{\g_0}u_\A ) \left( (\F_\A)_{\alpha\beta} + (F^{(2,1)}_\A)_{\alpha\beta} \right)   + (\h_\lambda)_{L_\A L_\A} (F^{(1)}_\A)_{\alpha\beta}     \label{W 1 1}
\\&\quad + \cro{\pth{1 + F^{(2,\pm)} + (F^{(1)})^2 }F^{(1)} + \dr^2 F^{(1)}}_{\a\b} , \non
\\ (W^{(1,2)}_\A)_{\a\b} & = -2 (-2 L_\A + \tBox_{\g_0}u_\A ) (F^{(2,2)}_\A)_{\alpha\beta}  +  \cro{ \pth{ F^{(2,1)} + \dr F^{(1)} + F^{(1)} }F^{(1)} }_{\a\b} , \label{W 1 2}
\\ (W^{(1)}_{v,T})_{\a\b} & = - \g_0^{-1}(\d v, \d v) (\gg^{(3,e)}_{v,T})_{\a\b}  \label{W 1 v T}
\\&\quad +  \Big\{\dr^{\leq 1}F^{(2,\pm)} + \pth{ \F + F^{(2,1)} + F^{(2,2)} + F^{(2,\pm)}  + (F^{(1)})^2 + F^{(1)} + \dr F^{(1)} }F^{(1)} \Big\}_{\a\b} .\non
\end{align}
The higher order terms are given by
\begin{align}
W^{(\geq 2)}_{\a\b} & = \tBox_{\g_\la}(\h_\la)_{\a\b}  + \sum_\A \sin\pth{\frac{u_\A}{\la}}\tBox_{\g_\la}(\F_\A)_{\a\b} \label{W 2}
\\&\quad +  \la \sum_{\substack{u\in\mathcal{N}\\T\in\{\cos,\sin\}}}T\pth{\frac{u}{\la}}\tBox_{\g_\la} (\gg^{(3,h)}_{u,T})_{\a\b}  + \la  \sum_{\substack{v\in\mathcal{I}\\T\in\{\cos,\sin\}}} T\left( \frac{v}{\lambda} \right)\tBox_{\g_\la} (\gg^{(3,e)}_{v,T})_{\a\b}\non
\\&\quad + \Big\{ \g_\la^{-1}\pth{\dr^{\leq 1}\pth{\gg^{(3,h)} + \gg^{(3,e)} + \F} + \dr^{\leq 2}\pth{ F^{(1)}  + F^{(2,1)} + F^{(2,2)} +  F^{(2,\pm)}} } \Big\}_{\a\b}^\osc \non.
\end{align}
\end{lemma}

We draw the reader's attention on several structural facts from Lemma \ref{lemma W}:
\begin{enumerate}
\item The main term produced by the wave part of the Ricci is a transport term along the rays: at order $\la^0$ we see the transport of $F^{(1)}_\A$ by the vector field $L_\A$ and at order $\la^1$ we see the transport of $\F_\A$, $F^{(2,1)}_\A$ and $F^{(2,2)}_\A$ by $L_\A$ again. By imposing transport equations for these tensors we will be able to absorb terms from the Ricci tensor oscillating at the same frequency, which precisely correspond to the admissible terms defined in Remark \ref{remark admissible forbidden}. Referring again to this remark, we already see in \eqref{exp W} the presence at order $\la^1$ of forbidden term such as $(F^{(2,2)}_\A)_{L_\A L_\A} (F^{(1)}_\A)_{\alpha\beta}$ oscillating like $\cos\pth{\frac{3u_\A}{\la}}$, we will deal with it differently.
\item At order $\la^2$, we obviously recover the wave operator acting on the remainder $\h_\la$. Note how we made a difference in \eqref{W 2} between second order derivatives of the various tensors in \eqref{expansion metric}. Second order derivatives of $F^{(1)}_\A$, $F^{(2,1)}_\A$, $F^{(2,2)}_\A$ and $F^{(2,\pm)}_{\A\B}$ are considered error terms and put in the brackets, while second order derivatives of $\F_\A$, $\gg^{(3,h)}_{u,T}$ and $\gg^{(3,e)}_{v,T}$ are carefully conserved and in particular the wave operator structure will crucially be used to prove well-posedness.
\item Opposite to the terms oscillating in null directions, the wave operator acts 'elliptically' on terms oscillating with phase $v\in\mathcal{I}$ since $\g_0^{-1}(\d v,\d v)\neq 0$ whenever $v\in\mathcal{I}$ (see \eqref{coherence}). This explains the presence of $F^{(2,\pm)}_{\A\B}$ in \eqref{W 0 pm} and of $\gg^{(3,e)}_{v,T}$ in \eqref{W 1 v T}.
\end{enumerate}

\subsubsection{The quadratic non-linearity}\label{section quadratic}

In this section we expand the quadratic non-linearity in \eqref{Ricci tensor GWC}, i.e $P_{\alpha\beta}(\g_\la)(\dr \g_\lambda ,\dr \g_\lambda) $. The result is contained in the next proposition, whose proof is postponed to Appendix \ref{appendix proof lemma P}.

\begin{lemma}\label{lemma P}
Under the assumptions \eqref{assumption}, the quadratic part of the Ricci tensor of $\g_\la$ admits the expansion
\begin{equation}\label{exp P}
\begin{aligned}
P_{\alpha\beta}(\g_\la)(\dr \g_\lambda, \dr \g_\lambda) & = P_{\alpha\beta}(\g_0)(\dr \g_0, \dr \g_0)  - \frac{1}{4}  \sum_{\A}  \left| F^{(1)}_\A \right|^2_{\g_0} \dr_\alpha u_\A \dr_\beta u_\A   + \sum_\A \sin\pth{\frac{u_\A}{\la}} (P^{(0,1)}_\A)_{\a\b} 
\\&\quad + \sum_\A \cos\pth{\frac{2u_\A}{\la}} (P^{(0,2)}_\A)_{\a\b} + \sum_{\A\neq \B,\pm}\cos\pth{\frac{u_\A\pm u_\B}{\la}} (P^{(0,\pm)}_{\A\B})_{\a\b}
\\&\quad + \la \sum_\A\pth{ \cos\pth{\frac{u_\A}{\la}} (P^{(1,1)}_\A)_{\a\b} + \sin\pth{ \frac{2u_\A}{\la}} (P^{(1,2)}_\A)_{\a\b} }
\\&\quad + \la  \sum_{\substack{u\in\mathcal{N}\\T\in\{\cos,\sin\}}}T\pth{\frac{u}{\la}} \dr_{(\a}u (\hat{P}^{(1)}_{u,T})_{\b)} + \la \sum_{\substack{v\in\mathcal{I}\\T\in\{\cos,\sin\}}}T\pth{ \frac{v}{\la}} (P^{(1)}_{v,T})_{\a\b} 
\\&\quad + \la^2  P^{(\geq 2)}_{\alpha\beta} .
\end{aligned}
\end{equation}
The terms of order 0 are given by
\begin{align}
(P^{(0,1)}_\A)_{\a\b} & = - 2\Ga(\g_0)^\mu_{(\a \rho} \dr^\rho u_\A (F^{(1)}_\A)_{\beta)\mu}  - (F^{(1)}_\A)^{\mu\nu} \dr_{(\a}u_\A \pth{ \dr_\mu (\g_0)_{\b)\nu}  - \frac{1}{2} \dr_{\b)} (\g_0)_{\mu\nu} },\label{P 0 1}
\\ (P^{(0,2)}_\A)_{\a\b} & = \frac{1}{4}\left| F^{(1)}_\A \right|^2_{\g_0}\dr_\alpha u_\A \dr_\beta u_\A  , \label{P 0 2}
\\ (P^{(0,\pm)}_{\A\B})_{\a\b} & = \frac{1}{4}\bigg( \pm\dr_{(\alpha}u_\A (F^{(1)}_\A)_{L_\B\si} (F^{(1)}_\B)_{\beta)}^\si  \pm\dr_{(\alpha}u_\B (F^{(1)}_\B)_{L_\A\si} (F^{(1)}_\A)_{\beta)}^\si  \label{P 0 pm}
\\&\hspace{2cm} \pm\frac{1}{2} \dr_{(\alpha} u_\A\dr_{\beta)} u_\B   \left| F^{(1)}_\A\cdot F^{(1)}_\B \right|_{\g_0}  \pm (F^{(1)}_\A)_{(\alpha L_\B} (F^{(1)}_\B)_{\beta) L_\A}\non
\\&\hspace{6cm}  \mp \g_0^{-1}(\d u_\A,\d u_\B)  (F^{(1)}_\A)_{(\alpha}^\nu (F^{(1)}_\B)_{\nu\beta)}   \bigg).\non
\end{align}
The terms of order 1 are given by
\begin{align}
(P^{(1,1)}_\A)_{\a\b} & = 2\Ga(\g_0)^{\mu}_{(\a\rho}  \dr^\rho u_\A\pth{(\F_\A)_{\beta)\mu} +  (F^{(2,1)}_\A)_{\beta)\mu}}  \label{P 1 1}
\\&\quad +  \cro{ \dr F^{(1)} + F^{(1)} F^{(2,\pm)} + F^{(1)}+(F^{(1)})^3 }_{\a\b} , \non
\\ (P^{(1,2)}_\A)_{\a\b} & = -4\Ga(\g_0)^{\mu}_{(\a\rho}  \dr^\rho u_\A (F^{(2,2)}_\A)_{\beta)\mu} + \cro{F^{(1)} \dr F^{(1)}+ (F^{(1)})^2}_{\a\b}, \label{P 1 2}
\\ (\hat{P}^{(1)}_{u,T})_{\b} & = \cro{\pth{1+F^{(1)}}\pth{(F^{(1)})^2 +\F + F^{(2,1)} +  F^{(2,2)}}}_{\b}, \label{Phat 1 u T}
\\ (P^{(1)}_{v,T})_{\a\b} & = \Big\{ F^{(2,\pm)} + (F^{(1)})^2+(F^{(1)})^3  + F^{(1)} \pth{\dr F^{(1)} + \F + F^{(2,1)} + F^{(2,2)} + F^{(2,\pm)} }  \Big\}_{\a\b}. \label{P 1 v T}
\end{align}
The terms of order 2 are simply of the form
\begin{align}
P^{(\geq 2)}_{\alpha\beta} & =   \cro{ \pth{ \g_\la^{-1}\g_\la^{-1}\dr\g_\la\dr\g_\la }^{(\geq 2)} }_{\a\b}^\osc .\label{P 2}
\end{align}
\end{lemma}

We draw the reader's attention on several structural facts from Lemma \ref{lemma P}:
\begin{enumerate}
\item It is clear from the resonant term in \eqref{exp P} that the quadratic part of the Ricci tensor is responsible for backreaction, i.e the fact that the Ricci tensor of the background spacetime $\g_0$ cannot vanish and needs to absorb the sum of $ \left| F^{(1)}_\A \right|^2_{\g_0}  \dr_\alpha u_\A \dr_\beta u_\A $, which in particular makes clear the null dusts structure of backreaction.
\item Following the first comment after Lemma \ref{lemma W}, we see the presence at order $\la^0$ of a forbidden term due to $P^{(0,2)}_\A$. This term is the high-high counterpart of backreaction and its particular null dust tensorial structure $\dr_\a u_\A \dr_\b u_\A$ given in \eqref{P 0 2} will be crucially used. Similarly, the forbidden harmonics $3u_\A$ is present at order $\la^1$ with a particular tensorial structure $\dr_{(\a}u (\hat{P}^{(1)}_{u,T})_{\b)}$ (recall that $\mathcal{N}$ includes the third harmonics $3u_\A$).
\end{enumerate}

\subsubsection{The gauge part}\label{section gauge term}

In this section we expand the gauge term $H^\rho$ which appears differentiated in \eqref{Ricci tensor GWC}. In the usual wave gauge, it is simply set to 0. In our case, we will set to 0 only some problematic terms in $H^\rho$ which we regroup and call $\Upsilon^\rho$. The terms that we will put in $\Upsilon^\rho$ will serve two main purposes: first recover hyperbolicity and obtain a well-posed final system and second simplify the control of the mixed harmonics at order $\la^1$. As a secondary purpose it will allow $\gg^{(3,h)}_{u,T}$ not to depend on the remainder $\h_\la$. For now, we simply give the definition of $\Upsilon^\rho$:
\begin{equation}\label{Upsilon}
\begin{aligned}
\Upsilon^\rho & \vcentcolon =  \g_\la^{\rho\sigma} \g_\la^{\mu\nu} \Bigg( \dr_\mu (\h_\lambda)_{\sigma\nu} - \half \dr_\sigma (\h_\lambda)_{\mu\nu} +  \sum_\A  \sin\left(\frac{u_\A}{\lambda} \right) \pth{\dr_\mu (\F_\A)_{\sigma\nu} - \half \dr_\sigma (\F_\A)_{\mu\nu} }  \Bigg)
\\&\quad + \la \g_\la^{\rho\sigma} \g_\la^{\mu\nu} \sum_{\substack{u\in\mathcal{N}\\T\in\{\cos,\sin\}}}T\pth{\frac{u}{\la}} \pth{\dr_\mu (\gg^{(3,h)}_{u,T})_{\sigma\nu} - \half \dr_\sigma (\gg^{(3,h)}_{u,T})_{\mu\nu}}
\\&\quad + \la \g_\la^{\rho\sigma} \g_\la^{\mu\nu} \sum_{\substack{v\in\mathcal{I}\\T\in\{\cos,\sin\}}}T\pth{\frac{v}{\la}} \pth{\dr_\mu (\gg^{(3,e)}_{v,T})_{\sigma\nu} - \half \dr_\sigma (\gg^{(3,e)}_{v,T})_{\mu\nu}}
\\&\quad +  \g_\la^{\rho\sigma} \g_\la^{\mu\nu}  \sum_{\substack{v\in\mathcal{I}\\T\in\{\cos,\sin\}}}T'\pth{\frac{v}{\la}}\pth{ \dr_\mu v(\gg^{(3,e)}_{v,T})_{\sigma\nu} - \half \dr_\sigma v(\gg^{(3,e)}_{v,T})_{\mu\nu} }
\\&\quad  -  \g_0^{\rho\sigma} \h_\la^{\mu\nu} \Bigg(\dr_\mu (\g_0)_{\sigma\nu} - \half \dr_\sigma (\g_0)_{\mu\nu} -  \sum_\A  \sin\left( \frac{u_\A}{\lambda} \right) \pth{ \dr_\mu u_\A (F^{(1)}_\A)_{\sigma\nu} - \half \dr_\sigma u_\A (F^{(1)}_\A)_{\mu\nu}  }\Bigg).
\end{aligned}
\end{equation}
The gauge term is expanded in the next proposition, whose proof is postponed to Appendix \ref{appendix proof lemma H}.

\begin{lemma}\label{lemma H}
Under the assumptions \eqref{assumption}, the gauge term $H^\rho$ satisfies
\begin{align*}
H^\rho & = \mathring{H}^\rho + \la^2 \Upsilon^\rho
\end{align*}
where $\mathring{H}^\rho$ admits the expansion
\begin{align*}
\mathring{H}^\rho & = \la  \sum_\A\pth{\cos\pth{\frac{u_\A}{\la}} (H^{(1,1)}_\A)^\rho + \sin\pth{\frac{2u_\A}{\la}} (H^{(1,2)}_\A)^\rho}  +\la \sum_{\A\neq\B,\pm}\sin\pth{\frac{u_\A\pm u_\B}{\la}}(H^{(1,\pm)}_{\A\B})^\rho 
\\&\quad + \la^2 (H^{(2)})^\rho + \la^3 (H^{(\geq 3)})^\rho.
\end{align*}
The terms of order 1 are given by
\begin{align}
(H^{(1,1)}_\A)^\rho & = \g_0^{\rho\sigma} \Pol{\F_\A + F^{(2,1)}_\A}{u_\A}_\si + \g_0^{\rho\si} \g_0^{\mu\nu} \pth{ \dr_\mu (F^{(1)}_\A)_{\sigma\nu} - \half \dr_\sigma (F^{(1)}_\A)_{\mu\nu} } \label{H 1 1}
\\&\quad - (F^{(1)}_\A)^{\mu\nu} \g_0^{\rho\sigma} \pth{\dr_\mu (\g_0)_{\sigma\nu} - \half \dr_\sigma (\g_0)_{\mu\nu} },\non
\\ (H^{(1,2)}_\A)^\rho  & = -2 \g_0^{\rho\sigma} \Pol{F^{(2,2)}_\A}{u_\A}_\si - \frac{1}{4}  \dr^\rho u_\A  \left| F^{(1)}_\A\right|^2_{\g_0} , \label{H 1 2}
\\ (H^{(1,\pm)}_{\A\B})^\rho  & = -\g_0^{\rho\sigma} \Pol{F^{(2,\pm)}_{\A\B}}{u_\A\pm u_\B}_\si - \frac{1}{8} \dr^\rho (u_\A\pm u_\B) \left| F^{(1)}_\A\cdot F^{(1)}_\B\right|_{\g_0}  \label{H 1 pm}
\\&\quad  - \frac{1}{4}  (F^{(1)}_\B)^{\nu}_{L_\A} (F^{(1)}_\A)_{\nu}^\rho \mp \frac{1}{4}  (F^{(1)}_\A)^{\nu}_{L_\B} (F^{(1)}_\B)_{\nu}^\rho  .\non
\end{align}
The terms of order 2 are given by
\begin{equation}\label{H 2 gauge}
\begin{aligned}
&(H^{(2)})^\rho 
\\& =  \sum_{\substack{u\in\mathcal{N}\\T\in\{\cos,\sin\}}}T'\pth{\frac{u}{\la}}\g_0^{\rho\sigma}  \Pol{\gg^{(3,h)}_{u,T}}{u}_\si 
\\&\quad +  \sum_{\substack{u\in\mathcal{N}\\T\in\{1,\cos,\sin\}}}T\pth{\frac{u}{\la}} \Big\{\dr F^{(2,1)} + \dr F^{(2,2)}
\\&\hspace{4cm} +  \pth{ \F + F^{(2,1)} + F^{(2,2)} + F^{(2,\pm)} + (F^{(1)})^2 + \dr F^{(1)} } \pth{1+F^{(1)}} \Big\}^\rho
\\&\quad +  \sum_{\substack{v\in\mathcal{I}\\T\in\{1,\cos,\sin\}}}T\pth{ \frac{v}{\la}} \Big\{ \dr F^{(2,\pm)}
\\&\hspace{4cm} +  \pth{ \F + F^{(2,1)} + F^{(2,2)} + F^{(2,\pm)} + (F^{(1)})^2 + \dr F^{(1)} } \pth{1+F^{(1)}} \Big\}^\rho.
\end{aligned}
\end{equation}
The terms of order 3 are of the form
\begin{align}
(H^{(\geq 3)})^\rho & = \Big\{ \pth{ (\g_\la^{-1} (\g_\la)^{(\geq 2)} + (\g_\la^{-1})^{(\geq 1)}(\g_\la^{-1})^{(\geq 1)}}
\\&\hspace{1cm}\times \pth{ \dr F^{(1)}  + \F + \dr^{\leq 1} F^{(2,1)} +  \dr^{\leq 1}F^{(2,2)} +  \dr^{\leq 1}F^{(2,\pm)} + \gg^{(3,h)} } \Big\}^{\rho,\osc} \non.
\end{align}
\end{lemma}

We draw the reader's attention on several structural facts from Lemma \ref{lemma H}:
\begin{enumerate}
\item The gauge term $H^\rho$ contains first order derivatives of the metric. Some appear in $\mathring{H}^\rho$ and some appear in $\Upsilon^\rho$, i.e are considered as problematic terms from the well-posedness point of view. In particular, first order derivatives of $\h_\la$, $\F_\A$, $\gg^{(3,h)}_{u,T}$ and $\gg^{(3,e)}_{v,T}$ have been included in $\Upsilon^\rho$, see the first three lines of \eqref{Upsilon}. By doing so, the second order derivatives of these tensors due to derivatives of $H^\rho$ in the Ricci tensor won't appear in the final system, where only the wave terms from \eqref{W 2} will be allowed (recall the second comment after Lemma \ref{lemma W}).
\item At order $\la^1$ and $\la^2$, we see the polarization tensors of $\F_\A$, $F^{(2,1)}_\A$, $F^{(2,2)}_\A$ and $\gg^{(3,h)}_{u,T}$ appearing as leading terms, where we recall that polarization tensors have been defined in Definition \ref{def pola}. Note that the polarization tensor of $\gg^{(3,e)}_{v,T}$ should also appear in \eqref{H 2 gauge}. However, we included the polarization tensor of $\gg^{(3,e)}_{v,T}$ in $\Upsilon^\rho$ in order to remove it from the final hierarchy, see the fourth line in \eqref{Upsilon}. The reason to do so will be explained after Proposition \ref{prop expression of the ricci}.
\end{enumerate}

\subsubsection{The Ricci tensor}\label{section ricci final}

In this section, we put together the results of Lemmas \ref{lemma W}, \ref{lemma P} and \ref{lemma H} and obtain the final expression of $R_{\a\b}(\g_\la)$ in the next proposition, whose proof is postponed to Appendix \ref{appendix proof prop ricci}.

\begin{prop}\label{prop expression of the ricci}
Under the assumptions \eqref{assumption}, the Ricci tensor of $\g_\la$ admits the expansion
\begin{align*}
R_{\a\b} (\g_\la) & = R^{(0)}_{\a\b} + \la R^{(1)}_{\a\b} + \la^2 R^{(\geq 2)}_{\a\b}.
\end{align*} 
The term $R^{(0)}_{\a\b}$ is given by
\begin{align}
R^{(0)}_{\a\b} & = \sum_\A \pth{ F_\A^2 -  \frac{1}{8} \left| F^{(1)}_\A \right|^2_{\g_0} } \dr_\mu u_\A \dr_\nu u_\A+ (R^{(0)}_{\mathrm{null}})_{\a\b} +  (R^{(0)}_{\mathrm{mixed}})_{\a\b}  , \label{R 0}
\end{align}
where 
\begin{align} 
2 (R^{(0)}_{\mathrm{null}})_{\a\b} & =  \sum_\A \sin\pth{\frac{u_\A}{\la}}\bigg( (\Ll_\A F^{(1)}_\A)_{\alpha\beta} -  \dr_{(\a}u_\A  \Pol{\F_\A }{u_\A}_{\b)} \non
\\&\hspace{4cm} -  \dr_{(\a}u_\A \pth{  \Pol{ F^{(2,1)}_\A}{u_\A}_{\b)} +(Q^{(0)}_\A)_{\b)} }\bigg)\label{R 0 null}
\\&\quad - 4 \sum_\A \cos\pth{\frac{2u_\A}{\la}} \dr_{(\alpha} u_\A\pth{\Pol{F^{(2,2)}_\A}{u_\A}_{\b)}  +   \frac{3}{32} \dr_{\beta)} u_\A  \left| F^{(1)}_\A \right|^2_{\g_0} }\non,
\\ 2(R^{(0)}_{\mathrm{mixed}})_{\a\b} & = \sum_{\A\neq\B,\pm} \cos\pth{\frac{u_\A\pm u_\B}{\la}}\pth{  -\mathcal{P}_{u_\A\pm u_\B}\pth{ F^{(2,\pm)}_{\A\B}}_{\a\b} + (\mathcal{I}^{(0,\pm)}_{\A\B})_{\a\b} }, \label{R 0 mixed}
\end{align}
with
\begin{align}
(Q^{(0)}_\A)_{\b} & \vcentcolon=  \g_0^{\mu\nu} \pth{ \dr_\mu (F^{(1)}_\A)_{\b\nu} - \half \dr_\b (F^{(1)}_\A)_{\mu\nu} }  ,\label{Q 0}
\\ (\mathcal{I}^{(0,\pm)}_{\A\B})_{\a\b} & \vcentcolon= - \frac{1}{4}   \pth{ (F^{(1)}_\B)_{L_\A L_\A}   (F^{(1)}_\A)_{\alpha\beta} + (F^{(1)}_\A)_{L_\B L_\B}   (F^{(1)}_\B)_{\alpha\beta} }   + (P^{(0,\pm)}_{\A\B})_{\a\b} \label{I 0 pm}
\\&\quad +\dr_{(\a}(u_\A\pm u_\B)\bigg(  - \frac{1}{8} \dr_{\b)} (u_\A\pm u_\B) \left| F^{(1)}_\A\cdot F^{(1)}_\B\right|_{\g_0}  \non
\\&\hspace{4cm}  - \frac{1}{4} (F^{(1)}_\B)^{\nu}_{L_\A} (F^{(1)}_\A)_{\nu\b)} \mp \frac{1}{4} (F^{(1)}_\A)^{\nu}_{L_\B} (F^{(1)}_\B)_{\nu\b)} \bigg). \non
\end{align}
The term $R^{(1)}_{\a\b}$ is given by
\begin{align}
R^{(1)}_{\a\b} & = (R^{(1)}_{\mathrm{null}})_{\a\b} +  (R^{(1)}_{\mathrm{mixed}})_{\a\b} + \half  (\g_0)_{\rho(\a}(\dr_{\b)}\Upsilon^\rho)^{(-1)}  ,\label{R 1}
\end{align}
where
\begin{align}
2(R^{(1)}_{\mathrm{null}})_{\a\b} & = -\sum_\A \cos\pth{\frac{u_\A}{\lambda}} \bigg(   (\Ll_\A\F_\A)_{\alpha\beta} + (\Ll_\A F^{(2,1)}_\A)_{\alpha\beta} + (\h_\lambda)_{L_\A L_\A} (F^{(1)}_\A)_{\alpha\beta} \label{R 1 null}
\\&\hspace{5cm}    - \D_{(\a} \Pol{\F_\A+F^{(2,1)}_\A}{u_\A}_{\b)} +  (\tilde{R}^{(1,1)}_\A)_{\a\b}    \bigg) \non
\\&\quad + 2\sum_\A\sin\pth{\frac{2u_\A}{\lambda}} \bigg(    (\Ll_\A F^{(2,2)}_\A)_{\alpha\beta}  -  \D_{(\a} \Pol{F^{(2,2)}_\A}{u_\A}_{\b)} +  (\tilde{R}^{(1,2)}_\A)_{\a\b}   \bigg) \non
\\&\quad + \sum_\A\bigg( 2  \sin\pth{\frac{u_\A}{\la}}  \sin\pth{\frac{2u_\A}{\la}}  - \cos\pth{\frac{u_\A}{\lambda}} \cos\pth{ \frac{2u_\A}{\lambda} } \bigg) (F^{(2,2)}_\A)_{L_\A L_\A} (F^{(1)}_\A)_{\alpha\beta} \non
\\&\quad  - \sum_\A \sin\pth{\frac{2u_\A}{\la}} (\F_\A)_{L_\A L_\A} (F^{(1)}_\A)_{\a\b}  \non
\\&\quad +  \sum_{\substack{u\in\mathcal{N}\\T\in\{\cos,\sin\}}}T\pth{\frac{u}{\la}} \dr_{(\a}u \pth{  -  \Pol{\gg^{(3,h)}_{u,T} }{u}_{\b)} + (\tilde{R}^{(1,h,T,u)}_\A)_{\b)} }  , \non
\\ 2(R^{(1)}_{\mathrm{mixed}})_{\a\b} & = \sum_{\substack{v\in\mathcal{I}\\T\in\{\cos,\sin\} }} T\left( \frac{v}{\lambda} \right) \pth{ \g_0^{-1}(\d v, \d v) (\gg^{(3,e)}_{v,T})_{\a\b} + (\tilde{R}^{(1,e,T,v)}_\A)_{\a\b} }. \label{R 1 mixed}
\end{align}
and
\begin{align}
\tilde{R}^{(1,1)}_\A & =  \cro{\pth{F^{(2,\pm)} + (F^{(1)})^2 }F^{(1)} + \dr^{\leq 2} F^{(1)} + (F^{(1)})^2},\label{Rtilde 1 1}
\\ \tilde{R}^{(1,2)}_\A & = \cro{ \pth{ F^{(2,1)} + \dr^{\leq 1} F^{(1)} }F^{(1)} }_{\a\b},\label{Rtilde 1 2}
\\ \tilde{R}^{(1,h,T,u)}_\A & = \Big\{\dr F^{(2,1)} + \dr F^{(2,2)} \label{R 1 h}
\\&\hspace{1cm} +  \pth{ \F + F^{(2,1)} + F^{(2,2)} + F^{(2,\pm)} + (F^{(1)})^2 + \dr F^{(1)} } \pth{1+F^{(1)}} \Big\}_{\b)}, \non
\\ \tilde{R}^{(1,e,T,v)}_\A & =  \Big\{\dr^{\leq 1}F^{(2,\pm)} + (F^{(1)})^2 \label{R 1 e}
\\&\hspace{1cm}+ \pth{ \F + F^{(2,1)} + F^{(2,2)} + F^{(2,\pm)}  + (F^{(1)})^2 + \dr^{\leq 1} F^{(1)} }\pth{1+F^{(1)}} \Big\}.\non
\end{align}
The term $R^{(\geq 2)}_{\a\b}$ is given by
\begin{equation}\label{R 2}
\begin{aligned}
2R^{(\geq 2)}_{\a\b} & =  -\tBox_{\g_\la}(\h_\la)_{\a\b}  - \sum_\A \sin\pth{\frac{u_\A}{\la}}\tBox_{\g_\la}(\F_\A)_{\a\b} 
\\&\quad -  \la \sum_{\substack{u\in\mathcal{N}\\T\in\{\cos,\sin\}}}T\pth{\frac{u}{\la}}\tBox_{\g_\la} (\gg^{(3,h)}_{u,T})_{\a\b}  - \la  \sum_{\substack{v\in\mathcal{I}\\T\in\{\cos,\sin\}}} T\left( \frac{v}{\lambda} \right)\tBox_{\g_\la} (\gg^{(3,e)}_{v,T})_{\a\b}
\\&\quad  + \tilde{R}^{(2)}_{\a\b} + \pth{ H^\rho \dr_\rho (\g_\la)_{\a\b} + (\g_\la)_{\rho(\a} \dr_{\b)}H^\rho }^{(\geq 2)}.
\end{aligned}
\end{equation}
where
\begin{equation}\label{R 2 w T}
\begin{aligned}
\tilde{R}^{(2)}_{\a\b} & = \Big\{\pth{ \g_\la^{-1}\g_\la^{-1}\dr\g_\la\dr\g_\la }^{(\geq 2)}  + \g_\la^{-1}\Big(\dr^{\leq 1}\gg^{(3,h)} + \dr^{\leq 1}\gg^{(3,e)} + \dr^{\leq 2} F^{(1)} + \dr^{\leq 1} \F 
\\&\hspace{6cm}+ \dr^{\leq 2}F^{(2,1)} + \dr^{\leq 2}F^{(2,2)} + \dr^{\leq 2} F^{(2,\pm)} \Big)  \Big\}_{\a\b}^\osc .
\end{aligned}
\end{equation}
\end{prop}

In Section \ref{section hierarchy} and \ref{section reformulation} below we will precisely state the hierarchy of equations and relations that will ensure $R_{\a\b}(\g_\la)=0$. Before that, we briefly discuss the mechanisms at stake, following the comments made after Lemmas \ref{lemma W}, \ref{lemma P} and \ref{lemma H}:
\begin{enumerate}
\item \textbf{Admissible terms.} The admissible terms at order $\la^0$ and $\la^1$, i.e the $\sin\pth{\frac{u_\A}{\la}}$ terms in \eqref{R 0 null} and the $\cos\pth{\frac{u_\A}{\la}}$ and $\sin\pth{\frac{2u_\A}{\la}}$ terms in \eqref{R 1 null} can be set to zero by imposing transport equations for $F^{(1)}_\A$, $\F_\A$, $F^{(2,1)}_\A$ and $F^{(2,2)}_\A$. The seemingly redundant presence of $\F_\A$ is justified by the term $(\h_\lambda)_{L_\A L_\A} (F^{(1)}_\A)_{\alpha\beta}$ in \eqref{R 1 null}, which will be absorbed by the RHS of the transport equation for $\F_\A$.
\item \textbf{Forbidden terms at order $\la^0$.} In \eqref{R 0 null}, the forbidden terms in $\cos\pth{\frac{2u_\A}{\la}}$ can be set to zero by imposing a polarization condition for $F^{(2,2)}_\A$. This is only possible if the forbidden term that need to be absorbed has the tensorial structure $\dr_{(\a}u_\A Q_{\b)}$ for some 1-form $Q$. Here, this is even better since $Q=\d u_\A$ (recall \eqref{P 0 2})! If the polarization condition for $F^{(2,2)}_\A$ holds, this extra bit of structure would imply $(F^{(2,2)}_\A)_{L_\A L_\A}=0$ (where we used \eqref{Pol L}).
\item \textbf{Forbidden terms at order $\la^1$.} In \eqref{R 1 null}, two types of forbidden terms occur. First, the term $(F^{(2,2)}_\A)_{L_\A L_\A} (F^{(1)}_\A)_{\alpha\beta}$ coming from the quasi-linear wave operator oscillates like $\cos\pth{\frac{3u_\A}{\la}}$. However, the previous comment precisely says that it vanishes thanks to the structure of the semi-linear forbidden terms that the polarization condition for $F^{(2,2)}_\A$ is absorbing. Second, the term $\dr_{(\a}u(\tilde{R}^{(1,h,T,u)}_\A)_{\b)}$ precisely has the required structure to be absorbed by a polarization condition for $\gg^{(3,h)}_{u,T}$.
\item \textbf{Mixed harmonics at order $\la^0$.} The term $R^{(0)}_{\mathrm{mixed}}$ given by \eqref{R 0 mixed} will be simply canceled by asking $F^{(2,\pm)}_{\A\B}$ to solve 
\begin{align*}
\mathcal{P}_{u_\A\pm u_\B}\pth{ F^{(2,\pm)}_{\A\B}} = \mathcal{I}^{(0,\pm)}_{\A\B}.
\end{align*}
Since $\g_0^{-1}(\d(u_\A\pm u_\B),\d(u_\A\pm u_\B))\neq 0$ (recall \eqref{coherence}), Lemma \ref{lem Pv} shows that this equation is solvable if and only if $\Pol{\mathcal{I}^{(0,\pm)}_{\A\B}}{u_\A\pm u_\B}=0$. We will check this in Section \ref{section ensuring R 0} but we draw the reader's attention on the importance of this simple algebraic fact which is a reminder of Einstein vacuum equations' transparency.
\item \textbf{Mixed harmonics at order $\la^1$.} The structure of $R^{(1)}_{\mathrm{mixed}}$ differs from the one of $R^{(0)}_{\mathrm{mixed}}$ since instead of having the operator $\PP_v$ acting on $\gg^{(3,e)}_{v,T}$, we only have the part of $\PP_v$ coming from the wave part of the Ricci tensor, i.e $-\g_0^{-1}(\d v,\d v)\mathrm{Id}$. Opposite to $\PP_v$, this operator has full range and no condition on the RHS needs to hold, which simplifies greatly our construction since we don't need the exact expression of $\tilde{R}^{(1,e,T,v)}_\A$. The absence of the other half of $\PP_v$ is due to our choice of generalised wave gauge, i.e our choice of $\Upsilon^\rho$ (recall the second comment after Lemma \ref{lemma H}). We see here how the gauge helps us recover true ellipticity when considering mixed harmonics, as it usually does help us recover true hyperbolicity when solving \eqref{EVE}.
\end{enumerate}

\subsection{The hierarchy}\label{section hierarchy}

In this section, following the general mechanisms presented after Proposition \ref{prop expression of the ricci}, we derive a hierarchy of equations and relations for the various terms in $\g_\la$ which will ensure $R_{\a\b} (\g_\la) =0$ by ensuring $R^{(0)}_{\a\b}=0$, $R^{(1)}_{\a\b}=0$ and $R^{(\geq 2)}_{\a\b}=0$. 

\subsubsection{Ensuring $R^{(0)}_{\a\b}=0$}\label{section ensuring R 0}

Looking at \eqref{R 0}, we first want the non-oscillating terms to vanish. For that, we impose the following backreaction condition for $F^{(1)}_\A$:
\begin{align*}
\left| F^{(1)}_\A \right|^2_{\g_0} & = 8 F_\A^2.
\end{align*}
We wish to have $R^{(0)}_{\mathrm{null}}=0$. Looking at \eqref{R 0 null}, we cancel the admissible harmonics by imposing the following transport equation for $F^{(1)}_\A$
\begin{align*}
\Ll_\A F^{(1)}_\A & = 0,
\end{align*}
as well as a polarization condition for $\F_\A$ and $F^{(2,1)}_\A$ 
\begin{align*}
\Pol{\F_\A }{u_\A} & = 0,
\\ \Pol{F^{(2,1)}_\A}{u_\A} & = - Q^{(0)}_\A,
\end{align*}
where $Q^{(0)}_\A$ is defined in \eqref{Q 0} (this is not clear that \eqref{Q 0} actually defines a tensor, see Remark \ref{remark Q0} for a justification of this fact). We cancel the forbidden harmonics in $R^{(0)}_{\mathrm{null}}$ by imposing a polarization condition for $F^{(2,2)}_\A$
\begin{align*}
\Pol{F^{(2,2)}_\A}{u_\A} & = -  \frac{3}{32} \left| F^{(1)}_\A \right|^2_{\g_0} \d u_\A.
\end{align*}
For convenience, we define the following tensors
\begin{align}
V^{(2,1)}_\A & \vcentcolon=  \Pol{F^{(2,1)}_\A}{u_\A} + Q^{(0)}_\A,\label{def 1 V 2 1}
\\ V^{(2,2)}_\A & \vcentcolon= \Pol{F^{(2,2)}_\A}{u_\A} +  \frac{3}{32} \left| F^{(1)}_\A \right|^2_{\g_0} \d u_\A,\label{def 1 V 2 2}
\end{align}
so that the polarization conditions for $F^{(2,1)}_\A$ and $F^{(2,2)}_\A$ rewrite
\begin{align}
V^{(2,1)}_\A & = 0,\label{premiere V 2 1 = 0}
\\ V^{(2,2)}_\A & = 0.\label{premiere V 2 2 = 0}
\end{align}

\begin{remark}\label{remark Q0}
One might wonder about the status of the polarization conditions \eqref{premiere V 2 1 = 0}-\eqref{premiere V 2 2 = 0}, in particular if they are gauge conditions or not. This is in fact not the case since $V^{(2,1)}_\A$ and $V^{(2,2)}_\A$ as defined in \eqref{def 1 V 2 1} and \eqref{def 1 V 2 2} are tensors, so that their vanishing is not a gauge condition. Note that the tensoriality of $V^{(2,2)}_\A$ is obvious from \eqref{def 1 V 2 2}, while the tensoriality of $Q^{(0)}_\A$ in \eqref{def 1 V 2 1} follows from its alternative expression 
\begin{align*}
(Q^{(0)}_\A)_\si & = \g_0^{\mu\nu} \pth{ \D_\mu (F^{(1)}_\A)_{\si\nu} - \half \D_\si (F^{(1)}_\A)_{\mu\nu}}, 
\end{align*}
which follows itself from \eqref{Q 0} and \eqref{wave condition BG}.
\end{remark}

We cancel the mixed term $R^{(0)}_{\mathrm{mixed}}$ given by \eqref{R 0 mixed} by imposing the following equation for $F^{(2,\pm)}_{\A\B}$
\begin{align}
\PP_{u_\A\pm u_\B}\pth{F^{(2,\pm)}_{\A\B}} & = \mathcal{I}^{(0,\pm)}_{\A\B}, \label{eq F2pm first}
\end{align}
where $\mathcal{I}^{(0,\pm)}_{\A\B}$ is defined by \eqref{I 0 pm}. Thanks to Lemma \ref{lem Pv}, this equation is solvable if and only its RHS satisfies some polarization condition.

\begin{lemma}\label{lem F2pm}
We assume that \eqref{assumption} holds. If we define
\begin{align*}
F^{(2,\pm)}_{\A\B} & \vcentcolon = -\frac{\mathcal{I}^{(0,\pm)}_{\A\B} }{\g_0^{-1}(\d(u_\A\pm u_\B),\d(u_\A\pm u_\B))},
\end{align*}
then \eqref{eq F2pm first} holds.
\end{lemma}

\begin{proof}
According to Lemma \ref{lem Pv} it suffices to show that 
\begin{align}
\Pol{ \mathcal{I}_{\A\B}^{(0,\pm)} }{ u_\A\pm u_\B } = 0.\label{Pol I 0 pm}
\end{align}
For clarity we define
\begin{align*}
A_{\a\b} & \vcentcolon =  - \frac{1}{4}   \pth{ (F^{(1)}_\B)_{L_\A L_\A}   (F^{(1)}_\A)_{\alpha\beta} + (F^{(1)}_\A)_{L_\B L_\B}   (F^{(1)}_\B)_{\alpha\beta} },
\\ B_{\a\b} & \vcentcolon = \dr_{(\a}(u_\A\pm u_\B)\bigg(  - \frac{1}{8} \dr_{\b)} (u_\A\pm u_\B) \left| F^{(1)}_\A\cdot F^{(1)}_\B\right|_{\g_0}  
  - \frac{1}{4} (F^{(1)}_\B)^{\nu}_{L_\A} (F^{(1)}_\A)_{\nu\b)} \mp \frac{1}{4} (F^{(1)}_\A)^{\nu}_{L_\B} (F^{(1)}_\B)_{\nu\b)} \bigg).
\end{align*}
so that $ (\mathcal{I}^{(0,\pm)}_{\A\B})_{\a\b} = A_{\a\b}  + B_{\a\b} + (P^{(0,\pm)}_{\A\B})_{\a\b}$. Thanks to \eqref{assumption} we first have
\begin{align*}
\Pol{A}{u_\A\pm u_\B}_\a & =   \frac{1}{4}   \pth{ \pm(F^{(1)}_\B)_{L_\A L_\A}   (F^{(1)}_\A)_{\alpha L_\B} + (F^{(1)}_\A)_{L_\B L_\B}   (F^{(1)}_\B)_{\alpha L_\A} }.
\end{align*}
For $B_{\a\b}$, \eqref{Pol dv S} gives
\begin{align*}
\Pol{B}{u_\A\pm u_\B}_\a & = \mp \half \g_0^{-1}(\d u_\A, \d u_\B)   \bigg( (F^{(1)}_\B)^{\nu}_{L_\A}  (F^{(1)}_\A)_{\a\nu} \pm (F^{(1)}_\A)^{\nu}_{L_\B} (F^{(1)}_\B)_{\a\nu} 
\\&\hspace{5cm} + \half  \left| F^{(1)}_\A \cdot F^{(1)}_\B \right|_{g_0}  \dr_\a (u_\A\pm u_\B)  \bigg).
\end{align*}
For $(P^{(0,\pm)}_{\A\B})_{\a\b}$, \eqref{P 0 pm} and \eqref{assumption} give
\begin{align*}
\Pol{P^{(0,\pm)}_{\A\B}}{u_\A \pm u_\B}_\a & =  - (P^{(0,\pm)}_{\A\B})_{\a (L_\A \pm L_\B)} + \half \tr_{\g_0}P^{(0,\pm)}_{\A\B} ((L_\A)_\a \pm (L_\B)_\a)
\\& = \pm \frac{1}{4} \dr_{\a} (u_\A\pm u_\B) \g_0^{-1}(\d u_\A,\d u_\B) \left| F^{(1)}_\A\cdot F^{(1)}_\B \right|_{\g_0}
\\&\quad + \frac{1}{2}\g_0^{-1}(\d u_\A,\d u_\B) \pth{ (F^{(1)}_\A)_{L_\B\si} (F^{(1)}_\B)_{\a}^\si  \pm (F^{(1)}_\B)_{L_\A\si} (F^{(1)}_\A)_{\a}^\si }
\\&\quad - \frac{1}{4} (F^{(1)}_\A)_{ L_\B L_\B} (F^{(1)}_\B)_{\a L_\A}  \mp\frac{1}{4} (F^{(1)}_\A)_{\alpha L_\B} (F^{(1)}_\B)_{L_\A L_\A}
\\& = - \Pol{A}{u_\A\pm u_\B}_\a - \Pol{B}{u_\A\pm u_\B}_\a,
\end{align*}
which concludes the proof of \eqref{Pol I 0 pm} and thus of the lemma.
\end{proof}

\begin{remark}
If one is interested in the strict equivalent to Choquet-Bruhat's singlephase approximate construction from \cite{ChoquetBruhat1969} in the multiphase case, the above conditions are sufficient since they indeed lead to $R_{\a\b}(\g_\la)=\GO{\la}$. 
\end{remark}

\subsubsection{Ensuring $R^{(1)}_{\a\b}=0$}\label{section ensuring R 1}

Looking at \eqref{R 1}, we wish to have $R^{(1)}_{\mathrm{null}}=0$ and $R^{(1)}_{\mathrm{mixed}}=0$. We first consider $R^{(1)}_{\mathrm{null}}$ and rewrite it with the help of the tensors $V^{(2,1)}_\A$ and $V^{(2,2)}_\A$ defined by \eqref{def 1 V 2 1} and \eqref{def 1 V 2 2}. Thanks to \eqref{R 1 null} we obtain
\begin{equation}\label{R 1 null bis}
\begin{aligned}
2(R^{(1)}_{\mathrm{null}})_{\a\b} & = -\sum_\A \cos\pth{\frac{u_\A}{\lambda}} \bigg(   (\Ll_\A\F_\A)_{\alpha\beta} + (\Ll_\A F^{(2,1)}_\A)_{\alpha\beta} + (\h_\lambda)_{L_\A L_\A} (F^{(1)}_\A)_{\alpha\beta} +  (\tilde{R}^{(1,1)}_\A)_{\a\b}
\\&\hspace{5cm}   - \D_{(\a} \Pol{\F_\A}{u_\A}_{\b)}  - \D_{(\a} (V^{(2,1)}_\A)_{\b)}  + \D_{(\a} (Q^{(0)}_\A)_{\b)}     \bigg) 
\\&\quad + 2\sum_\A\sin\pth{\frac{2u_\A}{\lambda}} \bigg(    (\Ll_\A F^{(2,2)}_\A)_{\alpha\beta}  -  \D_{(\a} (V^{(2,2)}_\A)_{\b)} 
\\&\hspace{4.5cm} + \frac{3}{32}\D_{(\a} \pth{ \left| F^{(1)}_\A \right|^2_{\g_0}\d u_\A}_{\b)} +  (\tilde{R}^{(1,2)}_\A)_{\a\b}   \bigg)  
\\&\quad - \sum_\A\bigg( 2  \sin\pth{\frac{u_\A}{\la}}  \sin\pth{\frac{2u_\A}{\la}}  - \cos\pth{\frac{u_\A}{\lambda}} \cos\pth{ \frac{2u_\A}{\lambda} } \bigg) (V^{(2,2)}_\A)_{L_\A} (F^{(1)}_\A)_{\alpha\beta}  
\\&\quad  - \sum_\A \sin\pth{\frac{2u_\A}{\la}} (\F_\A)_{L_\A L_\A} (F^{(1)}_\A)_{\a\b} 
\\&\quad +  \sum_{u\in\mathcal{N},T\in\{\cos,\sin\}}T\pth{\frac{u}{\la}} \dr_{(\a}u \pth{  -  \Pol{\gg^{(3,h)}_{u,T} }{u}_{\b)} + (\tilde{R}^{(1,h,T,u)}_\A)_{\b)} }  . 
\end{aligned}
\end{equation}
Since the above polarization conditions would imply $V^{(2,1)}_\A=V^{(2,2)}_\A=\Pol{\F_\A}{u_\A}=0$, we cancel the admissible harmonics in $(R^{(1)}_{\mathrm{null}})_{\a\b}$ by imposing transport equations for $F^{(2,1)}_\A$ and $F^{(2,2)}_\A$:
\begin{align}
(\Ll_\A F^{(2,1)}_\A)_{\a\b} & = -\D_{(\a} (Q^{(0)}_\A)_{\b)} -  (\tilde{R}^{(1,1)}_\A)_{\a\b} , \label{premiere eq F 2 1}
\\ (\Ll_\A F^{(2,2)}_\A)_{\alpha\beta}  & = - \frac{3}{32}\D_{(\a} \pth{ \left| F^{(1)}_\A \right|^2_{\g_0}\d u_\A}_{\b)} -  (\tilde{R}^{(1,2)}_\A)_{\a\b}, \label{premiere eq F 2 2}
\end{align}
The equation for $\F_\A$ needs to be adjusted so that the coupling with the equation for $\h_\la$ leads to a well-posed system in the high-frequency limit $\la\to 0$. As in \cite{Touati2023a} we impose the following transport equation
\begin{align}
\Ll_\A\F_\A & = - \Pi_\leq \pth{ (\h_\lambda)_{L_\A L_\A} }F^{(1)}_\A ,\label{premiere eq F}
\end{align}
where the operator $\Pi_\leq$ is defined by
\begin{align*}
\Pi_\leq (f) = \mathcal{F}^{-1}\pth{\chi_\la \mathcal{F}(f)},
\end{align*}
where $\chi_\la(\xi)=\chi(\la\xi)$ for $\chi:\R^3\longrightarrow [0,1]$ a smooth function supported in $\{|\xi|\leq 2\}$ and such that $\chi_{|_{\{|\xi|\leq 1\}}}=1$ and where $\mathcal{F}$ is the standard Fourier transform on $\R^3$. In order to cancel the non-tangential terms containing forbidden harmonics we impose a polarization condition for $\gg^{(3,h)}_{u,T}$
\begin{align}
 \Pol{\gg^{(3,h)}_{u,T} }{u}& = \tilde{R}^{(1,h,T,u)}_\A. \label{premiere pola g 3 h}
\end{align}
In order to cancel $R^{(1)}_{\mathrm{mixed}}$ given by \eqref{R 1 mixed}, we simply define $\gg^{(3,e)}_{v,T}$ by
\begin{align}
\gg^{(3,e)}_{v,T} & = - \frac{1}{\g_0^{-1}(\d v, \d v)} \tilde{R}^{(1,e,T,v)}_\A,\label{premiere def g 3 e}
\end{align}
where we recall that if $v\in \mathcal{I}$, then $\g_0^{-1}(\d v, \d v)\neq 0$ (see Section \ref{section BG}).

\subsubsection{Ensuring $R^{(\geq 2)}_{\a\b}=0$}\label{section ensuring R 2}

Thanks to its expression given by \eqref{R 2}, in order to cancel $R^{(\geq 2)}_{\a\b}$ it suffices to impose a wave equation for the remainder $\h_\la$. This equation needs to be adjusted to handle the coupling with the equation for $\F_\A$, and we also need to set aside the problematic derivatives all contained in $\Upsilon^\rho$. Therefore we impose the following wave equation for $\h_\la$
\begin{align*}
\tBox_{\g_\la}(\h_\la)_{\a\b}  & = - \sum_\A \sin\pth{\frac{u_\A}{\la}}\tBox_{\g_\la}(\F_\A)_{\a\b} -  \la \sum_{\substack{u\in\mathcal{N}\\T\in\{\cos,\sin\}}}T\pth{\frac{u}{\la}}\tBox_{\g_\la} (\gg^{(3,h)}_{u,T})_{\a\b} 
\\&\quad  - \la  \sum_{\substack{v\in\mathcal{I}\\T\in\{\cos,\sin\}}} T\left( \frac{v}{\lambda} \right)\tBox_{\g_\la} (\gg^{(3,e)}_{v,T})_{\a\b} + \tilde{R}^{(2)}_{\a\b}  +\pth{ \mathring{H}^\rho \dr_\rho (\g_\la)_{\a\b} + (\g_\la)_{\rho(\a} \dr_{\b)}\mathring{H}^\rho }^{(\geq 2)}
\\&\quad - \frac{1}{\la}\sum_\A \cos\pth{\frac{u_\A}{\la}}\Pi_\geq\pth{(\h_\lambda)_{L_\A L_\A}}(F^{(1)}_\A)_{\a\b},
\end{align*}
where the operator $\Pi_\geq$ is defined by $\Pi_\geq=\mathrm{Id}-\Pi_\leq$. We also impose the following generalised wave gauge condition
\begin{align*}
\Upsilon^\rho & = 0 .
\end{align*}

\subsection{Reformulation of Theorem \ref{theo evol}}\label{section reformulation}

The conclusion of Sections \ref{section ensuring R 0}, \ref{section ensuring R 1} and \ref{section ensuring R 2} is the following: if the tensors $F^{(1)}_\A$, $\F_\A$, $F^{(2,1)}_\A$, $F^{(2,2)}_\A$ and $\h_\la$ solve the system
\begin{align}
\Ll_\A F^{(1)}_\A & = 0,\label{eq F1}
\\ (\Ll_\A F^{(2,1)}_\A)_{\a\b} & = -\D_{(\a} (Q^{(0)}_\A)_{\b)} -  (\tilde{R}^{(1,1)}_\A)_{\a\b} , \label{eq F21}
\\ (\Ll_\A F^{(2,2)}_\A)_{\alpha\beta}  & = - \frac{3}{32}\D_{(\a} \pth{ \left| F^{(1)}_\A \right|^2_{\g_0}\d u_\A}_{\b)} -  (\tilde{R}^{(1,2)}_\A)_{\a\b},\label{eq F22}
\\ \Ll_\A\F_\A & = - \Pi_\leq \pth{ (\h_\lambda)_{L_\A L_\A} }F^{(1)}_\A,\label{eq F}
\\ \tBox_{\g_\la}(\h_\la)_{\a\b} & = - \sum_\A \sin\pth{\frac{u_\A}{\la}}\tBox_{\g_\la}(\F_\A)_{\a\b} -  \la \sum_{\substack{u\in\mathcal{N}\\T\in\{\cos,\sin\}}}T\pth{\frac{u}{\la}}\tBox_{\g_\la} (\gg^{(3,h)}_{u,T})_{\a\b} \label{eq h}
\\&\quad  - \la  \sum_{\substack{v\in\mathcal{I}\\T\in\{\cos,\sin\}}} T\left( \frac{v}{\la} \right)\tBox_{\g_\la} (\gg^{(3,e)}_{v,T})_{\a\b} + \tilde{R}^{(2)}_{\a\b}  +\pth{ \mathring{H}^\rho \dr_\rho (\g_\la)_{\a\b} + (\g_\la)_{\rho(\a} \dr_{\b)}\mathring{H}^\rho }^{(\geq 2)}\non
\\&\quad - \frac{1}{\la}\sum_\A \cos\pth{\frac{u_\A}{\la}}\Pi_\geq\pth{(\h_\la)_{L_\A L_\A}}(F^{(1)}_\A)_{\a\b},\non
\end{align}
together with the polarization conditions 
\begin{align}
\Pol{F^{(1)}_\A}{u_\A} & = 0,\label{pola F1}
\\ F^{(1)}_{L_\A \Lb_\A} & = 0,\label{pola F1 bis}
\\ \Pol{\F_\A}{u_\A} & = 0,\label{pola F}
\\ V^{(2,1)}_\A & = 0,\label{pola F21}
\\ V^{(2,2)}_\A & = 0,\label{pola F22}
\\ \Pol{\gg^{(3,h)}_{u,T} }{u} & =  \tilde{R}^{(1,h,T,u)}_\A,\label{pola g3h}
\end{align}
the backreaction condition
\begin{align}
\left| F^{(1)}_\A \right|^2_{\g_0} = 8 F_\A^2 \label{energy F1},
\end{align}
the generalised wave gauge condition 
\begin{align}
\Upsilon^\rho & =0,\label{GWC}
\end{align}
with $F^{(2,\pm)}_{\A\B}$ and $\gg^{(3,e)}_{v,T}$ defined by
\begin{align}
(F^{(2,\pm)}_{\A\B})_{\a\b} & \vcentcolon = - \frac{(\mathcal{I}^{(0,\pm)}_{\A\B})_{\a\b} }{\g_0^{-1}(\d(u_\A\pm u_\B),\d(u_\A\pm u_\B))}, \label{def F2pm}
\\ \gg^{(3,e)}_{v,T} & \vcentcolon = - \frac{1}{\g_0^{-1}(\d v, \d v)} \tilde{R}^{(1,e,T,v)}_\A, \label{def g3e}
\end{align}
then the metric $\g_\la$ defined by \eqref{expansion metric} solves the Einstein vacuum equations \eqref{EVE}. The initial data for the transport-wave system \eqref{eq F1}-\eqref{eq h} consist in initial values on $\Si_0$ for
\begin{align*}
(F^{(1)}_\A)_{\a\b},\quad (F^{(2,1)}_\A)_{\a\b},\quad (F^{(2,2)}_\A)_{\a\b},\quad (\F_\A)_{\a\b},\quad (\h_\la)_{\a\b} \quad \text{and} \quad \T_\la (\h_\la)_{\a\b},
\end{align*}
where $\T_\la$ is the future-directed unit normal to $\Si_0$ for $\g_\la$ (which can be computed from $\g_{\la|_{\Si_0}}$). Therefore, the task of proving Theorem \ref{theo evol} is divided into three steps:
\begin{enumerate}
\item In Section \ref{section initial data}, we construct initial data such that the algebraic conditions \eqref{pola F1}-\eqref{GWC} and the constraint equations are satisfied on $\Si_0$.
\item In Section \ref{section solving the system}, we solve the system \eqref{eq F1}-\eqref{eq h} on $[0,1]\times \R^3$.
\item In Section \ref{section propagation}, we show that the algebraic conditions \eqref{pola F1}-\eqref{GWC} actually hold on the full spacetime $[0,1]\times \R^3$.
\end{enumerate}

\begin{remark}\label{remark F2pm}
Thanks to \eqref{def F2pm}, \eqref{I 0 pm} and \eqref{P 0 pm} we have $F^{(2,\pm)}_{\A\B} = \cro{ (F^{(1)})^2 }$. This would be used without mention in the rest of this article.
\end{remark}

\begin{remark}\label{remark g3h}
The tensors $\gg^{(3,h)}_{u,T}$ only need to satisfy the polarization condition \ref{pola g3h} and in particular are not obtained by solving partial differential equations. Therefore, we can simply define $\gg^{(3,h)}_{u,T}$ for $u=ku_\A$ by setting
\begin{align*}
(\gg^{(3,h)}_{u,T})_{L_\A L_\A} & \vcentcolon = - \frac{1}{k} (\tilde{R}^{(1,h,T,u)}_\A)_{L_\A},
\\ (\gg^{(3,h)}_{u,T})_{L_\A e^{(\i)}_\A} & \vcentcolon = - \frac{1}{k} (\tilde{R}^{(1,h,T,u)}_\A)_{e^{(\i)}_\A},
\\ (\gg^{(3,h)}_{u,T})_{e^{(1)}_\A e^{(1)}_\A} & \vcentcolon = - \frac{1}{k} (\tilde{R}^{(1,h,T,u)}_\A)_{\Lb_\A},
\end{align*}
and by setting all the other components of $\gg^{(3,h)}_{u,T}$ in the frame $\pth{L_\A,\Lb_\A,e^{(1)}_\A,e^{(2)}_\A}$ to 0. In particular, thanks to \eqref{R 1 h} this implies
\begin{align}
\gg^{(3,h)}_{u,T} & = \Big\{\dr F^{(2,1)} + \dr F^{(2,2)} +  \pth{ \F + F^{(2,1)} + F^{(2,2)} + (F^{(1)})^2 + \dr F^{(1)} } \pth{1+F^{(1)}} \Big\} , \label{g3h}
\end{align}
where we already applied Remark \ref{remark F2pm} and replaced $F^{(2,\pm)}$ by $(F^{(1)})^2$ in the schematic notation.
\end{remark}


\section{Construction of initial data}\label{section initial data}

In this section, we construct initial data for the system \eqref{eq F1}-\eqref{eq h} on $\Si_0$.  They need to be such that the algebraic conditions \eqref{pola F1}-\eqref{GWC} hold on $\Si_0$ and also such that the constraint equations are satisfied by the induced metric on $\Si_0$ and the second fundamental form of $\Si_0$. These equations are
\begin{equation}\label{constraint equations}
\begin{aligned}
R(g) - |k|^2_g + (\tr_g k)^2 & = 0,
\\ \div_g k - \d \tr_g k & = 0,
\end{aligned}
\end{equation}
where $R(g)$ is the scalar curvature of $g$.

\subsection{From a seed to solutions to the constraint equations}

The construction of initial data is based on a seed, with which we first construct a family of solutions to the constraint equations and then pick a particular element of this family to define appropriate initial data for the system \eqref{eq F1}-\eqref{eq h}. 

\begin{mydef}\label{def seed}
A seed is a family $\pth{ \bar{F}^{(1)}_\A}_{\A\in\mathcal{A}}$ of symmetric 2-tensors on $\Si_0$ such that for all $\A\in\mathcal{A}$ we have
\begin{align}\label{seed form}
\bar{F}^{(1)}_\A = F_\A \pth{ \th^+_\A \pth{ e^{(1),\flat}_\A\otimes e^{(1),\flat}_\A - e^{(2),\flat}_\A\otimes e^{(2),\flat}_\A } + \th^\times_\A \pth{ e^{(1),\flat}_\A\otimes e^{(2),\flat}_\A + e^{(2),\flat}_\A\otimes e^{(1),\flat}_\A }  },
\end{align}
where $e^{(\i),\flat}_\A$ denotes the 1-form canonically associated to the vector field $e^{(\i)}_\A$ by $g_0$ and where $\pth{ \th^+_\A, \th^\times_\A}$ are numbers satisfying $\pth{\th^+_\A}^2  + \pth{\th^\times_\A}^2  = 4$.
\end{mydef}

The expression \eqref{seed form} directly implies that $\bar{F}^{(1)}_\A$ satisfies
\begin{equation}\label{assum F1}
\begin{aligned}
\tr_{g_0}\bar{F}^{(1)}_\A & = 0,
\\ (\bar{F}^{(1)}_\A)_{N_\A i} & = 0,
\\ \left|\bar{F}^{(1)}_\A\right|^2_{g_0} & = 8F_\A^2.
\end{aligned}
\end{equation}
Moreover, the assumptions made on $F_\A$ in Section \ref{section BG} imply that $ \bar{F}^{(1)}_\A$ is supported in $B_R$ and satisfies
\begin{align}\label{estim seed}
\l  \bar{F}^{(1)}_\A \r_{H^N(\Si_0)} \lesssim \e.
\end{align}

\begin{remark}\label{remark seed}
Note the analogy between \eqref{seed form} and the TT gauge of linearized gravity, where $\th^+_\A$ and $\th^\times_\A$ would play the role as the coefficients of the two possible polarizations for the gravitational wave.
\end{remark}

\noindent Before stating our main result on the constraint equations, we introduce two operators acting on the space of symmetric 2-tensors on $\Si_0$:
\begin{align}
\bar{\PP}^{[1]}_v(S) & = S - \frac{1}{2} \pth{ \tr_{g_0}S - S_{N_vN_v} }g_0,\label{def P1(T)}
\\ \bar{\PP}^{[2]}_v(S)_{ij} & = S_{ij}  +  (N_v)_{(i}  \pth{(N_v)_{j)} \tr_{g_0}S - S_{N_v j)}} - \half \pth{ \tr_{g_0}S - S_{N_v N_v}} (g_0)_{ij}, \label{def P2(T)}
\end{align}
where $v$ is a scalar function on $\Si_0$ such that $|\nab_{g_0} v|_{g_0}\neq 0$ and 
\begin{align*}
N_v \vcentcolon =\frac{\nab_{g_0} v}{|\nab_{g_0} v|_{g_0}}.
\end{align*}
Note that $N_\A=N_{u_\A}$ and $N^{(\pm)}_{\A\B}=N_{u_\A\pm u_\B}$ where $N_\A$ and $N^{(\pm)}_{\A\B}$ are defined in \eqref{def vector fields init}. Next lemma gives important properties of these operators which will be used in Section \ref{section initial data final}, its proof is left to the reader.

\begin{lemma}\label{lem P 1et2}
Let $v$ a scalar function on $\Si_0$ such that $|\nab_{g_0} v|_{g_0}\neq 0$. We have $\bar{\PP}^{[1]}_v\circ\bar{\PP}^{[1]}_v=\bar{\PP}^{[1]}_v$ and $\bar{\PP}^{[2]}_v\circ\bar{\PP}^{[2]}_v=\bar{\PP}^{[2]}_v$. Moreover we have
\begin{align}
\mathrm{ran}\bar{\PP}^{[1]}_v & =  \enstq{S \text{ symmetric 2-tensors on $\Si_0$}}{\tr_{g_0}S = S_{N_vN_v}} ,\label{image P1}
\\ \mathrm{ran}\bar{\PP}^{[2]}_v & =  \enstq{S \text{ symmetric 2-tensors on $\Si_0$}}{(\tr_{g_0}S) (N_v)_i  = S_{iN_v}} .\label{image P2}
\end{align}
\end{lemma}

\noindent Finally, if $X$ and $N$ are two vector fields on $\Si_0$ we define
\begin{align}
(N\tilde{\otimes}X)_{ij} \vcentcolon= N_{(i} X_{j)} - \half X_N (g_0)_{ij},\label{def produit}
\end{align}
where $X_N= g_0(X,N)$. We are now ready to state our main result on the constraint equations.

\begin{prop}\label{prop constraint main}
Let $\pth{ \bar{F}^{(1)}_\A}_{\A\in\mathcal{A}}$ be a seed. Let $\pth{\ka^{(1,1)}_\A, \ka^{(1,2)}_\A}_{\A\in\mathcal{A}}$ and $\pth{\ga^{(2,\pm)}_{\A\B},\ka^{(1,\pm)}_{\A\B}}_{\A,\B\in\mathcal{A}, \A\neq \B}$ be families of symmetric 2-tensors on $\Si_0$ supported in $B_R$ with the symmetry property $\ga^{(2,\pm)}_{\A\B}=\ga^{(2,\pm)}_{\B\A}$ and $\ka^{(1,\pm)}_{\A\B}=\ka^{(1,\pm)}_{\B\A}$ and such that
\begin{align}
\l \ga^{(2,\pm)}_{\A\B} \r_{H^N(\Si_0)} + \l  \ka^{(1,1)}_\A\r_{H^{N-1}(\Si_0)} + \l  \ka^{(1,2)}_\A \r_{H^{N-1}(\Si_0)} + \l  \ka^{(1,\pm)}_{\A\B} \r_{H^{N-1}(\Si_0)} \leq \e,\label{estim ga kappa param}
\end{align}
If $\e$ is small enough, there exists $(g_\la,k_\la)$ a solution of the constraint equations on $\Si_0$ of the form
\begin{equation}\label{g la}x
\begin{aligned}
g_\la & = g_0 + \la \sum_\A \cos\pth{ \frac{u_\A}{\la}} \bar{F}^{(1)}_\A +  \lambda^2 \sum_\A \sin\pth{\frac{u_\A}{\la}}  4\wc{\ffi}^{(2,1)}_{\A} g_0  +  \lambda^2 \sum_\A \cos\pth{\frac{2u_\A}{\la}} \frac{3}{4} F_\A^2  g_0 
\\&\quad + \la^2\sum_{\A\neq\B,\pm}\cos\pth{\frac{u_\A\pm u_\B}{\la}}\pth{ \bar{\PP}^{[1]}_{u_\A\pm u_\B}\pth{\ga^{(2,\pm)}_{\A\B}}  + 4\wc{\ffi}^{(2,\pm)}_{\A\B}g_0  } + \la^2 h_\la
\end{aligned}
\end{equation}
and
\begin{equation}\label{k la}
\begin{aligned}
k_\la & = k_0 + \half \sum_\A \sin\pth{ \frac{u_\A}{\la}} |\nab_{g_0} u_\A|_{g_0}\bar{F}^{(1)}_\A 
\\&\quad + \la\sum_\A \cos\pth{\frac{u_\A}{\la}} \pth{ \bar{\PP}^{[2]}_{u_\A}\pth{\ka^{(1,1)}_\A} +  \frac{N_\A\tilde{\otimes} \wc{X}^{(2,1)}_\A}{|\nab_{g_0} u_\A|_{g_0}} }
\\&\quad + \la\sum_\A \sin\pth{\frac{2u_\A}{\la}} \pth{ \bar{\PP}^{[2]}_{u_\A}\pth{\ka^{(1,2)}_\A}  -\frac{3}{2} |\nab_{g_0} u_\A|_{g_0} F_\A^2 N_\A\tilde{\otimes} N_\A  }
\\&\quad + \la\sum_{\A\neq\B,\pm}\sin\pth{\frac{u_\A\pm u_\B}{\la}} \pth{  \bar{\PP}^{[2]}_{u_\A\pm u_\B}\pth{\ka^{(1,\pm)}_{\A\B}} - \frac{N^{(\pm)}_{\A\B}\tilde{\otimes} \wc{X}^{(2,\pm)}_{\A\B}}{|\nab_{g_0}\pth{ u_\A\pm u_\B}|_{g_0}} } + \la^2 \tilde{k}^{cons}_\la,
\end{aligned}
\end{equation}
where $\wc{\ffi}^{(2,1)}_{\A}$, $\wc{\ffi}^{(2,\pm)}_{\A\B}$, $\wc{X}^{(2,1)}_\A$ and $\wc{X}^{(2,\pm)}_{\A\B}$ are defined by
\begin{align}
\wc{\ffi}^{(2,1)}_{\A} & \vcentcolon = \frac{1}{8|\nab_{g_0} u_\A|_{g_0}^2}(\bar{F}^{(1)}_\A)_{ij} \pth{- \dr^i \dr^j u_\A + |\nab_{g_0} u_\A|_{g_0}k_0^{ij} + \half  \dr_{\ell} u_\A  \dr^\ell g_0^{ij} },  \label{wcffi 21}
\\ \wc{\ffi}^{(2,\pm)}_{\A\B} & \vcentcolon =  \frac{ \left| \bar{F}^{(1)}_\A \cdot \bar{F}^{(1)}_\B \right|_{g_0}}{64 |\nab_{g_0}(u_\A\pm u_\B)|_{g_0}^2} \Big( 2 |\nab_{g_0} u_\A|_{g_0}^2 + 2 |\nab_{g_0} u_\B|_{g_0}^2 \label{wcffi 2pm}
\\&\hspace{5cm} \pm 3 |\nab_{g_0} u_\A\cdot \nab_{g_0} u_\B|_{g_0} \mp |\nab_{g_0} u_\A|_{g_0} |\nab_{g_0} u_\B|_{g_0} \Big)  \non
\\&\quad \mp \frac{ (\bar{F}^{(1)}_\A)_{i \nab_{g_0} u_\B} (\bar{F}^{(1)}_\B)_{\nab_{g_0} u_\A}^i}{32 |\nab_{g_0}(u_\A\pm u_\B)|_{g_0}^2} ,\nonumber
\end{align}
and
\begin{align}
(\wc{X}^{(2,1)}_\A)_i & \vcentcolon = \half  \dr^b \pth{|\nab_{g_0} u_\A|_{g_0} \bar{F}^{(1)}_\A}_{bi} +\frac{1}{4} |\nab_{g_0} u_\A|_{g_0} (\bar{F}^{(1)}_\A)_{bc} \dr_i g_0^{bc} \label{wcX 21}
\\&\quad +  \half |\nab_{g_0} u_\A|_{g_0} \pth{ \dr_a g_0^{ca} + \half g_0^{ab}g_0^{cd}\dr_d (g_0)_{ab} } (\bar{F}^{(1)}_\A)_{ic}   - \half \dr_i u_\A (\bar{F}^{(1)}_\A)_{bc} (k_0)^{bc}\nonumber,
\\ (\wc{X}^{(2,\pm)}_{\A\B})_i & \vcentcolon = - \frac{1}{8} |\nab_{g_0} u_\B|_{g_0} (\bar{F}^{(1)}_\A)^{b}_{\nab_{g_0} u_\B} (\bar{F}^{(1)}_\B)_{bi} - \frac{1}{8}  |\nab_{g_0} u_\A|_{g_0} (\bar{F}^{(1)}_\B)^{b}_{\nab_{g_0} u_\A} (\bar{F}^{(1)}_\A)_{bi} \non
\\&\quad + \bigg( \frac{1}{8} \pth{|\nab_{g_0} u_\A|_{g_0} \dr_i u_\A+ |\nab_{g_0} u_\B|_{g_0} \dr_i u_\B} \label{wcX 2pm}
\\&\hspace{2cm} +\frac{1}{16}  \pth{ \pm  |\nab_{g_0} u_\B|_{g_0}   \dr_i u_\A    \pm  |\nab_{g_0} u_\A|_{g_0}   \dr_i u_\B }\bigg) \left| \bar{F}^{(1)}_\A \cdot \bar{F}^{(1)}_\B \right|_{g_0}.\nonumber 
\end{align}
Moreover, the remainders $h_\la$ and $ \tilde{k}^{cons}_\la$ belong to the spaces $H^5_\de(\Si_0)$ and $H^4_{\de+1}(\Si_0)$ respectively and satisfy
\begin{align}
\max_{r\in\llbracket 0,4 \rrbracket} \la^r \l \nab^{r+1}h_\la \r_{L^2_{\de+r+1}(\Si_0)} & \leq C_{\mathrm{cons}} \e, \label{estim reste cons 1}
\\ \max_{r\in\llbracket 0,4 \rrbracket} \la^r \l \nab^{r}\tilde{k}^{cons}_\la \r_{L^2_{\de+r+1}(\Si_0)} & \leq C_{\mathrm{cons}} \e, \label{estim reste cons 2}
\end{align}
for some $C_{\mathrm{cons}}=C_{\mathrm{cons}}(N,\de,R)>0$. 
\end{prop}

The result of Proposition \ref{prop constraint main} will be used as a black box in the final construction of initial data in Proposition \ref{prop ID main}, and its proof, based on the conformal method, is postponed to Appendix \ref{appendix proof prop constraint main}. The freedom in the choice of the tensors $\ga^{(2,\pm)}_{\A\B}$, $\ka^{(1,1)}_\A$, $\ka^{(1,2)}_\A$ and $\ka^{(1,\pm)}_{\A\B}$ will be crucially used to match the solution of the constraint equations and the data induced by the spacetime metric. This is related to an inherent redundancy in the geometric optics approach for \ref{EVE}. Indeed, solving the constraint roughly provides the zeroth and first derivatives of the metric on $\Si_0$. However, our geometric optics construction implies that the first terms in the ansatz for $\g_\la$ actually solve first order transport equation, so that their first order derivatives are not free. This difficulty was already present in the singlephase case of \cite{Touati2023a,Touati2023b}, but the multiphase case brings a new one: the tensor $F^{(2,\pm)}_{\A\B}$ solves an algebraic equation and is defined by \eqref{def F2pm} so that even its zeroth derivatives is not free. Therefore, we need to make sure that the solution $(g_\la,k_\la)$ of the constraint equations matches the expression of the spacetime ansatz by choosing appropriate tensors $\ga^{(2,\pm)}_{\A\B}$, $\ka^{(1,1)}_\A$, $\ka^{(1,2)}_\A$ and $\ka^{(1,\pm)}_{\A\B}$.

\subsection{Initial data for the spacetime metric $\g_\la$}\label{section initial data final}

In this section, we define initial data for the system \eqref{eq F1}-\eqref{eq h}. The properties of these initial data are summarized in Proposition \ref{prop ID main} at the end of the section.

\subsubsection{Initial data for the metric and the induced metric}\label{section induced metric}

We start by defining initial data for the induced metric $\g_{\la|_{\Si_0}}$. In order to ensure that the induced metric on $\Si_0$ is given by $g_\la$ from Proposition \ref{prop constraint main}, we will make a particular choice of tensor $\ga^{(2,\pm)}_{\A\B}$.

\paragraph{Initial data for $F^{(1)}_\A$.} For $F^{(1)}_\A$ we define the initial data
\begin{align}
(F^{(1)}_\A)_{ij|_{\Si_0}} & \vcentcolon = (\bar{F}^{(1)}_\A)_{ij},\label{ID F1 ij}
\\ (F^{(1)}_\A)_{0\a|_{\Si_0}} & \vcentcolon = 0. \label{ID F1 0}
\end{align}
Thanks to \eqref{assum F1}, these initial data are such that the conditions \eqref{pola F1}-\eqref{pola F1 bis} and \eqref{energy F1} hold on $\Si_0$. Moreover, they allow us to obtain the initial value of derivatives of $F^{(1)}_\A$. While spacelike derivatives can be obtained by directly differentiating \eqref{ID F1 ij} or \eqref{ID F1 0}, time derivatives can be deduced from the transport equation we want $F^{(1)}_\A$ to solve in the spacetime, i.e \eqref{eq F1}. Indeed, an equation of the form $\Ll_\A T = S$ rewrites on $\Si_0$:
\begin{align}\label{dt via transport}
\dr_t T_{\mu\nu}  & = N_\A T_{\mu\nu} + (\dr_t - N_\A)^\a \Ga(\g_0)^\rho_{\a(\mu}T_{\nu)\rho} + \frac{1}{2|\nab_{g_0} u_\A|_{g_0}} (\Box_{\g_0}u_\A)T_{\mu\nu}  - \frac{1}{2|\nab_{g_0} u_\A|_{g_0}} S_{\mu\nu} ,
\end{align}
where we used $L_{\A|_{\Si_0}} = |\nab_{g_0} u_\A|_{g_0} (\dr_t - N_\A)$. In the case of $F^{(1)}_\A$, this implies in particular
\begin{align}
\dr_t (F^{(1)}_\A)_{00|_{\Si_0}} & = 0, \label{dt F1 00}
\\ \dr_t (F^{(1)}_\A)_{0i|_{\Si_0}} & =  (\bar{F}^{(1)}_\A)_{i k} g_0^{k\ell} \pth{  \dr_t (\g_0)_{0\ell} +  (k_0)_{N_\A\ell} } \label{dt F1 0i},
\\ \dr_t (F^{(1)}_\A)_{ij|_{\Si_0}} & =  N_\A (\bar{F}^{(1)}_\A)_{ij}  +   (\dr_t - N_\A)^\rho \Gamma(\g_0)^k_{\rho (i} (\bar{F}^{(1)}_\A)_{k j)}+ \frac{1}{2|\nabla_{g_0} u_\A|_{g_0}}(\Box_{\g_0}u_\A)  (\bar{F}^{(1)}_\A)_{ij},  \label{dt F1 ij} 
\end{align}
where we used \eqref{ID F1 ij}-\eqref{ID F1 0}. Using again \eqref{assum F1}, \eqref{dt F1 ij} implies
\begin{align}
g_0^{ij}\dr_t (F^{(1)}_\A)_{ij|_{\Si_0}} & = -2 (\bar{F}^{(1)}_\A)_{ij}  k_0^{ij}, \label{trace dt F1}
\\ N_\A^i N_\A^j  \dr_t (F^{(1)}_\A)_{ij|_{\Si_0}} & = 0,\label{NN dt F1}
\\ N_\A^i (e^{(\i)}_\A)^j \dr_t(F^{(1)}_\A)_{ij|_{\Si_0}} & =\pth{  -  N_\A (N_\A)_\ell  -   (k_0)_{N_\A\ell}    + \half  N_\A^i  N_\A^j  \dr_\ell (g_0)_{j i} } g_0^{k\ell}(\bar{F}^{(1)}_\A)_{k e^{(\i)}_\A}.  \label{Ne dt F1}
\end{align}
Together with the definition \eqref{Q 0}, the identities \eqref{dt F1 00}, \eqref{dt F1 0i} and \eqref{trace dt F1} allow us to compute $Q^{(0)}_{\A|_{\Si_0}}$:
\begin{align}
(Q^{(0)}_{\A})_{\ell |_{\Si_0}} & = -  (\bar{F}^{(1)}_\A)_{\ell k} g_0^{ki} \pth{  \dr_t (\g_0)_{0i} +  (k_0)_{N_\A i} }+ g_0^{ij}\dr_i (\bar{F}^{(1)}_\A)_{\ell j} + \half (\bar{F}^{(1)}_\A)_{ij}   \dr_\ell g_0^{ij} ,\label{Q 0 ell}
\\ (Q^{(0)}_{\A})_{0|_{\Si_0}} & = (\bar{F}^{(1)}_\A)_{ij}  k_0^{ij} .\label{Q 0 0}
\end{align}

\paragraph{Initial data for $\F_\A$, $F^{(2,1)}_\A$ and $F^{(2,2)}_\A$.} For $\F_\A$, we simply define
\begin{align}
\F_{\A |_{\Si_0}} & \vcentcolon= 0,\label{ID F}
\end{align}
which clearly implies that \eqref{pola F} holds on $\Si_0$. For $F^{(2,1)}_\A$ we define
\begin{align}
(F^{(2,1)}_\A)_{ij|_{\Si_0}} & \vcentcolon =    4\wc{\ffi}^{(2,1)}_{\A} (g_0)_{ij},\label{ID F21 ij}
\\ (F^{(2,1)}_\A)_{00|_{\Si_0}} & \vcentcolon= 0,\label{ID F21 00}
\\ (F^{(2,1)}_\A)_{0N_\A|_{\Si_0}} & \vcentcolon=  2\wc{\ffi}^{(2,1)}_{\A} - \frac{1}{2|\nab_{g_0} u_\A|_{g_0}^2}(Q^{(0)}_\A)_{{L_\A}|_{\Si_0}},\label{ID F21 0N}
\\ (F^{(2,1)}_\A)_{0 e^{(\i)}_\A|_{\Si_0}} & \vcentcolon=  \frac{1}{|\nab_{g_0} u_\A|_{g_0}} (Q^{(0)}_\A)_{{e^{(\i)}_\A}|_{\Si_0}},\label{ID F21 0e}
\end{align}
where we used \eqref{Q 0 ell}-\eqref{Q 0 0}, and for $F^{(2,2)}_\A$ we define
\begin{align}
 (F^{(2,2)}_\A)_{ij|_{\Si_0}} & \vcentcolon = \frac{3}{4} F_\A^2  (g_0)_{ij},\label{ID F22 ij}
\\ (F^{(2,2)}_\A)_{00|_{\Si_0}} & \vcentcolon= 0,\label{ID F22 00}
\\ (F^{(2,2)}_\A)_{0N_\A|_{\Si_0}} & \vcentcolon=  \frac{3}{8} F_\A^2, \label{ID F22 0N}
\\ (F^{(2,2)}_\A)_{0 e^{(\i)}_\A|_{\Si_0}} & \vcentcolon= 0. \label{ID F22 0e}
\end{align}

\begin{lemma}\label{lem pola F2 init}
The initial data \eqref{ID F21 ij}-\eqref{ID F22 0e} are such that \eqref{pola F21}-\eqref{pola F22} hold on $\Si_0$.
\end{lemma}

\begin{proof}
We recall that $L_{\A|_{\Si_0}} = |\nab_{g_0} u_\A|_{g_0} (\dr_t - N_\A)$ and $\Lb_{\A|_{\Si_0}} = |\nab_{g_0} u_\A|_{g_0} (\dr_t + N_\A)$. We start with \eqref{pola F21}. From \eqref{Pol L} and \eqref{def 1 V 2 2} we have
\begin{align*}
(V^{(2,1)}_\A)_{L_\A|_{\Si_0}} & = \Pol{F^{(2,1)}_\A}{u_\A}_{L_\A|_{\Si_0}} + (Q^{(0)}_{\A})_{L_\A |_{\Si_0}}
\\& = - (F^{(2,1)}_\A)_{L_\A L_\A|_{\Si_0}} + (Q^{(0)}_{\A})_{L_\A |_{\Si_0}}
\\& =  |\nab_{g_0} u_\A|_{g_0}^2\pth{ -  (F^{(2,1)}_\A)_{0 0|_{\Si_0}} +  2 (F^{(2,1)}_\A)_{0 N_\A|_{\Si_0}}  -  (F^{(2,1)}_\A)_{N_\A N_\A|_{\Si_0}} }  + (Q^{(0)}_{\A})_{L_\A |_{\Si_0}}
\\& = |\nab_{g_0} u_\A|_{g_0}^2\pth{  4\wc{\ffi}^{(2,1)}_{\A} - \frac{1}{|\nab u_\A|^2}(Q^{(0)}_\A)_{{L_\A}|_{\Si_0}} -  (F^{(2,1)}_\A)_{N_\A N_\A|_{\Si_0}} } + (Q^{(0)}_{\A})_{L_\A |_{\Si_0}}
\\& = 0,
\end{align*}
and from \eqref{Pol e} we also have
\begin{align*}
(V^{(2,1)}_\A)_{e^{(\i)}_\A|_{\Si_0}} & = \Pol{F^{(2,1)}_\A}{u_\A}_{e^{(\i)}_\A|_{\Si_0}} + (Q^{(0)}_{\A})_{e^{(\i)}_\A |_{\Si_0}}
\\& = - (F^{(2,1)}_\A)_{e^{(\i)}_\A L_\A|_{\Si_0}} + (Q^{(0)}_{\A})_{e^{(\i)}_\A |_{\Si_0}}
\\& = -|\nab_{g_0} u_\A|_{g_0} \pth{  (F^{(2,1)}_\A)_{e^{(\i)}_\A 0|_{\Si_0}} - (F^{(2,1)}_\A)_{e^{(\i)}_\A N_\A|_{\Si_0}} } + (Q^{(0)}_{\A})_{e^{(\i)}_\A |_{\Si_0}}
\\& = -|\nab_{g_0} u_\A|_{g_0} \pth{ \frac{1}{|\nab_{g_0} u_\A|_{g_0}} (Q^{(0)}_\A)_{{e^{(\i)}_\A}|_{\Si_0}} - (F^{(2,1)}_\A)_{e^{(\i)}_\A N_\A|_{\Si_0}} } + (Q^{(0)}_{\A})_{e^{(\i)}_\A |_{\Si_0}}
\\& = 0,
\end{align*}
where we used \eqref{ID F21 ij} and \eqref{ID F21 00}-\eqref{ID F21 0e}. From \eqref{Pol Lb} we also have
\begin{align*}
(V^{(2,1)}_\A)_{\Lb_\A|_{\Si_0}} & = \Pol{F^{(2,1)}_\A}{u_\A}_{\Lb_\A|_{\Si_0}} + (Q^{(0)}_{\A})_{\Lb_\A |_{\Si_0}}
\\& =   |\nab_{g_0} u_\A|_{g_0}^2\pth{ (F^{(2,1)}_\A)_{ N_\A N_\A|_{\Si_0}} - \tr_{g_0} F^{(2,1)}_{\A\;\;|_{\Si_0}}}  
 + |\nab_{g_0} u_\A|_{g_0} \pth{  (Q^{(0)}_{\A})_{0 |_{\Si_0}} + N_\A^\ell  (Q^{(0)}_{\A})_{\ell |_{\Si_0}} }.
\end{align*}
Using \eqref{ID F21 ij} together with \eqref{image P1} on one hand and \eqref{Q 0 ell}-\eqref{Q 0 0} on the other hand we obtain
\begin{align*}
(V^{(2,1)}_\A)_{\Lb_\A|_{\Si_0}} & =  -8 |\nab_{g_0} u_\A|_{g_0}^2 \wc{\ffi}^{(2,1)}_{\A} 
 + |\nab_{g_0} u_\A|_{g_0} \pth{ (\bar{F}^{(1)}_\A)_{ij}  k_0^{ij} - \frac{(\bar{F}^{(1)}_\A)_{\ell j}}{|\nab_{g_0} u_\A|_{g_0}} g_0^{ij}\dr_i \dr^\ell u_\A + \half (\bar{F}^{(1)}_\A)_{ij} N_\A g_0^{ij} } 
\\& = 0,
\end{align*}
where we used \eqref{wcffi 21}. We now turn to \eqref{pola F22}. From \eqref{def 1 V 2 2} we have 
\begin{align*}
(V^{(2,2)}_\A)_{L_\A|_{\Si_0}} & =  \Pol{F^{(2,2)}_\A}{u_\A}_{L_\A|_{\Si_0}}
\\& = - (F^{(2,2)}_\A)_{L_\A L_\A|_{\Si_0}}
\\& = |\nab_{g_0} u_\A|_{g_0}^2\pth{ -  (F^{(2,2)}_\A)_{0 0|_{\Si_0}} +  2 (F^{(2,2)}_\A)_{0 N_\A|_{\Si_0}}  -  (F^{(2,2)}_\A)_{N_\A N_\A|_{\Si_0}} }
\\& =  |\nab_{g_0} u_\A|_{g_0}^2\pth{ \frac{3}{4} F_\A^2  -  \frac{3}{4} F_\A^2  }
\\& = 0,
\end{align*}
and
\begin{align*}
(V^{(2,2)}_\A)_{e^{(\i)}_\A|_{\Si_0}} & =  \Pol{F^{(2,2)}_\A}{u_\A}_{e^{(\i)}_\A|_{\Si_0}}
\\& = - (F^{(2,2)}_\A)_{e^{(\i)}_\A L_\A|_{\Si_0}}
\\& = - |\nab_{g_0} u_\A|_{g_0} \pth{ (F^{(2,2)}_\A)_{e^{(\i)}_\A 0|_{\Si_0}} - (F^{(2,2)}_\A)_{e^{(\i)}_\A N_\A|_{\Si_0}} }
\\& = 0,
\end{align*}
where we used \eqref{ID F22 ij} and \eqref{ID F22 00}-\eqref{ID F22 0e}. Using now only \eqref{ID F22 ij} we finally obtain
\begin{align*}
(V^{(2,2)}_\A)_{\Lb_\A|_{\Si_0}} & = \Pol{F^{(2,2)}_\A}{u_\A}_{\Lb_\A|_{\Si_0}} + \frac{3}{2}F_\A^2 |\nab u_\A|^2
\\& =   |\nab_{g_0} u_\A|_{g_0}^2\pth{ (F^{(2,2)}_\A)_{ N_\A N_\A|_{\Si_0}} - \tr_{g_0} F^{(2,2)}_{\A\;\;|_{\Si_0}}} + \frac{3}{2}F_\A^2 |\nab_{g_0} u_\A|_{g_0}^2
\\& = - \frac{3}{2}F_\A^2 |\nab_{g_0} u_\A|_{g_0}^2 + \frac{3}{2}F_\A^2 |\nab_{g_0} u_\A|_{g_0}^2
\\&=0,
\end{align*}
which concludes the proof of the lemma.
\end{proof}

\paragraph{Initial data for $\h_\la$.} In order to define initial data for $\h_\la$, we first need to obtain the values of $(\gg^{(3,h)}_{u,T})_{\a\b|_{\Si_0}}$ and $(\gg^{(3,e)}_{v,T})_{\a\b|_{\Si_0}}$. For $(\gg^{(3,h)}_{u,T})_{\a\b|_{\Si_0}}$, we use \eqref{g3h}, and note that the initial values of $F^{(1)}_\A$, $\F_\A$, $F^{(2,1)}_\A$ and $F^{(2,2)}_\A$ have been defined in \eqref{ID F1 ij}-\eqref{ID F1 0} and \eqref{ID F21 ij}-\eqref{ID F22 0e}, the initial values of $\nab F^{(1)}_\A$ and $\dr_t F^{(1)}_\A$ can be deduced from \eqref{ID F1 ij}-\eqref{ID F1 0} and \eqref{dt F1 00}-\eqref{dt F1 ij} respectively, the initial values of $\nab F^{(2,1)}_\A $ and $\nab F^{(2,2)}_\A$ can be deduced from \eqref{ID F21 ij}-\eqref{ID F22 0e} and for the time derivatives $\dr_t F^{(2,1)}_\A $ and $\dr_t F^{(2,2)}_\A$ we use the transport equations \eqref{eq F21}-\eqref{eq F22} we want $F^{(2,1)}_\A$ and $F^{(2,2)}_\A$ to solve in the spacetime. Thanks to \eqref{dt via transport} we obtain
\begin{align*}
\dr_t (F^{(2,1)}_\A)_{\a\b} & = \cro{ \nab^{\leq 1} F^{(2,1)} + \dr^{\leq 2} F^{(1)} + (F^{(1)})^2 + (F^{(1)})^3}_{\a\b},
\\ \dr_t (F^{(2,2)}_\A)_{\a\b} & = \cro{ \nab^{\leq 1} F^{(2,2)} +  \pth{ F^{(2,1)} + \dr^{\leq 1} F^{(1)} }F^{(1)} }_{\a\b},
\end{align*}
where we also used \eqref{Rtilde 1 1}-\eqref{Rtilde 1 2}. This allows us to completely define the initial value of $(\gg^{(3,h)}_{u,T})_{\a\b|_{\Si_0}}$. For $(\gg^{(3,e)}_{v,T})_{\a\b|_{\Si_0}}$, we use \eqref{def g3e} and \eqref{R 1 e} to obtain
\begin{align*}
\gg^{(3,e)}_{v,T} & =  \Big\{ \pth{ \F + F^{(2,1)} + F^{(2,2)}  + (F^{(1)})^2 + \dr^{\leq 1} F^{(1)} }\pth{1+F^{(1)}} \Big\}.
\end{align*}
Therefore, using the various definitions of initial data above, this allows us to completely define the initial value of $(\gg^{(3,e)}_{v,T})_{\a\b|_{\Si_0}}$. We can now define the initial data for $\h_\la$:
\begin{align}
(\h_\la)_{ij|_{\Si_0}} & \vcentcolon= (h_\la)_{ij} - \la \sum_{u\in\mathcal{N},T\in\{\cos,\sin\}}T\pth{\frac{u}{\la}} (\gg^{(3,h)}_{u,T})_{ij|_{\Si_0}}   -\la \sum_{v\in\mathcal{I},T\in\{\cos,\sin\}} T\pth{\frac{v}{\la}} (\gg^{(3,e)}_{v,T})_{ij|_{\Si_0}},  \label{ID h ij}
\\ (\h_\la)_{0\a|_{\Si_0}} & \vcentcolon= - \la \sum_{u\in\mathcal{N},T\in\{\cos,\sin\}}T\pth{\frac{u}{\la}} (\gg^{(3,h)}_{u,T})_{0\a|_{\Si_0}} -\la \sum_{v\in\mathcal{I},T\in\{\cos,\sin\}} T\pth{\frac{v}{\la}} (\gg^{(3,e)}_{v,T})_{0\a|_{\Si_0}}.   \label{ID h 0}
\end{align}

\paragraph{The induced metric.} Thanks to \eqref{ID F1 ij}, \eqref{ID F}, \eqref{ID F21 ij}, \eqref{ID F22 ij} and \eqref{ID h ij} we can compute the induced metric on $\Si_0$ by $\g_\la$:
\begin{align*}
(\g_\lambda)_{ij|_{\Si_0}} & = (g_0)_{ij} + \la \sum_\A  \cos\left( \frac{u_\A}{\lambda} \right) (\bar{F}^{(1)}_\A)_{ij} + \lambda^2\sum_\A  \sin\left(\frac{u_\A}{\lambda} \right) 4\wc{\ffi}^{(2,1)}_{\A} (g_0)_{ij} 
\\&\quad + \lambda^2\sum_\A  \cos\left( \frac{2u_\A}{\lambda}\right)  \frac{3}{4} F_\A^2  (g_0)_{ij} + \la^2 \sum_{\A\neq \B,\pm} \cos\left(\frac{u_\A\pm u_\B}{\lambda}\right) (F^{(2,\pm)}_{\A\B})_{ij|_{\Si_0}}  + \lambda^2 (h_\lambda)_{ij} .
\end{align*}
Therefore, $g_\la$ given by \eqref{g la} is the induced metric if and only if the following holds
\begin{align}\label{eq ga2pm 1}
(F^{(2,\pm)}_{\A\B})_{ij|_{\Si_0}} & = \bar{\PP}^{[1]}_{u_\A\pm u_\B}\pth{\ga^{(2,\pm)}_{\A\B}}_{ij}  + 4\wc{\ffi}^{(2,\pm)}_{\A\B}(g_0)_{ij}.
\end{align}
Since $(F^{(2,\pm)}_{\A\B})_{ij|_{\Si_0}}$ can be computed from \eqref{def F2pm} and \eqref{ID F1 ij}-\eqref{ID F1 0} and $\wc{\ffi}^{(2,\pm)}_{\A\B}$ from \eqref{wcffi 2pm}, \eqref{eq ga2pm 1} is an equation for $\ga^{(2,\pm)}_{\A\B}$. The following lemma shows how one can solve it.

\begin{lemma}\label{lem ga2pm}
If we define
\begin{align}
\ga^{(2,\pm)}_{\A\B} & \vcentcolon= (F^{(2,\pm)}_{\A\B})_{ij|_{\Si_0}}  - 4\wc{\ffi}^{(2,\pm)}_{\A\B}(g_0)_{ij},\label{def ga2pm}
\end{align}
then \eqref{eq ga2pm 1} holds.
\end{lemma}

\begin{proof}
If $\ga^{(2,\pm)}_{\A\B}$ is defined by \eqref{def ga2pm}, then \eqref{eq ga2pm 1} rewrites $\bar{\PP}^{[1]}_{u_\A\pm u_\B}\pth{\ga^{(2,\pm)}_{\A\B}}=\ga^{(2,\pm)}_{\A\B}$. Therefore, thanks to Lemma \ref{lem P 1et2}, proving the lemma is equivalent to
\begin{align}\label{condition ga2pm}
(\ga^{(2,\pm)}_{\A\B})_{N^{(\pm)}_{\A\B}N^{(\pm)}_{\A\B}} - \tr_{g_0}\ga^{(2,\pm)}_{\A\B} & = 0,
\end{align}
with $\ga^{(2,\pm)}_{\A\B}$ defined by \eqref{def ga2pm}. From \eqref{def ga2pm} we have
\begin{align*}
(\ga^{(2,\pm)}_{\A\B})_{N^{(\pm)}_{\A\B}N^{(\pm)}_{\A\B}} - \tr_{g_0}\ga^{(2,\pm)}_{\A\B} & = \mp\frac{(\mathcal{I}^{(0,\pm)}_{\A\B})_{N^{(\pm)}_{\A\B}N^{(\pm)}_{\A\B}} - \tr_{g_0}\mathcal{I}^{(0,\pm)}_{\A\B} }{2\pth{|\nab_{g_0} u_\A\cdot \nab_{g_0} u_\B|_{g_0} - |\nab_{g_0} u_\A|_{g_0}|\nab_{g_0} u_\B|_{g_0}}} + 8\wc{\ffi}^{(2,\pm)}_{\A\B}.
\end{align*}
From \eqref{I 0 pm}, \eqref{P 0 pm} and \eqref{ID F1 ij}-\eqref{ID F1 0} we find on $\Si_0$
\begin{align*}
&(\mathcal{I}^{(0,\pm)}_{\A\B})_{N^{(\pm)}_{\A\B}N^{(\pm)}_{\A\B}} - \tr_{g_0}\mathcal{I}^{(0,\pm)}_{\A\B} 
\\& =   - \frac{|\nab_{g_0} u_\A\cdot \nab_{g_0} u_\B|_{g_0} - |\nab_{g_0} u_\A|_{g_0}|\nab_{g_0} u_\B|_{g_0}}{2|\nab_{g_0} (u_\A\pm u_\B)|_{g_0}^2}  (\bar{F}^{(1)}_\B)_{\nab_{g_0} u_\A}^k (\bar{F}^{(1)}_\A)_{\nab_{g_0} u_\B k}
\\&\quad \pm  \frac{\left|\bar{F}^{(1)}_\A\cdot \bar{F}^{(1)}_\B\right|_{g_0}}{4|\nab_{g_0}(u_\A\pm u_\B)|_{g_0}^2} \Big( 2|\nab_{g_0} u_\A|_{g_0}^2 |\nab_{g_0} u_\A\cdot \nab_{g_0} u_\B|_{g_0} \pm |\nab_{g_0} u_\A|_{g_0}^2  | \nab_{g_0} u_\B|_{g_0}^2 
\\&\hspace{5cm} \pm3 |\nab_{g_0} u_\A\cdot \nab_{g_0} u_\B|_{g_0}^2  + 2 |\nab_{g_0} u_\A\cdot \nab_{g_0} u_\B|_{g_0} | \nab_{g_0} u_\B|_{g_0}^2
\\&\hspace{5cm}     - 2  |\nab_{g_0} u_\A|_{g_0}^3|\nab_{g_0} u_\B|_{g_0}    - 2   |\nab_{g_0} u_\A|_{g_0}|\nab_{g_0} u_\B|_{g_0}^3 
\\&\hspace{7cm}  \mp4 |\nab_{g_0} u_\A\cdot \nab_{g_0} u_\B|_{g_0}  |\nab_{g_0} u_\A|_{g_0}|\nab_{g_0} u_\B|_{g_0}     \Big).
\end{align*}
Using now \eqref{wcffi 2pm} we finally obtain \eqref{condition ga2pm} and thus the lemma.
\end{proof}

\subsubsection{Initial data for derivatives of the metric and the second fundamental form}\label{section second fundamental form}

In terms of initial data for the system \eqref{eq F1}-\eqref{eq h} it only remains to define the initial data for $\T_\la (\h_\la)_{\a\b}$, where $\T_\la$ is the future-directed unit normal to $\Si_0$ for $\g_\la$. As usual when solving the Einstein vacuum equations in wave coordinates, $\T_\la (\h_\la)_{ij}$ will set the second fundamental form and $\T_\la (\h_\la)_{0\a}$ will ensure that \eqref{GWC} holds on $\Si_0$. In order to ensure that the second fundamental form is given by $k_\la$ from Proposition \ref{prop constraint main}, we will also make a particular choice of $\ka^{(1,1)}_\A$, $\ka^{(1,2)}_\A$ and $\ka^{(1,\pm)}_{\A\B}$. We first need to compute $\T_\la$. 

\paragraph{The future-directed unit normal and the second fundamental form.}
Using the fact that $\dr_t$ is the future-directed unit normal to $\Si_0$ for $\g_0$ and using also \eqref{ID F1 0}, \eqref{ID F} and \eqref{ID h 0} we obtain
\begin{align}
\T_\la & = \dr_t + \la^2 \sum_\A\pth{ \sin\pth{\frac{u_\A}{\la}}\T^{(2,1)}_\A + \cos\pth{\frac{2u_\A}{\la}}\T^{(2,2)}_\A } \label{unit normal}
\\&\quad + \la^2\sum_{\A\neq\B,\pm}\cos\pth{\frac{u_\A\pm u_\B}{\la}}\T^{(2,\pm)}_{\A\B} + \la^3 \mathfrak{T}_\la, \non
\end{align}
where the vector fields $\T^{(2,1)}_\A$, $\T^{(2,2)}_\A$ and $\T^{(2,\pm)}_{\A\B}$ are tangent to $\Si_0$ and satisfy
\begin{align}
g_0\pth{\T^{(2,1)}_\A,\dr_i} & = - (F^{(2,1)}_\A)_{0i}, \label{T21}
\\ g_0\pth{\T^{(2,2)}_\A,\dr_i} & = - (F^{(2,2)}_\A)_{0i}, \label{T22}
\\ g_0\pth{\T^{(2,\pm)}_{\A\B},\dr_i} & = - (F^{(2,\pm)}_{\A\B})_{0i},\label{T2pm}
\end{align}
and where the vector field $\mathfrak{T}_\la=\mathfrak{T}_\la^\a \dr_\a$ satisfies
\begin{align}
\mathfrak{T}_\la^\a = \cro{\pth{ F^{(2,1)} + F^{(2,2)} + (F^{(1)})^2 }\pth{ (\g_\la^{-1})^{(\geq 1)}  + \g_\la^{-1}\pth{ F^{(2,1)} + F^{(2,2)} + (F^{(1)})^2 }} }^\osc.\label{mathfrak T}
\end{align}
We can now compute the second fundamental form of $\Si_0$ with respect to the spacetime metric $\g_\la$, using in particular \eqref{eikonal init}, \eqref{ID F1 ij}, \eqref{ID F21 ij}, \eqref{ID F22 ij} and \eqref{T21}-\eqref{mathfrak T}:
\begin{equation}\label{second fundamental form}
\begin{aligned}
-\half (\LL_{\T_\la}\g_\la)_{ij} & =  (k_0)_{ij} + \half \sum_\A \sin\pth{\frac{u_\A}{\la}}|\nab_{g_0} u_\A|_{g_0} (\bar{F}^{(1)}_\A)_{ij}
\\&\quad - \frac{\la}{2} \sum_\A \cos\pth{\frac{u_\A}{\la}}\Big( \dr_t (F^{(1)}_\A)_{ij} + 4 |\nab_{g_0} u_\A|_{g_0} \wc{\ffi}^{(2,1)}_{\A} (g_0)_{ij} - \dr_{(i}u_\A(F^{(2,1)}_\A)_{0j)}\Big)
\\&\quad + \la \sum_\A \sin\pth{\frac{2u_\A}{\la}} \pth{ \frac{3}{4}|\nab_{g_0} u_\A|_{g_0} F_\A^2  (g_0)_{ij} - \dr_{(i}u_\A(F^{(2,2)}_\A)_{0j)}}
\\&\quad + \frac{\la}{2} \sum_{\A\neq \B,\pm} \sin\pth{\frac{u_\A\pm u_\B}{\la}}\bigg( \pth{ |\nab_{g_0} u_\A|_{g_0} \pm |\nab_{g_0} u_\B|_{g_0}} (F^{(2,\pm)}_{\A\B})_{ij}  
\\&\hspace{7cm} - \dr_{(i}(u_\A\pm u_\B)(F^{(2,\pm)}_{\A\B})_{0j)} \bigg)
\\&\quad - \frac{\la^2}{2} \T_\la (\h_\la)_{ij} + \la^2 \tilde{k}^{evol}_\la,
\end{aligned}
\end{equation}
where the remainder is of the form 
\begin{align}
\tilde{k}^{evol}_\la & = \cro{\pth{ \g_\la \dr \T_\la + \T_\la \dr \pth{\g_\la - \la^2 \h_\la} }^{(\geq 2)}}^\osc,\label{ktilde evol}
\end{align}
where by $\T_\la$ we denote a component of the vector field $\T_\la$ in coordinates. Note that the only terms involved in \eqref{ktilde evol} that have not yet been defined are time derivatives of $\F$. They can be deduced from the transport equation \eqref{eq F} we want $\F_\A$ to satisfy in the spacetime. Thanks to \eqref{dt via transport} and \eqref{ID F}, \eqref{eq F} becomes on $\Si_0$
\begin{align*}
\dr_t (\F_\A)_{\mu\nu|_{\Si_0}}  & =  \frac{1}{2|\nab_{g_0} u_\A|_{g_0}} \Pi_\leq \pth{ (\h_\lambda)_{L_\A L_\A|_{\Si_0}} } (F^{(1)}_\A)_{\mu\nu|_{\Si_0}} .
\end{align*}
Therefore, $k_\la$ given by \eqref{k la} is the second fundamental form if and only if the expressions \eqref{k la} and \eqref{second fundamental form} coincide, that is if and only if the following four identities hold:
\begin{align}
 \bar{\PP}^{[2]}_{u_\A}\pth{\ka^{(1,1)}_\A}_{ij} &+  \frac{(N_\A\tilde{\otimes} \wc{X}^{(2,1)}_\A)_{ij}}{|\nab_{g_0} u_\A|_{g_0}} \label{eq ka11} 
\\& = -\half \pth{ 4|\nab_{g_0} u_\A|_{g_0} \wc{\ffi}^{(2,1)}_{\A} (g_0)_{ij}  - \dr_{(i} u_\A (F^{(2,1)}_\A)_{0j)} + \dr_t(F^{(1)}_\A)_{ij|_{\Si_0}} }, \non
\end{align}
\begin{align}
\bar{\PP}^{[2]}_{u_\A}\pth{\ka^{(1,2)}_\A}_{ij}   -\frac{3}{2} |\nab_{g_0} u_\A|_{g_0} F_\A^2 (N_\A\tilde{\otimes} N_\A)_{ij} = \frac{3}{4} |\nab_{g_0} u_\A|_{g_0} F_\A^2  (g_0)_{ij} - \dr_{(i} u_\A (F^{(2,2)}_\A)_{0j)} ,  \label{eq ka12} 
\end{align}
\begin{align}
 \bar{\PP}^{[2]}_{u_\A\pm u_\B}&\pth{\ka^{(1,\pm)}_{\A\B}}_{ij} - \frac{(N^{(\pm)}_{\A\B}\tilde{\otimes} \wc{X}^{(2,\pm)}_{\A\B})_{ij}}{|\nab_{g_0}\pth{ u_\A\pm u_\B}|_{g_0}} \label{eq ka1pm}
 \\&  = \half \pth{ (|\nab_{g_0} u_\A|_{g_0}\pm|\nab_{g_0} u_\B|_{g_0}) (F^{(2,\pm)}_{\A\B})_{ij|_{\Si_0}}  - \dr_{(i} (u_\A\pm u_\B)  (F^{(2,\pm)}_{\A\B})_{0j)} }  \non,
\end{align}
and
\begin{align} \label{ID dth ij first}
 -  \frac{1}{2} (\T_\la(\h_\lambda)_{ij})_{|_{\Si_0}} + (\tilde{k}^{evol}_\la)_{ij} & = (\tilde{k}^{cons}_\la)_{ij}.
\end{align}

\paragraph{Adjusting the second fundamental form.}
First, note that \eqref{ID dth ij first} defines the initial data for $\T_\la(\h_\lambda)_{ij}$:
\begin{align}\label{ID dth ij}
(\T_\la(\h_\lambda)_{ij})_{|_{\Si_0}} & \vcentcolon = - 2(\tilde{k}^{cons}_\la)_{ij} + 2(\tilde{k}^{evol}_\la)_{ij} .
\end{align}
Second, all the terms in \eqref{eq ka11} can be computed from \eqref{wcffi 21}, \eqref{wcX 21}, \eqref{dt F1 ij}, \eqref{ID F21 0N} and \eqref{ID F21 0e}. All the terms in \eqref{eq ka12} can be computed from \eqref{ID F21 0N} and \eqref{ID F21 0e}. All the terms in \eqref{eq ka1pm} can be computed from \eqref{def F2pm}, \eqref{ID F1 ij}, \eqref{ID F1 0} and \eqref{wcX 2pm}. Therefore, \eqref{eq ka11}-\eqref{eq ka1pm} are equations for $\ka^{(1,1)}_\A$, $\ka^{(1,2)}_\A$ and $\ka^{(1,\pm)}_{\A\B}$.

\begin{lemma}\label{lem ka11}
If we define 
\begin{align}
(\ka^{(1,1)}_\A)_{ij} & \vcentcolon= - \frac{(N_\A\tilde{\otimes} \wc{X}^{(2,1)}_\A)_{ij}}{|\nab_{g_0} u_\A|_{g_0}} - \half \dr_t(F^{(1)}_\A)_{ij|_{\Si_0}} \label{def ka11}
- \frac{|\nab_{g_0} u_\A|_{g_0}}{2}\pth{ 4\wc{\ffi}^{(2,1)}_{\A} (g_0)_{ij}  - (N_\A)_{(i} (F^{(2,1)}_\A)_{0j)}  },
\end{align}
then \eqref{eq ka11} holds.
\end{lemma}

\begin{proof}
If $\ka^{(1,1)}_\A$ is defined by \eqref{def ka11}, then \eqref{eq ka11} rewrites $\bar{\PP}^{[2]}_{u_\A}\pth{\ka^{(1,1)}_\A}=\ka^{(1,1)}_\A$. Therefore, thanks to Lemma \ref{lem P 1et2}, proving the lemma is equivalent to 
\begin{align*}
(\ka^{(1,1)}_\A)_{N_\A N_\A}-\tr_{g_0}\ka^{(1,1)}_\A & = 0,
\\ (\ka^{(1,1)}_\A)_{N_\A e^{(\i)}_\A} & = 0,
\end{align*}
with $\ka^{(1,1)}_\A$ defined by \eqref{def ka11}. Thanks to \eqref{wcffi 21}, \eqref{wcX 21}, \eqref{trace dt F1} and \eqref{NN dt F1} we first have 
\begin{align*}
(\ka^{(1,1)}_\A)_{N_\A N_\A}-\tr_{g_0}\ka^{(1,1)}_\A & =  - \frac{1}{|\nab_{g_0} u_\A|_{g_0}}(\wc{X}^{(2,1)}_\A)_{N_\A} + \half g_0^{ij} \dr_t(F^{(1)}_\A)_{ij|_{\Si_0}}  + 4|\nab_{g_0} u_\A|_{g_0}\wc{\ffi}^{(2,1)}_{\A}  
\\& =  \frac{1}{2|\nab_{g_0} u_\A|_{g_0}}(\bar{F}^{(1)}_\A)_{ij}\dr^i\dr^j u_\A   - \frac{1}{4}  (\bar{F}^{(1)}_\A)_{ij} N_\A g_0^{ij}   - \half  (\bar{F}^{(1)}_\A)_{ij} (k_0)^{ij}
\\&\quad + \frac{1}{2|\nab_{g_0} u_\A|_{g_0}}(\bar{F}^{(1)}_\A)_{ij} \pth{- \dr^i \dr^j u_\A + |\nab_{g_0} u_\A|_{g_0}k_0^{ij} + \half  \dr_{\ell} u_\A  \dr^\ell g_0^{ij} }
\\& = 0.
\end{align*}
Using now \eqref{wcffi 21}, \eqref{Ne dt F1}, \eqref{ID F21 0e} and \eqref{Q 0 ell} we have
\begin{align*}
(\ka^{(1,1)}_\A)_{N_\A e^{(\i)}_\A} & =  - \frac{1}{|\nab_{g_0} u_\A|_{g_0}}(\wc{X}^{(2,1)}_\A)_{e}   - \half N_\A^i (e^{(\i)}_\A)^j \dr_t(F^{(1)}_\A)_{ij}  + \frac{|\nab_{g_0} u_\A|_{g_0}}{2}  (F^{(2,1)}_\A)_{0e^{(\i)}_\A} 
\\& =   - \half g_0^{k\ell}(\bar{F}^{(1)}_\A)_{k e^{(\i)}_\A}  \bigg( - \dr^b (g_0)_{\ell b} + \half g_0^{ab}\dr_\ell (g_0)_{ab}+\dr_t (\g_0)_{0\ell} 
\\&\hspace{4cm} -  N_\A (N_\A)_\ell    + \half  N_\A^i  N_\A^j  \dr_\ell (g_0)_{j i} + \frac{1}{|\nab_{g_0} u_\A|_{g_0}} \dr_\ell |\nab_{g_0} u_\A|_{g_0}   \bigg) .
\end{align*}
We then use \eqref{dt g 0i} and \eqref{geodesic spatial} to obtain 
\begin{align*}
(\ka^{(1,1)}_\A)_{N_\A e^{(\i)}_\A} & =  -  \frac{1}{2|\nab_{g_0} u_\A|_{g_0}} (\bar{F}^{(1)}_\A)_{N_\A e^{(\i)}_\A}  N_\A |\nab_{g_0} u_\A|_{g_0}  
\\& = 0. 
\end{align*}
This concludes the proof of the lemma.
\end{proof}

\begin{lemma}\label{lem ka12}
If we define 
\begin{align}
(\ka^{(1,2)}_\A)_{ij} & \vcentcolon= \frac{3}{2}|\nab_{g_0} u_\A|_{g_0}F_\A^2(N_\A)_{i}(N_\A)_{j} \label{def ka12}
\end{align}
then \eqref{eq ka12} holds.
\end{lemma}

\begin{proof}
Thanks to \eqref{ID F22 0N} and \eqref{ID F22 0e}, the equation \eqref{eq ka12} rewrites
\begin{align*}
\bar{\PP}^{[2]}_{u_\A}\pth{\ka^{(1,2)}_\A}_{ij} & =\frac{3}{2}|\nab_{g_0} u_\A|_{g_0}F_\A^2(N_\A)_{i}(N_\A)_{j}.
\end{align*}
Therefore, Lemma \ref{lem P 1et2} implies that \eqref{def ka12} is a solution if and only if it satisfies
\begin{align*}
(\ka^{(1,2)}_\A)_{N_\A N_\A}-\tr_{g_0}\ka^{(1,2)}_\A & = 0,
\\ (\ka^{(1,2)}_\A)_{N_\A e^{(\i)}_\A} & = 0,
\end{align*}
which is obviously the case. 
\end{proof}

\begin{lemma}\label{lem ka1pm}
If we define 
\begin{align}
(\ka^{(1,\pm)}_{\A\B})_{ij} & \vcentcolon=  \frac{-(|\nab_{g_0} u_\A|_{g_0}\pm|\nab_{g_0} u_\B|_{g_0})(\mathcal{I}^{(0,\pm)}_{\A\B})_{ij|_{\Si_0}} + \dr_{(i} (u_\A\pm u_\B)(\mathcal{I}^{(0,\pm)}_{\A\B})_{0j)}    }{2\g_0^{-1}(\d(u_\A\pm u_\B),\d(u_\A\pm u_\B))}   \label{def ka1pm}
\\&\quad +  \frac{(N^{(\pm)}_{\A\B}\tilde{\otimes} \wc{X}^{(2,\pm)}_{\A\B})_{ij}}{|\nab_{g_0}\pth{ u_\A\pm u_\B}|_{g_0}}\non,
\end{align}
then \eqref{eq ka1pm} holds.
\end{lemma}

\begin{proof}
Since $F^{(2,\pm)}_{\A\B}$ is defined by \eqref{def F2pm}, the equation \eqref{eq ka1pm} with $\ka^{(1,\pm)}_{\A\B}$ defined by \eqref{def ka1pm} rewrites $\bar{\PP}^{[2]}_{u_\A\pm u_\B}\pth{\ka^{(1,\pm)}_{\A\B}}=\ka^{(1,\pm)}_{\A\B}$ and, thanks to Lemma \ref{lem P 1et2}, holds if and only if
\begin{align}
(\ka^{(1,\pm)}_{\A\B})_{N^{(\pm)}_{\A\B}N^{(\pm)}_{\A\B}} - \tr_{g_0}\ka^{(1,\pm)}_{\A\B} & = 0,\label{ka1pm 1}
\\ (\ka^{(1,\pm)}_{\A\B})_{N^{(\pm)}_{\A\B}Z} & = 0,\label{ka1pm 2}
\end{align}
for any vector field $Z$ such that $g_0\pth{N^{(\pm)}_{\A\B},Z}=0$. We start with the first identity. We have
\begin{align*}
(\ka^{(1,\pm)}_{\A\B})_{N^{(\pm)}_{\A\B}N^{(\pm)}_{\A\B}} &- \tr_{g_0}\ka^{(1,\pm)}_{\A\B} 
\\& = \mp \frac{|\nab_{g_0} u_\A|_{g_0}\pm|\nab_{g_0} u_\B|_{g_0}  }{4\g_0^{-1}(\d u_\A,\d u_\B)}  \pth{ (\mathcal{I}^{(0,\pm)}_{\A\B})_{N^{(\pm)}_{\A\B}N^{(\pm)}_{\A\B}} - \tr_{g_0}\mathcal{I}^{(0,\pm)}_{\A\B} } + \frac{(\wc{X}^{(2,\pm)}_{\A\B})_{N^{(\pm)}_{\A\B}}}{|\nab_{g_0}\pth{ u_\A\pm u_\B}|_{g_0}}
\\& = -4 (|\nab_{g_0} u_\A|_{g_0}\pm|\nab_{g_0} u_\B|_{g_0} ) \wc{\ffi}^{(2,\pm)}_{\A\B}  + \frac{(\wc{X}^{(2,\pm)}_{\A\B})_{N^{(\pm)}_{\A\B}}}{|\nab_{g_0}\pth{ u_\A\pm u_\B}|_{g_0}},
\end{align*}
where we used \eqref{condition ga2pm} and \eqref{def ga2pm}. Using now \eqref{wcffi 2pm} and \eqref{wcX 2pm} we obtain \eqref{ka1pm 1}. We now turn to \eqref{ka1pm 2}. Let $Z$ such that $g_0\pth{N^{(\pm)}_{\A\B},Z}=0$, we have
\begin{align*}
(\ka^{(1,\pm)}_{\A\B})_{N^{(\pm)}_{\A\B}Z} & = \frac{-(|\nab_{g_0} u_\A|_{g_0}\pm|\nab_{g_0} u_\B|_{g_0})(\mathcal{I}^{(0,\pm)}_{\A\B})_{N^{(\pm)}_{\A\B}Z} + |\nab_{g_0}(u_\A\pm u_\B)|_{g_0} (\mathcal{I}^{(0,\pm)}_{\A\B})_{0Z}  }{\pm4\g_0^{-1}(\d u_\A,\d u_\B)}  
\\&\quad  +  \frac{1}{|\nab_{g_0}\pth{ u_\A\pm u_\B}|_{g_0}} (\wc{X}^{(2,\pm)}_{\A\B})_{Z}.
\end{align*}
From \eqref{I 0 pm}, \eqref{P 0 pm} and \eqref{ID F1 ij}-\eqref{ID F1 0} we find on $\Si_0$
\begin{align*}
-(|\nab_{g_0} u_\A|_{g_0}&\pm|\nab_{g_0} u_\B|_{g_0})(\mathcal{I}^{(0,\pm)}_{\A\B})_{N^{(\pm)}_{\A\B}Z} + |\nab_{g_0}(u_\A\pm u_\B)|_{g_0} (\mathcal{I}^{(0,\pm)}_{\A\B})_{0Z} 
\\& =  \pm \frac{(F^{(1)}_\A)^{\ell}_{\nab_{g_0} u_\B}  (F^{(1)}_\B)_{\ell Z}}{2|\nab_{g_0} (u_\A\pm u_\B)|_{g_0} }  |\nab_{g_0} u_\B|_{g_0} \Big( |\nab_{g_0} u_\A\cdot \nab_{g_0} u_\B|_{g_0} - |\nab_{g_0} u_\A|_{g_0} |\nab_{g_0} u_\B|_{g_0} \Big)
\\&\quad \pm \frac{(F^{(1)}_\B)^{\ell}_{\nab_{g_0} u_\A}   (F^{(1)}_\A)_{\ell Z}}{2|\nab_{g_0}(u_\A\pm u_\B)|_{g_0}}  |\nab_{g_0} u_\A|_{g_0} \Big( |\nab_{g_0} u_\A\cdot \nab_{g_0} u_\B|_{g_0} - |\nab_{g_0} u_\A|_{g_0} |\nab_{g_0} u_\B|_{g_0} \Big)
\\&\quad  +   \frac{1}{4}\left|\bar{F}^{(1)}_\A\cdot \bar{F}^{(1)}_\B\right|_{g_0}\frac{|\nab_{g_0} u_\A\cdot Z|_{g_0}}{|\nab_{g_0}(u_\A\pm u_\B)|_{g_0}} \Big( \mp |\nab_{g_0} u_\A|_{g_0} + |\nab_{g_0} u_\B|_{g_0} \Big)
\\&\hspace{6cm}\times \Big(|\nab_{g_0} u_\A\cdot \nab_{g_0} u_\B|_{g_0} - |\nab_{g_0} u_\A|_{g_0} |\nab_{g_0} u_\B|_{g_0} \Big) ,
\end{align*}
where we also used $|\nab_{g_0} u_\B\cdot Z|_{g_0}=\mp |\nab_{g_0} u_\A\cdot Z|_{g_0} $.  Using \eqref{wcX 2pm} we finally obtain \eqref{ka1pm 2} and thus the lemma.
\end{proof}

\paragraph{Setting the gauge.} We conclude this section by defining initial data for $\T_\la(\h_\la)_{0 \a}$. They will be chosen so that \eqref{GWC} holds on $\Si_0$. Thanks to the definition of $\Upsilon^\rho$ in \eqref{Upsilon} this condition rewrites
\begin{align}
 \g_\la^{\mu\nu} \pth{ \dr_\mu (\h_\la)_{\si\nu} - \half \dr_\si (\h_\la)_{\mu\nu}} & = \tilde{\Upsilon}_\si ,\label{condition dth0}
\end{align}
where we defined
\begin{equation*}
\begin{aligned}
\tilde{\Upsilon}_\si & \vcentcolon = - \g_\la^{\mu\nu}   \sum_\A  \sin\left(\frac{u_\A}{\lambda} \right) \pth{\dr_\mu (\F_\A)_{\si\nu} - \half \dr_\si (\F_\A)_{\mu\nu} } 
\\&\quad - \la  \g_\la^{\mu\nu} \sum_{\substack{u\in\mathcal{N}\\T\in\{\cos,\sin\}}}T\pth{\frac{u}{\la}} \pth{\dr_\mu (\gg^{(3,h)}_{u,T})_{\si\nu} - \half \dr_\si (\gg^{(3,h)}_{u,T})_{\mu\nu}}
\\&\quad - \la  \g_\la^{\mu\nu} \sum_{\substack{v\in\mathcal{I}\\T\in\{\cos,\sin\}}}T\pth{\frac{v}{\la}} \pth{\dr_\mu (\gg^{(3,e)}_{v,T})_{\si\nu} - \half \dr_\si (\gg^{(3,e)}_{v,T})_{\mu\nu}}
\\&\quad  -  \g_\la^{\mu\nu}  \sum_{\substack{v\in\mathcal{I}\\T\in\{\cos,\sin\}}}T'\pth{\frac{v}{\la}}\pth{ \dr_\mu v(\gg^{(3,e)}_{v,T})_{\si\nu} - \half \dr_\si v(\gg^{(3,e)}_{v,T})_{\mu\nu} }
\\&\quad +  (\g_\la)_{\rho\si} \g_0^{\rho\ga} \h_\la^{\mu\nu} \Bigg(\dr_\mu (\g_0)_{\ga\nu} - \half \dr_\ga (\g_0)_{\mu\nu} -  \sum_\A  \sin\left( \frac{u_\A}{\lambda} \right) \pth{ \dr_\mu u_\A (F^{(1)}_\A)_{\ga\nu} - \half \dr_\ga u_\A (F^{(1)}_\A)_{\mu\nu}  }\Bigg).
\end{aligned}
\end{equation*}
Note that $\tilde{\Upsilon}^\rho_{|_{\Si_0}}$ can be completely computed. Moreover, \eqref{condition dth0} can be ensured by the following choice of initial data for $\dr_t(\h_\la)_{\a0}$:
\begin{align}
\dr_t (\h_\la)_{00|_{\Si_0}}  & \vcentcolon = \frac{2}{\g_\la^{00}}\pth{\tilde{\Upsilon}_0 -  \g_\la^{0i}  \dr_i (\h_\la)_{00} - \g_\la^{ij} \pth{ \dr_i (\h_\la)_{0 j} - \half \dr_t (\h_\la)_{ij}}}, \label{ID dth 00}
\\ \dr_t (\h_\la)_{k0|_{\Si_0}}  & \vcentcolon = \frac{1}{\g_\la^{00}}\bigg( \tilde{\Upsilon}_k + \half \g_\la^{00}\dr_k (\h_\la)_{00} - \g_\la^{0i} \pth{ \dr_t (\h_\la)_{k i} - \dr_k (\h_\la)_{0i} +  \dr_i (\h_\la)_{k0}}  \label{ID dth 0k}
\\&\hspace{6cm}- \g_\la^{ij} \pth{ \dr_i (\h_\la)_{k j} - \half \dr_k (\h_\la)_{ij}}\bigg).\non
\end{align}
Note that the RHS of \eqref{ID dth 00} and \eqref{ID dth 0k} are fully known on $\Si_0$ and that we used the fact that $\g_\la^{00}\neq 0$. Note also that the initial data for $\dr_t(\h_\la)_{\a0|_{\Si_0}}$ allows us easily to obtain $\T_\la(\h_\la)_{\a0|_{\Si_0}}$.

\subsubsection{Conclusion}

Combining the results of the previous sections, we can prove the following.

\begin{prop}\label{prop ID main}
Let $\pth{ \bar{F}^{(1)}_\A}_{\A\in\mathcal{A}}$ be a seed. There exists a choice of initial data for the system \eqref{eq F1}-\eqref{eq h} such that 
\begin{itemize}
\item[(i)] the algebraic conditions \eqref{pola F1}-\eqref{GWC} hold on $\Si_0$,
\item[(ii)] the induced metric and the second fundamental form computed from $\g_{\la|_{\Si_0}}$ solve the constraint equations on $\Si_0$.
\end{itemize}
Moreover, the initial data for $F^{(1)}_\A$, $F^{(2,1)}_\A$, $F^{(2,2)}_\A$ are supported in $B_R$ and there exists $C_{\mathrm{init}}=C_{\mathrm{init}}(N,\de,R)>0$ such that
\begin{align}
\l F^{(1)}_\A\r_{H^N(\Si_0)} + \l F^{(2,1)}_\A\r_{H^{N-1}(\Si_0)} + \l F^{(2,2)}_\A\r_{H^{N-1}(\Si_0)} & \leq C_{\mathrm{init}} \e,\label{estim ID 1}
\\ \max_{r\in\llbracket 0 ,4 \rrbracket}\la^r\l \dr \nab^r \h_\la \r_{L^2_{\de+r+1}(\Si_0)} & \leq C_{\mathrm{init}} \e.\label{estim ID 2}
\end{align}
\end{prop}

\begin{proof}
We start by defining initial data for $F^{(1)}_\A$ and $\F_\A$ following \eqref{ID F1 ij}-\eqref{ID F1 0} and \eqref{ID F}. Thanks to \eqref{Q 0 ell}-\eqref{Q 0 0}, we can then define initial data for $F^{(2,1)}_\A$ and $F^{(2,2)}_\A$ following \eqref{ID F21 ij}-\eqref{ID F22 0e}. This then allows us to define every terms on the RHS of \eqref{def ga2pm}, \eqref{def ka11}, \eqref{def ka12} and \eqref{def ka1pm} and thus allows to define $\ga^{(2,\pm)}_{\A\B}$, $\ka^{(1,1)}_\A$, $\ka^{(1,2)}_\A$ and $\ka^{(1,\pm)}_{\A\B}$ by \eqref{def ga2pm}, \eqref{def ka11}, \eqref{def ka12} and \eqref{def ka1pm} respectively. Note that these tensors only depend on the seed $\pth{ \bar{F}^{(1)}_\A}_{\A\in\mathcal{A}}$ and are all of the form $\cro{ (\bar{F}^{(1)})^2 }$, which thanks to \eqref{estim seed} implies
\begin{align}
\l \ga^{(2,\pm)}_{\A\B} \r_{H^N} + \l  \ka^{(1,1)}_\A\r_{H^N} + \l  \ka^{(1,2)}_\A \r_{H^N} + \l  \ka^{(1,\pm)}_{\A\B} \r_{H^N} \lesssim \e^2.
\end{align}
In particular, if $\e$ is small enough then \eqref{estim ga kappa param} holds and we can apply Proposition \ref{prop constraint main} and obtain a solution of the constraint equations $(g_\la,k_\la)$ given by \eqref{g la}-\eqref{k la}. This allows us to define every terms on the RHS of \eqref{ID h ij}-\eqref{ID h 0}, \eqref{ID dth ij} and \eqref{ID dth 00}-\eqref{ID dth 0k}, which in turn define initial data for $\h_\la$ and $\T_\la\h_\la$. This concludes the definition of initial data for the system \eqref{eq F1}-\eqref{eq h}. Moreover, thanks to Lemmas \ref{lem ga2pm}, \ref{lem ka11}, \ref{lem ka12} and \ref{lem ka1pm}, the induced metric and the second fundamental form computed from $\g_\la$ on $\Si_0$ are given by $(g_\la,k_\la)$ and thus are solving the constraint equations on $\Si_0$. We now verify that the algebraic conditions \eqref{pola F1}-\eqref{GWC} hold on $\Si_0$: \eqref{pola F1}-\eqref{pola F1 bis} and \eqref{energy F1} hold thanks to \eqref{ID F1 ij}-\eqref{ID F1 0} and \eqref{assum F1}, \eqref{pola F} holds thanks to \eqref{ID F}, \eqref{pola F21}-\eqref{pola F22} hold thanks to Lemma \ref{lem pola F2 init}, \eqref{pola g3h} already holds thanks to Remark \ref{remark g3h}, \eqref{GWC} holds thanks to the choice of initial data for $\dr_t(\h_\la)_{0\a|_{\Si_0}}$. Finally, the estimate \eqref{estim ID 1} follows from \eqref{estim seed}, \eqref{ID F1 ij}-\eqref{ID F1 0}, \eqref{ID F21 ij}-\eqref{ID F22 0e} and the estimate \eqref{estim ID 2} follows in addition from \eqref{estim reste cons 1}-\eqref{estim reste cons 2}, \eqref{ID h ij}-\eqref{ID h 0}, \eqref{ID dth ij} and \eqref{ID dth 00}-\eqref{ID dth 0k}.
\end{proof}


\section{Solving the system}\label{section solving the system}

In this section, we present our result on well-posedness for the system \eqref{eq F1}-\eqref{eq h} with the initial data constructed in Proposition \ref{prop ID main}. It is contained in the following proposition.

\begin{prop}\label{prop system}
If $\e$ is small enough and $\la\leq\e$, then there exists a unique solution \[\pth{F^{(1)}_\A,F^{(2,1)}_\A,F^{(2,2)}_\A,\F_\A,\h_\la}\] to the system \eqref{eq F1}-\eqref{eq h} on $[0,1]\times\R^3$ with the initial data of Proposition \ref{prop ID main}. Moreover, the tensors $F^{(1)}_\A$, $F^{(2,1)}_\A$, $F^{(2,2)}_\A$ and $\F_\A$ are supported in $J^+_0(B_R)$ and there exists $C=C(N,\de,R)>0$ such that 
\begin{align}
\l F^{(1)}_\A \r_{H^N} + \l F^{(2,1)}_\A \r_{H^{N-2}} + \l F^{(2,2)}_\A \r_{H^{N-2}} & \leq C \e, \label{estim F1 F21 F22}
\\ \l \F_\A \r_{L^2} + \max_{r\in\llbracket 1,6\rrbracket}\lambda^{r-1} \l \nabla^r \F_\A \r_{L^2}+ \max_{r\in\llbracket 0,5\rrbracket}\lambda^{r} \l \dr_t \nabla^r \F_\A \r_{L^2} + \max_{r\in\llbracket 0,4\rrbracket}\lambda^{r+1} \l \dr_t^2 \nabla^r \F_\A \r_{L^2} & \leq C\e, \label{estim F}
\\ \max_{r\in\llbracket 0,4\rrbracket} \lambda^r \l \nabla^r \tBox_{\g_\la} \F_\A \r_{L^2} & \leq C \e,\label{estim box F}
\\ \max_{r\in\llbracket 0,4\rrbracket}\lambda^r\pth{ \l \dr_t \nabla^r \h_\la \r_{L^2_{\delta+1+r}} + \l \nabla \nabla^r \h_\la \r_{L^2_{\delta+1+r}} } + \max_{r\in\llbracket 0,3\rrbracket} \lambda^{r+1} \l \dr_t^2 \nabla^r \h_\la \r_{L^2_{\delta+2+r}} & \leq C \e.  \label{estim h}
\end{align}
\end{prop}

The proof of Proposition \ref{prop system} follows exactly the same lines as the corresponding statements in the singlephase case of \cite{Touati2023a}, namely Theorems 6.1 and 7.1 there. Indeed, the reasons why well-posedness of the system \eqref{eq F1}-\eqref{eq h} does not follow from simple arguments are not linked with the directions of oscillations and thus don't differ from the singlephase to the the multiphase case. Therefore, we won't reproduce the argument and refer the reader to Sections 6 and 7 of \cite{Touati2023a}. We will however provide an outline of what is done there. As in the singlephase case, proving well-posedness is done in two steps: we first consider the equations \eqref{eq F1}-\eqref{eq F22} and then the equations \eqref{eq F}-\eqref{eq h}.

\paragraph{Solving for $F^{(1)}_\A$, $F^{(2,1)}_\A$ and $F^{(2,2)}_\A$.} The equations \eqref{eq F1}-\eqref{eq F22} form a triangular system for the unknowns $F^{(1)}_\A$, $F^{(2,1)}_\A$ and $F^{(2,2)}_\A$. Thanks to \eqref{Q 0} and \eqref{Rtilde 1 1}-\eqref{Rtilde 1 2} the equations \eqref{eq F1}-\eqref{eq F22} rewrite schematically
\begin{align*}
\Ll_\A F^{(1)}_\A & = 0,
\\ \Ll_\A F^{(2,1)}_\A & = \cro{  \dr^{\leq 2} F^{(1)} + (F^{(1)})^2 + (F^{(1)})^3 },
\\ \Ll_\A F^{(2,2)}_\A & = \cro{ F^{(1)}\dr F^{(1)} + F^{(2,1)}F^{(1)} +  (F^{(1)})^2  }.
\end{align*}
Since the transport operator $\Ll_\A$ only depends on the background spacetime, a combination of the characteristics method and the energy estimate for the vector field $L_\A$ easily implies the existence and uniqueness of solutions $F^{(1)}_\A$, $F^{(2,1)}_\A$ and $F^{(2,2)}_\A$ to these equations on $[0,1]\times \R^3$. The estimate \eqref{estim F1 F21 F22} follows from the following consideration: it holds on $\Si_0$ (see Proposition \ref{prop ID main}) and is easily propagated for $F^{(1)}_\A$, for $F^{(2,1)}_\A$ we lose two derivatives because of the $\dr^{\leq 2}F^{(1)}$ on the RHS, and because of the $F^{(2,1)}$ on the RHS of the equation for $F^{(2,2)}_\A$ it lives at the same level of regularity. The only harmless difference with the singlephase case is that there is as many equations as number of directions of oscillation, and since we don't keep track on how they mix in the equations (meaning for instance that one could see a $F^{(1)}_\B$ in the equation for $F^{(2,1)}_\A$ with $\A\neq\B$) the estimate \eqref{estim F1 F21 F22} is far from being uniform in $|\mathcal{A}|$. 

\paragraph{Solving for $\F_\A$ and $\h_\la$.} Once $F^{(1)}_\A$, $F^{(2,1)}_\A$ and $F^{(2,2)}_\A$ are defined, the equations \eqref{eq F}-\eqref{eq h} form an independent coupled system for $\F_\A$ and $\h_\la$. Besides the fact that there is a transport equation for each $\A\in\mathcal{A}$, the only difference between the multiphase and singlephase case is the presence of $\tBox_{\g_\la}\gg^{(3,e)}_{v,T}$ on the RHS of \eqref{eq h}. Indeed this term is not present in the singlephase case since its purpose is precisely to deal with the mixed harmonics in $R^{(1)}_{\a\b}$. However, this new term dos not change the structure of the system. Indeed, thanks to \eqref{def g3e} and \eqref{R 1 e}, $\gg^{(3,e)}_{v,T}$ contains terms depending on $F^{(1)}_\A$, $F^{(2,1)}_\A$ and $F^{(2,2)}_\A$ and a linear term in $\F_\A$, the latter being the only one of interest at this stage. Since the term $\tBox_{\g_\la}\gg^{(3,e)}_{v,T}$ comes with an extra $\la$ power in \eqref{eq h} this new term is at least better than the $\tBox_{\g_\la}\F_\A$ in \eqref{eq h} and already present in the singlephase case. In conclusion, one can say that the argument (which we will briefly outline below) of Section 7 of \cite{Touati2023a} applies also here without any change. 

\saut
The major issue with the coupled system \eqref{eq F}-\eqref{eq h} is the \textit{a priori} loss of one derivative, which prevents a straightforward proof of well-posedness. Indeed, from an energy estimate on \eqref{eq F} one concludes that $\F_\A$ lives at the same level of regularity as $\h_\la$, and from an energy estimate on \eqref{eq h} one concludes that one derivative of $\h_\la$ lives at the same level of regularity as two derivatives of $\F_\A$ (because of $\tBox_{\g_\la}\F_\A$ on the RHS). The system \eqref{eq F}-\eqref{eq h} is thus \textit{a priori} ill-posed. The operators $\Pi_\leq$ and $\Pi_\geq$ have been introduced precisely to solve this problem. Indeed, applying the Bernstein estimate 
\begin{align}\label{bernstein}
\l \nab \Pi_\leq (u) \r_{L^2} &\lesssim \frac{1}{\la} \l u \r_{L^2}
\end{align} 
to \eqref{eq F} we can estimate two derivatives of $\F_\A$ by one derivative of $\h_\la$ at the cost of the loss of one $\la$ power. This would however prove well-posedness of \eqref{eq F}-\eqref{eq h} only on a time scale of order $\la$, which is far from satisfactory if one is interested in the limit $\la\to0$. Therefore, one absolutely needs to use the structure of the term $\tBox_{\g_\la}\F_\A$ on the RHS of \eqref{eq h}. This is done by decomposing it as follows
\begin{align}\label{box F}
\tBox_{\g_\la}\F_\A & = \tBox_{\g_0}\F_\A + \pth{\g_\la^{\mu\nu}-\g_0^{\mu\nu}}\dr_\mu \dr_\nu \F_\A.
\end{align}
We explain how the two terms in \eqref{box F} can be estimated:
\begin{itemize}
\item  The first term in \eqref{box F} can be estimated by commuting \eqref{eq F} and using the estimate (proved in Appendix B of \cite{Touati2023a})
\begin{align*}
\l [L_\A,\Box_{\g_0}]f \r_{L^2} \leq C(C_0)\pth{ \l \dr L_\A f \r_{L^2} + \l \Box_{\g_0}f \r_{L^2} + \l \dr f \r_{L^2} },
\end{align*}
where it is crucial that the only second order derivatives appearing on the RHS include a $L_\A$ or are directly a wave operator. This implies that $\Box_{\g_0}\F_\A$ can be estimated roughly by $\dr \h_\la$ and $\Box_{\g_0}\Pi_\leq(\h)$. The former is itself controlled by $\Box_{\g_\la}\F_\A$ thanks to \eqref{eq h} and the latter requires to commute $\Box_{\g_0}$ and $\Pi_\leq$. In Appendix C of \cite{Touati2023a} we prove the following estimate
\begin{align*}
\l \left[ u , \Pi_\leq  \right]\nabla v \r_{L^2} & \lesssim  \left( \l \nabla  u \r_{L^\infty}  +   \lambda \l u \r_{H^{\frac{7}{2}}}  \right)\l v \r_{L^2}.
\end{align*}
This estimates shows in particular how one gain a derivative at top order.
\item For the second term in \eqref{box F}, instead of relying on the structure of the derivatives we benefit from the fact that $\g_\la^{\mu\nu}-\g_0^{\mu\nu}=\GO{\la}$. This extra factor $\la$ makes now possible the use of the Bernstein estimate \eqref{bernstein}.
\end{itemize}
Integrating this procedure in a bootstrap argument with \eqref{estim F}-\eqref{estim h} as bootstrap assumptions allows us to prove well-posedness for the system \eqref{eq F}-\eqref{eq h} and obtain a unique solution $\F_\A$ and $\h_\la$ satisfying \eqref{estim F}-\eqref{estim h} on the desired time scale. This concludes our outline of the proof of Proposition \ref{prop system}, and we again refer to \cite{Touati2023a} for all the technical details.


\section{Propagation of polarization and gauge conditions}\label{section propagation}

In this section, we conclude the proof of Theorem \ref{theo evol} by showing that the metric $\g_\la$ on $[0,1]\times\R^3$ defined by \eqref{expansion metric} where $\pth{F^{(1)}_\A,F^{(2,1)}_\A,F^{(2,2)}_\A,\F_\A,\h_\la}$ are constructed in Section \ref{section solving the system} and where $F^{(2,\pm)}_{\A\B}$, $\gg^{(3,e)}_{v,T}$ and $\gg^{(3,h)}_{u,T}$ are respectively defined in \eqref{def F2pm}, \eqref{def g3e} and Remark \ref{remark g3h}, is a solution of the Einstein vacuum equations \eqref{EVE}. As first stated in Section \ref{section reformulation}, this will be done by showing that the algebraic conditions \eqref{pola F1}-\eqref{GWC} holds on $[0,1]\times\R^3$ (\eqref{pola g3h} already holds thanks to the definition of $\gg^{(3,h)}_{u,T}$, see Remark \ref{remark g3h}), relying obviously on the fact that they all hold on $\Si_0$ thanks to Proposition \ref{prop ID main}. These conditions are separated in two groups:
\begin{itemize}
\item First, the conditions on $F^{(1)}_\A$ and $\F_\A$, i.e \eqref{pola F1}-\eqref{pola F}, \eqref{energy F1}, which are all propagated by directly using the transport equations that $F^{(1)}_\A$ and $\F_\A$ satisfy on $[0,1]\times\R^3$, i.e \eqref{eq F1} and \eqref{eq F}. This is done in Section \ref{section propa 1}.
\item Second, the conditions on $F^{(2,1)}_\A$ and $F^{(2,2)}_\A$, i.e \eqref{pola F21}-\eqref{pola F22}, together with the generalised wave gauge condition \eqref{GWC}, which are propagated using the contracted Bianchi identities. This is done in Section \ref{section propa 2}.
\end{itemize}

\begin{remark}
As explained in Remark \ref{remark Q0}, the conditions \eqref{pola F21}-\eqref{pola F22} are not gauge conditions since they don't depend on coordinates. However we treat them as gauge conditions and propagate them with the contracted Bianchi identities.
\end{remark}

\subsection{Propagation via the system} \label{section propa 1}

The proof of the following lemma is identical to the corresponding ones in \cite{Touati2023a}, we only sketch its proof and refer the reader to Section 8.1 in \cite{Touati2023a} for more details.

\begin{lemma}\label{lem pola F1 F}
The polarization conditions \eqref{pola F1}-\eqref{pola F} and the backreaction condition \eqref{energy F1} hold on $[0,1]\times \R^3$.
\end{lemma}

\begin{proof}
By contracting the transport equation \eqref{eq F1} with various vector fields we obtain the transport system
\begin{equation}\label{system pola F1}
\begin{aligned}
\Ll_\A \pth{ \Pol{F^{(1)}_\A}{u_\A}_{L_\A}} & = 0,
\\ \Ll_\A \pth{ \Pol{F^{(1)}_\A}{u_\A}_{e^{(\mathbf{k})}_\A}} & = - \g_0\pth{ \D_{L_\A}\Lb_\A,e^{(\mathbf{k})}_\A } \Pol{F^{(1)}_\A}{u_\A}_{L_\A} 
\\&\quad - 2 \sum_{\i,\j=1,2} \de_{\i\j}\g_0\pth{ \D_{L_\A}e^{(\mathbf{k})}_\A,e^{(\i)}_\A} \Pol{F^{(1)}_\A}{u_\A}_{e^{(\j)}_\A},
\\ \Ll_\A \pth{ \Pol{F^{(1)}_\A}{u_\A}_{\Lb_\A}} & =  -2 \sum_{\i,\j=1,2} \de_{\i\j} \g_0\pth{\D_{L_\A}\Lb_\A,e^{(\i)}_\A} \Pol{F^{(1)}_\A}{u_\A}_{e^{(\j)}_\A},
\\ \Ll_\A \pth{ (F^{(1)}_\A)_{L_\A\Lb_\A} } & = 2 \sum_{\i,\j=1,2} \de_{\i\j} \g_0\pth{\D_{L_\A}\Lb_\A,e^{(\i)}_\A} \Pol{F^{(1)}_\A}{u_\A}_{e^{(\j)}_\A}.
\end{aligned}
\end{equation}
Since $\Pol{F^{(1)}_\A}{u_\A}_{|_{\Si_0}}=0$ and $(F^{(1)}_\A)_{L_\A\Lb_\A|_{\Si_0}}=0$, the system \eqref{system pola F1} implies that \eqref{pola F1}-\eqref{pola F1 bis} hold on $[0,1]\times \R^3$. Similarly, using now \eqref{pola F1}-\eqref{pola F1 bis} we deduce from \eqref{eq F1} the following transport equation 
\begin{align*}
\pth{ - L_\A + \Box_{\g_0}u_\A} \left| F^{(1)}_\A \right|_{\g_0}^2 & = 0.
\end{align*}
Using now the background transport equation satisfied by $F_\A$ (recall \eqref{BG system}), we obtain
\begin{align}\label{transport energy}
\pth{ - L_\A + \Box_{\g_0}u_\A} \pth{ \left| F^{(1)}_\A \right|_{\g_0}^2 - 8 F_\A^2} & = 0.
\end{align}
Since we have 
\begin{align*}
\pth{ \left| F^{(1)}_\A \right|_{\g_0}^2 - 8 F_\A^2}_{|_{\Si_0}} =0,
\end{align*}
the equation \eqref{transport energy} implies that \eqref{energy F1} holds on $[0,1]\times \R^3$. Finally, using again \eqref{pola F1} we deduce from \eqref{eq F} that $\Pol{\F_\A}{u_\A}$ satisfies the same system as $\Pol{F^{(1)}_\A}{u_\A}$, i.e \eqref{system pola F1} with $F^{(1)}_\A$ replaced by $\F_\A$. Since $\Pol{\F_\A}{u_\A}_{|_{\Si_0}}=0$, we again obtain that \eqref{pola F} holds on $[0,1]\times \R^3$, which concludes the proof of the lemma.
\end{proof}

\subsection{Propagation via the contracted Bianchi identities} \label{section propa 2}

We now propagate from $\Si_0$ the conditions \eqref{pola F21}-\eqref{pola F22} and \eqref{GWC}. This is done by deriving extra equations for the quantities $V^{(2,1)}_\A$, $V^{(2,2)}_\A$ and $\Upsilon^\rho$ from the contracted Bianchi identities, which state that the Einstein tensor of any Lorentzian metric is divergence free. We first compute the Einstein tensor of $\g_\la$.

\begin{lemma}\label{lem G}
The Einstein tensor of $\g_\la$ admits the decomposition
\begin{align*}
G_{\a\b}(\g_\la) & = G(V)_{\a\b} + G(\Upsilon)_{\a\b}.
\end{align*}
The term $G(V)_{\a\b}$ is given by
\begin{align}\label{expansion G(V)}
G(V)_{\a\b} & = G^{(0)}_{\a\b} + \la G^{(1)}_{\a\b} + \GO{\la^2},
\end{align}
where
\begin{align}
G^{(0)}_{\a\b} & = -\half \sum_\A \sin\pth{\frac{u_\A}{\la}} \pth{ \dr_{(\a} u_\A (V_\A^{(2,1)})_{\b)} + (V^{(2,1)}_\A)_{L_\A} (\g_0)_{\a\b} } \label{G 0}
\\&\quad -2 \sum_\A \cos\pth{\frac{2u_\A}{\la}} \pth{ \dr_{(\a} u_\A (V_\A^{(2,2)})_{\b)} + (V^{(2,2)}_\A)_{L_\A} (\g_0)_{\a\b} } ,\non
\\ G^{(1)}_{\a\b} & =  \frac{1}{2}\sum_\A \cos\pth{\frac{u_\A}{\lambda}} \pth{ \D_{(\a} (V^{(2,1)}_\A)_{\b)} - \div_{g_0} V^{(2,1)}_\A (\g_0)_{\a\b} } \label{G 1}
\\&\quad  - \sum_\A\sin\pth{\frac{2u_\A}{\lambda}} \pth{ \D_{(\a} (V^{(2,2)}_\A)_{\b)}  - \div_{g_0}V^{(2,2)}_\A (\g_0)_{\a\b}}\non
\\&\quad + \sum_{\substack{\A , k=1,2,3\\ T\in\{\cos,\sin\}}} T\pth{\frac{ku_\A}{\la}} \cro{V^{(2,1)} + V^{(2,2)}} (F^{(1)}_\A)_{\a\b} \non
\\&\quad + \sum_{\substack{v\in\mathcal{I}\\T\in\{\cos,\sin\}}} T\pth{\frac{v}{\la}}\cro{  (V^{(2,1)} + V^{(2,2)}) F^{(1)} }_{\a\b}.\non
\end{align}
The term $G(\Upsilon)_{\a\b}$ is given by
\begin{align}\label{def G upsilon}
G(\Upsilon)_{\a\b} & = \frac{\la^2}{2} \bigg( (\g_\la)_{\rho(\a}\dr_{\b)}\Upsilon^\rho - (\g_\la)_{\a\b}\dr_\rho\Upsilon^\rho + \Upsilon^\rho \dr_\rho (\g_\la)_{\a\b} - \half \g_\la^{\mu\nu}\Upsilon^\rho \dr_\rho (\g_\la)_{\mu\nu}(\g_\la)_{\a\b} \bigg).
\end{align}
\end{lemma}

\begin{proof}
Thanks to Lemma \ref{lem pola F1 F}, the assumptions \eqref{assumption} are satisfied and we can use the results of Proposition \ref{prop expression of the ricci}. Thanks to Lemma \ref{lem pola F1 F} again, the non-oscillating term in \eqref{R 0} vanish and thanks to \eqref{R 0 mixed}, \eqref{def F2pm} and Lemma \ref{lem F2pm} we have $(R^{(0)}_{\mathrm{mixed}})_{\a\b}=0$. Therefore we have $R^{(0)}_{\a\b} = (R^{(0)}_{\mathrm{null}})_{\a\b}$. Moreover, the transport equation \eqref{eq F1}, the polarization condition \eqref{pola F} (which holds thanks to Lemma \ref{lem pola F1 F}) and the definitions \eqref{def 1 V 2 1}-\eqref{def 1 V 2 2} imply that
\begin{align}\label{R0 final}
R^{(0)}_{\a\b} & = - \half \sum_\A \sin\pth{\frac{u_\A}{\la}}  \dr_{(\a}u_\A (V^{(2,1)}_\A )_{\b)} - 2  \sum_\A \cos\pth{\frac{2u_\A}{\la}}  \dr_{(\a}u_\A (V^{(2,2)}_\A )_{\b)}.
\end{align}
We now compute $R^{(1)}_{\a\b}$. The definition \eqref{def g3e} of $\gg^{(3,e)}_{v,T}$ precisely ensures that $(R^{(1)}_{\mathrm{mixed}})_{\a\b}=0$ so that 
\begin{align}\label{R1 final}
R^{(1)}_{\a\b} = (R^{(1)}_{\mathrm{null}})_{\a\b} + \half (\g_0)_{\rho(\a}(\dr_{\b)}\Upsilon^\rho)^{(-1)}.
\end{align}
Moreover, the transport equations \eqref{eq F21}-\eqref{eq F} and the conditions \eqref{pola F} and \eqref{pola g3h} imply
\begin{align}\label{R1 final bis}
(R^{(1)}_{\mathrm{null}})_{\a\b} & = \half\sum_\A \cos\pth{\frac{u_\A}{\lambda}} \bigg(  -  \Pi_\geq \pth{ (\h_\lambda)_{L_\A L_\A} }(F^{(1)}_\A)_{\alpha\beta}   + \D_{(\a} (V^{(2,1)}_\A)_{\b)}  \bigg) \non
\\&\quad - \sum_\A\sin\pth{\frac{2u_\A}{\lambda}}  \D_{(\a} (V^{(2,2)}_\A)_{\b)}  
\\&\quad - \half\sum_\A\bigg( 2  \sin\pth{\frac{u_\A}{\la}}  \sin\pth{\frac{2u_\A}{\la}}  - \cos\pth{\frac{u_\A}{\lambda}} \cos\pth{ \frac{2u_\A}{\lambda} } \bigg) (V^{(2,2)}_\A)_{L_\A} (F^{(1)}_\A)_{\alpha\beta}  ,\non
\end{align}
where we used \eqref{R 1 null bis}. Finally, thanks to the wave equation \eqref{eq h} we have from \eqref{R 2}:
\begin{align}\label{R2 final}
R^{(\geq 2)}_{\a\b} & =  \frac{1}{2\la}\sum_\A \cos\pth{\frac{u_\A}{\la}}\Pi_\geq\pth{(\h_\lambda)_{L_\A L_\A}}(F^{(1)}_\A)_{\a\b} +\half\pth{ \Upsilon^\rho \dr_\rho (\g_\la)_{\a\b} + (\g_\la)_{\rho(\a} \dr_{\b)}\Upsilon^\rho }^{(\geq 0)} .
\end{align}
Combining \eqref{R0 final}-\eqref{R2 final} we obtain the decomposition
\begin{align*}
R_{\a\b}(\g_\la) & =  R(V)_{\a\b} + R(\Upsilon)_{\a\b}
\end{align*}
with 
\begin{align}
R(V)_{\a\b} & \vcentcolon = - \half \sum_\A \sin\pth{\frac{u_\A}{\la}}  \dr_{(\a}u_\A (V^{(2,1)}_\A )_{\b)} - 2  \sum_\A \cos\pth{\frac{2u_\A}{\la}}  \dr_{(\a}u_\A (V^{(2,2)}_\A )_{\b)} \non
\\&\quad +  \frac{\la}{2}\sum_\A \cos\pth{\frac{u_\A}{\lambda}} \D_{(\a} (V^{(2,1)}_\A)_{\b)}  - \la\sum_\A\sin\pth{\frac{2u_\A}{\lambda}}  \D_{(\a} (V^{(2,2)}_\A)_{\b)}   \label{R V}
\\&\quad - \frac{\la}{2}\sum_\A\bigg( 2  \sin\pth{\frac{u_\A}{\la}}  \sin\pth{\frac{2u_\A}{\la}}  - \cos\pth{\frac{u_\A}{\lambda}} \cos\pth{ \frac{2u_\A}{\lambda} } \bigg) (V^{(2,2)}_\A)_{L_\A} (F^{(1)}_\A)_{\alpha\beta},  \non
\\ R(\Upsilon)_{\a\b} & \vcentcolon = \frac{\la^2}{2}\pth{ \Upsilon^\rho \dr_\rho (\g_\la)_{\a\b} + (\g_\la)_{\rho(\a} \dr_{\b)}\Upsilon^\rho }. \label{R Upsilon}
\end{align}
We define 
\begin{align*}
G(V)_{\a\b} & \vcentcolon = R(V)_{\a\b} - \half \g_\la^{\mu\nu}R(V)_{\mu\nu}(\g_\la)_{\a\b},
\\ G(\Upsilon)_{\a\b} & \vcentcolon = R(\Upsilon)_{\a\b} - \half \g_\la^{\mu\nu}R(\Upsilon)_{\mu\nu}(\g_\la)_{\a\b},
\end{align*}
so that the Einstein tensor of $\g_\la$ satisfies $G_{\a\b}(\g_\la) = G(V)_{\a\b} + G(\Upsilon)_{\a\b}$. Moreover, $G(\Upsilon)_{\a\b}$ is indeed given by \eqref{def G upsilon}. It remains to compute $G(V)_{\a\b}$, for that we compute the trace $\g_\la^{\mu\nu}R(V)_{\mu\nu}$:
\begin{align*}
\g_\la^{\mu\nu}R(V)_{\mu\nu} & =  \sum_\A \sin\pth{\frac{u_\A}{\la}}  (V^{(2,1)}_\A )_{L_\A} + 4  \sum_\A \cos\pth{\frac{2u_\A}{\la}}  (V^{(2,2)}_\A )_{L_\A}
\\&\quad + \la\sum_\A \cos\pth{\frac{u_\A}{\lambda}} \div_{g_0} V^{(2,1)}_\A  - 2\la\sum_\A\sin\pth{\frac{2u_\A}{\lambda}} \div_{g_0}V^{(2,2)}_\A
\\&\quad + \la \sum_{\substack{v\in\mathcal{I}\\T\in\{\cos,\sin\}}} T\pth{\frac{v}{\la}}\cro{ F^{(1)} (V^{(2,1)} + V^{(2,2)}) } + \GO{\la^2}.
\end{align*}
Plugging this into the definition of $G(V)_{\a\b}$, we obtain the decomposition \eqref{expansion G(V)}.
\end{proof}

The contracted Bianchi identities applied to the metric $\g_\la$ now read
\begin{align}\label{cBI}
\div_{\g_\la} G(V) + \div_{\g_\la} G(\Upsilon) & = 0.
\end{align}
To deduce extra equations from \eqref{cBI} we compute the divergence of $G(\Upsilon)$ and $G(V)$ in the two following lemmas.

\begin{lemma}\label{lem divG(Ups)}
Let $K$ be any compact set of $\R^3$.  We have 
\begin{align*}
\dive_{\g_\la} G(\Upsilon) & = \sum_{\substack{v\in\mathcal{I}\\T\in\{\cos,\sin\}}} T\pth{\frac{v}{\la}} \cro{\gg^{(3,e)}} + \mathcal{R}_1
\end{align*}
with 
\begin{align}\label{reste 1}
\l \mathcal{R}_1 \r_{L^2(K)} \lesssim  \la.
\end{align}
\end{lemma}

\begin{proof}
From \eqref{def G upsilon} we obtain
\begin{align*}
\dive_{\g_\la} G(\Upsilon)_\a & = \frac{\la^2}{2} \pth{ (\g_\la)_{\rho\a}\tBox_{\g_\la}\Upsilon^\rho + \mathcal{B}_\a }
\end{align*}
where $\mathcal{B}_\a$ is schematically of the form
\begin{align*}
\mathcal{B} & =\pth{ 1 + \g_\la^{-1}\g_\la } \pth{ \g_\la^{-1} \dr \Upsilon \dr \g_\la + \g_\la^{-1}  \Upsilon \dr^2\g_\la + \g_\la^{-1}  \g_\la^{-1}\Upsilon \dr\g_\la\dr \g_\la  }.
\end{align*}
Since $\Upsilon=\GO{1}$ and $\g_\la=\GO{1}$ both with oscillating coefficients, the worst terms in $\mathcal{B}$ are $\dr\Upsilon$ and $\dr^2\g_\la$, and note that they are not multiplied together. Using the regularity and the estimates stated in Proposition \ref{prop system}, we conclude easily that
\begin{align*}
\l \mathcal{B} \r_{L^2(K)} \lesssim  \frac{1}{\la}.
\end{align*} 
We now look at the term $\tBox_{\g_\la}\Upsilon^\rho$. From \eqref{Upsilon} we obtain the following schematic expression of $\Upsilon^\rho$:
\begin{align*}
\Upsilon & \vcentcolon =    \cro{\dr\h_\la} +  \sum_\A  \sin\left(\frac{u_\A}{\lambda} \right) \cro{\dr\F_\A + F^{(1)}} +  \sum_{\substack{v\in\mathcal{I}\\T\in\{\cos,\sin\}}}T'\pth{\frac{v}{\la}}\cro{\gg^{(3,e)}} + \text{better terms},
\end{align*}
where for clarity we choose not to write the $\g_\la^{-1}$ factors (which are lower order terms from the point of view of the $\la$ behaviour) and where the better terms includes terms with the same amounts of derivatives but with better $\la$ behaviour (for instance, terms of the form $\la\dr\gg^{(3,e)}$ or $\la\dr\gg^{(3,h)}$). Therefore, using \eqref{Box T f} we obtain
\begin{equation}\label{box upsilon}
\begin{aligned}
\tBox_{\g_\la}\Upsilon & \vcentcolon =    \cro{\dr\tBox_{\g_\la}\h_\la}  +  \cro{ \dr\tBox_{\g_\la}\F + \dr^2F^{(1)} }^\osc + \frac{1}{\la} \cro{ \dr L_\A  \F_\A + \dr^{\leq 1}\F_\A + \dr^{\leq 1}F^{(1)}}^\osc
\\&\quad + \frac{1}{\la^2}  \sum_{\substack{v\in\mathcal{I}\\T\in\{\cos,\sin\}}} T\pth{\frac{v}{\la}}\cro{\gg^{(3,e)}} + \frac{1}{\la} \cro{\dr^{\leq 1}\gg^{(3,e)}}^\osc + \cro{ \tBox_{\g_\la}\gg^{(3,e)}}^\osc + \text{better terms},
\end{aligned}
\end{equation}
where we used the fact that the commutators $\left[ \tBox_{\g_\la}, \dr \right]$ and $[L_\A,\dr]$ are better terms. Compared to the singlephase construction of \cite{Touati2023a}, the only new term in \eqref{box upsilon} is \[\frac{1}{\la^2}  \sum_{\substack{v\in\mathcal{I}\\T\in\{\cos,\sin\}}} T\pth{\frac{v}{\la}}\cro{\gg^{(3,e)}},\] which we keep as it is. The other terms can be estimated as in Lemma 8.2 of \cite{Touati2023a} (using in particular the fact that the worst term in $\gg^{(3,e)}$ is a linear term in $\F$), we don't write the details and obtain that there are all bounded by $\frac{1}{\la}$ in $L^2(K)$. This proves the lemma.
\end{proof}

\begin{lemma}\label{lem divG(V)}
We have 
\begin{align*}
\dive_{\g_\la}G(V)_\a & = \half \sum_\A \sin\pth{\frac{u_\A}{\la}}\Big( 2\D_{L_\A} (V^{(2,1)}_\A)_{\a} +  \div_{\g_0}L_\A (V^{(2,1)}_\A)_{\a}    + (V^{(2,1)}_\A)^\b\D_{[\b}(L_\A)_{\a]}   \Big)
\\&\quad + 2 \sum_\A \cos\pth{\frac{2u_\A}{\la}} \Big(  2 \D_{L_\A} (V^{(2,2)}_\A)_{\a} + \div_{\g_0}L_\A (V^{(2,2)}_\A)_{\a} + (V^{(2,2)}_\A)^\b\D_{[\b}(L_\A)_{\a]} \Big)
\\&\quad +  \sum_{\substack{v\in\mathcal{I}\\T\in\{\cos,\sin\}}} T\pth{\frac{v}{\la}}\cro{  (V^{(2,1)} + V^{(2,2)}) F^{(1)} }_{\a\b} + (\mathcal{R}_2)_\a
\end{align*}
with 
\begin{align}\label{reste 2}
\l \mathcal{R}_2 \r_{L^2(K)} \lesssim  \la.
\end{align}
\end{lemma}

\begin{proof}
Since $G(V)$ contains oscillating terms at the order $\la^0$ we have 
\begin{align*}
\dive_{\g_\la}G(V)_\a & = \frac{1}{\la} \pth{ \dive_{\g_\la}G(V)_\a }^{(-1)} + \pth{ \dive_{\g_\la}G(V)_\a }^{(0)} + \GO{\la}.
\end{align*}
Recall that if $T$ is a symmetric 2-tensor we have in coordinates 
\begin{align*}
\dive_{\g_\la}T_\a & = \g_\la^{\rho\b}  \dr_\rho T_{\a\b} - \g_\la^{\rho\b}\Ga(\g_\la)_{\rho(\a}^{\mu}T_{\mu\b)} .
\end{align*}
Therefore, we have thanks to Lemma \ref{lem G}:
\begin{align*}
\pth{ \dive_{\g_\la}G(V)_\a }^{(-1)} & = -\half \sum_\A \cos\pth{\frac{u_\A}{\la}} \dr^\b u_\A\pth{ \dr_{(\a} u_\A (V_\A^{(2,1)})_{\b)} + (V^{(2,1)}_\A)_{L_\A} (\g_0)_{\a\b} } 
\\&\quad + 4 \sum_\A \sin\pth{\frac{2u_\A}{\la}} \dr^\b u_\A\pth{ \dr_{(\a} u_\A (V_\A^{(2,2)})_{\b)} + (V^{(2,2)}_\A)_{L_\A} (\g_0)_{\a\b} } 
\\& = 0.
\end{align*}
Moreover, we have
\begin{align*}
\pth{ \dive_{\g_\la}G(V)_\a }^{(0)} & = \pth{\div_{g_0} G^{(0)}_{\a}}^{(0)}   +  \g_0^{\rho\b} \pth{ \dr_\rho G^{(1)}_{\a\b} }^{(-1)}   + \pth{ \g_\la^{\rho\b}}^{(1)} \pth{ \dr_\rho G(V)_{\a\b} }^{(-1)}  - \g_0^{\rho\b} (\tilde{\Ga}^{(0)})_{\rho(\a}^{\mu} G^{(0)}_{\mu\b)}
\\& \vcentcolon= I + II + III + IV,
\end{align*}
where
\begin{align*}
(\tilde{\Ga}^{(0)})_{\rho\a}^{\mu} & = -\half\sum_\B  \sin\pth{\frac{u_\B}{\la}} \pth{\dr_{(\rho} u_\B (F^{(1)}_\B)_{\a)}^\mu - \dr^\mu  u_\B (F^{(1)}_\B)_{\rho\a} }.
\end{align*}
For $I$, \eqref{G 0} implies
\begin{align*}
I & = \half \sum_\A \sin\pth{\frac{u_\A}{\la}} \Big( (L_\A)_{\a} \div_{\g_0} (V^{(2,1)}_\A) + (V^{(2,1)}_\A)_{\a}\div_{\g_0}L_\A 
\\&\hspace{4cm} +  \D_{L_\A} (V^{(2,1)}_\A)_{\a} + (V^{(2,1)}_\A)^\rho\D_\rho(L_\A)_{\a} - \dr_\a(V^{(2,1)}_\A)_{L_\A}  \Big)
\\&\quad + 2 \sum_\A \cos\pth{\frac{2u_\A}{\la}} \Big( (L_\A)_{\a} \div_{\g_0} (V^{(2,2)}_\A) + (V^{(2,2)}_\A)_{\a}\div_{\g_0}L_\A 
\\&\hspace{5cm} +  \D_{L_\A} (V^{(2,2)}_\A)_{\a} + (V^{(2,2)}_\A)^\rho\D_\rho(L_\A)_{\a} - \dr_\a(V^{(2,2)}_\A)_{L_\A}  \Big).
\end{align*}
For $II$, \eqref{G 1} implies
\begin{align*}
II & = - \frac{1}{2}\sum_\A \sin\pth{\frac{u_\A}{\lambda}}  \pth{ - L_\A^\b \D_{\a} (V^{(2,1)}_\A)_{\b} - \D_{L_\A} (V^{(2,1)}_\A)_{\a}  - \div_{\g_0} V^{(2,1)}_\A \dr_\a u_\A }
\\&\quad  - 2\sum_\A\cos\pth{\frac{2u_\A}{\lambda}}  \pth{- L_\A^\b \D_{\a} (V^{(2,2)}_\A)_{\b} - \D_{L_\A} (V^{(2,2)}_\A)_{\a}  - \div_{\g_0}V^{(2,2)}_\A \dr_\a u_\A }
\\&\quad + \sum_{\substack{v\in\mathcal{I}\\T\in\{\cos,\sin\}}} T\pth{\frac{v}{\la}}\cro{  (V^{(2,1)} + V^{(2,2)}) F^{(1)} }_{\a\b}
\end{align*}
For $III$ and $IV$, \eqref{G 0} simply implies
\begin{align*}
III + IV & = \sum_{\substack{v\in\mathcal{I}\\T\in\{\cos,\sin\}}} T\pth{\frac{v}{\la}}\cro{  (V^{(2,1)} + V^{(2,2)}) F^{(1)} }_{\a\b}.
\end{align*}
Putting everything together, we obtain the lemma, using in particular the fact that all the $\la$ terms in $G(V)$ are oscillating with amplitudes depending only on $V^{(2,1)}_\A$, $V^{(2,2)}_\A$ and $F^{(1)}$ which can be bounded independently of $\la$, thus producing the remainder term $\mathcal{R}_2$.
\end{proof}

\begin{lemma}
The polarization conditions \eqref{pola F21}-\eqref{pola F22} holds on $[0,1]\times \R^3$.
\end{lemma}

\begin{proof}
Thanks to the contracted Bianchi identities \eqref{cBI} and Lemmas \ref{lem divG(Ups)} and \ref{lem divG(V)} we have
\begin{align*}
0 & = \half \sum_\A \sin\pth{\frac{u_\A}{\la}}\Big( 2\D_{L_\A} (V^{(2,1)}_\A)_{\a} +  \div_{\g_0}L_\A (V^{(2,1)}_\A)_{\a}    + (V^{(2,1)}_\A)^\b\D_{[\b}(L_\A)_{\a]}   \Big)
\\&\quad + 2 \sum_\A \cos\pth{\frac{2u_\A}{\la}} \Big(  2 \D_{L_\A} (V^{(2,2)}_\A)_{\a}  + \div_{\g_0}L_\A (V^{(2,2)}_\A)_{\a} + (V^{(2,2)}_\A)^\b\D_{[\b}(L_\A)_{\a]} \Big)
\\&\quad +  \sum_{\substack{v\in\mathcal{I}\\T\in\{\cos,\sin\}}} T\pth{\frac{v}{\la}}\cro{  (V^{(2,1)} + V^{(2,2)}) F^{(1)} + \gg^{(3,e)}}_{\a} + \mathcal{R}_1 + \mathcal{R}_2.
\end{align*}
We multiply this equality by $\sin\pth{\frac{u_\B}{\la}}$ and obtain
\begin{equation}\label{div G a}
\begin{aligned}
0 & = \half \sin^2\pth{\frac{u_\B}{\la}} \pth{   (V^{(2,1)}_\B)_{\a}\div_{\g_0}L_\B +  2\D_{L_\B} (V^{(2,1)}_\B)_{\a} + (V^{(2,1)}_\B)^\b\D_{[\b}(L_\B)_{\a]}   } 
\\&\quad +  \frac{1}{4} \sum_{\A\in\mathcal{A}\setminus\{\B\},\pm} \cos\pth{\frac{u_\A\pm u_\B}{\la}} \cro{\dr^{\leq 1}V^{(2,1)}}_\a +  \sum_{\A,\pm} \sin\pth{\frac{u_\B\pm 2u_\A}{\la}}  \cro{\dr^{\leq 1}V^{(2,2)}}_\a
\\&\quad +  \sum_{\substack{v\in\mathcal{I}\\T\in\{\cos,\sin\}}} T\pth{\frac{v}{\la}}\sin\pth{\frac{u_\B}{\la}}\cro{  (V^{(2,1)} + V^{(2,2)}) F^{(1)} + \gg^{(3,e)} }_{\a} 
\\&\quad +\sin\pth{\frac{u_\B}{\la}}\pth{\mathcal{R}_1 + \mathcal{R}_2}.
\end{aligned}
\end{equation}
We will use the following claim: for all $T:\R\longrightarrow\R$ smooth, $2\pi$-periodic with $\int_0^{2\pi}T=0$ and for all $z\in\mathcal{Z}$ the sequence of functions $\pth{T\pth{\frac{z}{\la}}}_{\la\in(0,1]}$ converges weakly to 0 in $L^2(K)$ when $\la\to 0$ and where $K$ is any compact subset of $\R^3$. This follows from \eqref{spatial} and was already mentioned in Section \ref{section BG}. Here, we use this claim on each $\Si_t$ for $t\in[0,1]$ and $K=\Si_t\cap J^+_0(B_R)$, i.e the support of $V^{(2,1)}_\A$ and $V^{(2,2)}_\A$. Thus, we have the following weak limits in $L^2(K)$ when $\la\to 0$:
\begin{align*}
\sin^2\pth{\frac{u_\B}{\la}} & \rightharpoonup \half,
\\ \cos\pth{\frac{u_\A\pm u_\B}{\la}} & \rightharpoonup 0,
\\ \sin\pth{\frac{u_\B\pm 2u_\A}{\la}}  & \rightharpoonup 0,
\\ T\pth{\frac{v}{\la}}\sin\pth{\frac{u_\B}{\la}} & \rightharpoonup 0,
\end{align*}
where we used that $ T\pth{\frac{v}{\la}}\sin\pth{\frac{u_\B}{\la}}$ with $v\in\mathcal{I}$ and $T\in\{\cos,\sin\}$ is a linear combination of functions of the form $S\pth{\frac{z}{\la}}$ for $z\in\mathcal{Z}$ and $S$ periodic with zero mean. Since $\F_\A$ is uniformly bounded in $H^1_{loc}$ with respect to $\la$ (see \eqref{estim F}), we still get the weak convergence
\begin{align*}
T\pth{\frac{v}{\la}}\sin\pth{\frac{u_\B}{\la}} \cro{\gg^{(3,e)}} \rightharpoonup 0,
\end{align*}
where we again used the fact that the worst term in $\gg^{(3,e)}$ is a linear term in $\F_\A$. Using now \eqref{reste 1} and \eqref{reste 2}, we can take the weak limit in $L^2(K)$ when $\la\to 0$ of the equality \eqref{div G a} and get for all $\B\in\mathcal{A}$ 
\begin{align*}
2\D_{L_\B} (V^{(2,1)}_\B)_{\a}  & =  (V^{(2,1)}_\B)^\b\D_{[\a}(L_\B)_{\b]} - (V^{(2,1)}_\B)_{\a}\div_{\g_0}L_\B.
\end{align*}
Since $V^{(2,1)}_{\B\;\;|_{\Si_0}}=0$ this implies $V^{(2,1)}_\B=0$ everywhere, i.e that \eqref{pola F21} holds on $[0,1]\times \R^3$. We come back to the contracted Bianchi identities (plugging in particular $V^{(2,1)}_\C=0$):
\begin{align*}
0 & =  2 \sum_\A \cos\pth{\frac{2u_\A}{\la}} \Big(  2 \D_{L_\A} (V^{(2,2)}_\A)_{\a}  + \div_{\g_0}L_\A (V^{(2,2)}_\A)_{\a} + (V^{(2,2)}_\A)^\b\D_{[\b}(L_\A)_{\a]} \Big)
\\&\quad +  \sum_{\substack{v\in\mathcal{I}\\T\in\{\cos,\sin\}}} T\pth{\frac{v}{\la}}\cro{ V^{(2,2)} F^{(1)} + \gg^{(3,e)}}_{\a} + \mathcal{R}_1 + \mathcal{R}_2.
\end{align*}
We multiply this equality by $\cos\pth{\frac{2u_\B}{\la}}$:
\begin{align*}
0 & =  2 \cos^2\pth{\frac{2u_\B}{\la}} \Big(  2 \D_{L_\B} (V^{(2,2)}_\B)_{\a}  + \div_{\g_0}L_\B (V^{(2,2)}_\B)_{\a} + (V^{(2,2)}_\B)^\b\D_{[\b}(L_\B)_{\a]} \Big)
\\&\quad +  \sum_{\A\in\mathcal{A}\setminus\{\B\},\pm} \cos\pth{\frac{2(u_\A\pm u_\B)}{\la}}\Big(  2 \D_{L_\A} (V^{(2,2)}_\A)_{\a}  + \div_{\g_0}L_\A (V^{(2,2)}_\A)_{\a} + (V^{(2,2)}_\A)^\b\D_{[\b}(L_\A)_{\a]} \Big)
\\&\quad +  \sum_{\substack{v\in\mathcal{I}\\T\in\{\cos,\sin\}}} T\pth{\frac{v}{\la}}\cos\pth{\frac{2u_\B}{\la}}\cro{V^{(2,2)} F^{(1)}+ \gg^{(3,e)} }_{\a\b} +\cos\pth{\frac{2u_\B}{\la}}\pth{ \mathcal{R}_1 + \mathcal{R}_2}.
\end{align*}
Using the claim as above, taking the weak limit in $L^2(K)$ when $\la\to 0$ of this expression implies
\begin{align*}
2\D_{L_\C} (V^{(2,2)}_\C)_{\a}  & =  (V^{(2,2)}_\C)^\b\D_{[\a}(L_\C)_{\b]} - (V^{(2,2)}_\C)_{\a}\div_{\g_0}L_\C.
\end{align*}
Since $V^{(2,2)}_{\C\;\;|_{\Si_0}}=0$ this implies $V^{(2,2)}_\C=0$ everywhere. This concludes the proof of the lemma.
\end{proof}

\begin{lemma}
The generalised wave gauge condition \eqref{GWC} holds on $[0,1]\times \R^3$.
\end{lemma}

\begin{proof}
Thanks to the previous proposition we have $V^{(2,1)}_\A=V^{(2,2)}_\A=0$ so that \eqref{R V} implies that $R(V)=0$ and that the Einstein tensor of $\g_\la$ is given by
\begin{align}\label{final expression G}
G_{\a\b}(\g_\la) & = \frac{\la^2}{2} \bigg( (\g_\la)_{\rho(\a}\dr_{\b)}\Upsilon^\rho - (\g_\la)_{\a\b}\dr_\rho\Upsilon^\rho  + \Upsilon^\rho \dr_\rho (\g_\la)_{\a\b} - \half \g_\la^{\mu\nu}\Upsilon^\rho \dr_\rho (\g_\la)_{\mu\nu}(\g_\la)_{\a\b} \bigg).
\end{align}
From there, it is standard to deduce from the contracted Bianchi identities that $\Upsilon^\rho$ satisfies a linear system of the form
\begin{align}\label{system Ups}
\tBox_{\g_\la} \Upsilon^\rho & = A^{\rho\mu}_\si \dr_\mu \Upsilon^\si + B^\rho_\si \Upsilon^\si,
\end{align}
with $A^{\rho\mu}_\si$ and $B^{\rho}_\si$ regular coefficients. From Proposition \ref{prop ID main} recall that $\Upsilon^\rho_{|_{\Si_0}}=0$. Moreover again from Proposition \ref{prop ID main} we know that the constraint equations are solved by the induced metric and the second fundamental form so that $G(\g_\la)_{\T_\la \T_\la|_{\Si_0}} = 0$ and $G(\g_\la)_{\T_\la i|_{\Si_0}} = 0$, where $\T_\la$ is the future-directed unit normal for $\g_\la$ to $\Si_0$. From \eqref{final expression G} we have
\begin{align*}
G(\g_\la)_{\T_\la i|_{\Si_0}} & = \frac{\la^2}{2}(\g_\la)_{\rho i} \pth{ \T_\la\Upsilon^\rho - \T_\la^\rho \dr_t\Upsilon^0  },
\\ G(\g_\la)_{\T_\la \T_\la|_{\Si_0}} & = \frac{\la^2}{2} \pth{ 2(\T_\la)_\rho \T_\la\Upsilon^\rho + \dr_t\Upsilon^0  },
\end{align*}
where we used $\Upsilon^\rho_{|_{\Si_0}}=0$. The constraint equations being solved thus implies
\begin{align}
\T_\la\Upsilon^\rho & = \T_\la^\rho \dr_t\Upsilon^0 , \label{C1}
\\  \dr_t\Upsilon^0 & = - 2(\T_\la)_\rho \T_\la\Upsilon^\rho .\label{C2}
\end{align}
Thanks to \eqref{C1} and $\g_\la(\T_\la,\T_\la)=-1$, \eqref{C2} implies $\dr_t\Upsilon^0  = 0 $ so that \eqref{C1} becomes in turn $\T_\la\Upsilon^\rho=0$. Therefore, $\Upsilon^\rho$ satisfies the linear wave system \eqref{system Ups} and $\Upsilon^\rho_{|_{\Si_0}}=\T_\la\Upsilon^\rho_{|_{\Si_0}}=0$, which concludes the proof of the proposition.
\end{proof}

This concludes the proof of Theorem \ref{theo evol}, after the identifications
\begin{align*}
\sum_{\substack{w\in\mathcal{N}_1\cup\mathcal{N}_2\cup\mathcal{I}_2\\T\in\{\cos,\sin\}}} T\pth{\frac{w}{\la}}F^{(2,w,T)}_\la & =   \sum_\A \pth{ \sin\pth{\frac{u_\A}{\la}} \pth{ \F_\A + F^{(2,1)}_\A } + \cos\pth{\frac{2u_\A}{\la}} F^{(2,2)}_\A } 
\\&\quad + \la^2 \sum_{\A\neq \B,\pm}   \cos\pth{\frac{u_\A\pm u_\B}{\la}} F^{(2,\pm)}_{\A\B},
\\ \widetilde{\g_\la} & =   \h_\la  + \la \sum_{\substack{u\in\mathcal{N}\\T\in\{\cos,\sin\}}}T\pth{\frac{u}{\la}} \gg^{(3,h)}_{u,T} + \la \sum_{\substack{v\in\mathcal{I}\\T\in\{\cos,\sin\}}} T\pth{\frac{v}{\la}} \gg^{(3,e)}_{v,T}.
\end{align*}
The estimates \eqref{estim theo 1}-\eqref{estim theo 3} can then be derived from \eqref{estim F1 F21 F22}-\eqref{estim F} and \eqref{estim h}.

\section{Declaration}

This research received no specific grant from any funding agency in the public, private, or not-for-profit sectors.


\appendix

\section{Proofs of Section \ref{section expansion Ricci}}\label{appendix expansion Ricci}

In the following sections we will use the following trigonometric identities:
\begin{align}
\cos(a)\cos(b)  & = \half \sum_{\pm} \cos(a\pm b) ,\label{trigo 1}
\\ \sin(a) \sin(b)  & = \half \sum_{\pm}\mp \cos(a\pm b) , \label{trigo 2}
\\ \sin(a)\cos(b)  & = \half \sum_{\pm}\sin(a\pm b).\label{trigo 3}
\end{align}
We will also use the following identities on double sums:
\begin{align}
\sum_{\A\neq \B,\pm}  \cos\pth{\frac{u_\A\pm u_\B}{\la}} F^{(\pm)}_{\A\B} & = \half \sum_{\A\neq \B,\pm}  \cos\pth{\frac{u_\A\pm u_\B}{\la}}\pth{ F^{(\pm)}_{\A\B} + F^{(\pm)}_{\B\A}}, \label{double sum 1}
\\ \sum_{\A\neq \B,\pm} \sin\pth{\frac{u_\A\pm u_\B}{\la}} F^{(\pm)}_{\A\B} & = \half \sum_{\A\neq \B,\pm} \sin\pth{\frac{u_\A\pm u_\B}{\la}}\pth{ F^{(\pm)}_{\A\B} \pm F^{(\pm)}_{\B\A}}, \label{double sum 2}
\end{align}
which simply follow from 
\begin{align*}
\sum_{\A\neq \B} T_{\A\B} = \half \sum_{\A\neq \B}(T_{\A\B} + T_{\B\A}).
\end{align*}
Moreover, from \eqref{expansion metric} we can obtain the inverse of $\g_\lambda$: 
\begin{align*}
\g_\la^{\mu\nu} & =  \g_0^{\mu\nu} + \la (\g_\la^{\mu\nu})^{(1)}  + \la^2 (\g_\la^{\mu\nu})^{(2)}  + \la^3 (\g_\la^{\mu\nu})^{(\geq 3)} , 
\end{align*}
where
\begin{align}
(\g_\lambda^{\mu\nu})^{(1)}  & = - \sum_\A \cos\pth{\frac{u_\A}{\lambda}}(F^{(1)}_\A)^{\mu\nu},\label{inverse 1}
\\ (\g_\lambda^{\mu\nu})^{(2)} & = - \h_\lambda^{\mu\nu} -  \sum_\A \pth{\sin\left(\frac{u_\A}{\lambda} \right) \left( (\F_\A)^{\mu\nu} + (F^{(2,1)}_\A)^{\mu\nu} \right) + \cos\left( \frac{2u_\A}{\lambda}\right) (F^{(2,2)}_\A)^{\mu\nu}}\label{inverse 2}
\\&\quad +  \sum_{\A\neq\B,\pm} \cos\left(\frac{u_\A\pm u_\B}{\lambda}\right)\pth{  \half (F^{(1)}_\A)^{\nu\sigma} (F^{(1)}_\B)^\mu_{\sigma}  - (F^{(2,\pm)}_{\A\B} )^{\mu\nu}}\non
\\&\quad +  \sum_{\A} \cos^2\pth{\frac{u_\A}{\lambda}}  (F^{(1)}_\A)^{\nu\sigma} (F^{(1)}_\A)^\mu_{\sigma},\non
\end{align}
and where on the RHS of \eqref{inverse 1} and \eqref{inverse 2} the indexes are moved with respect to the background metric $\g_0$.

\subsection{Proof of Lemma \ref{lemma W}}\label{appendix proof lemma W}

The expansion of the quasi-linear wave operator $\tBox_{\g_\lambda}(\g_\lambda)_{\alpha\beta}$ follows from a systematic use of the exact formula
\begin{equation} \label{Box T f}
\begin{aligned}
\tBox_g \left( T\left( \frac{v}{\lambda} \right) f \right) & = \frac{1}{\la^2}T''\pth{ \frac{v}{\la}} g^{-1}(\d v, \d v) f  +  \frac{1}{\la}T'\pth{ \frac{v}{\la}} \pth{ 2g^{\mu\nu}\dr_\mu v \dr_\nu f + (\tBox_g v)f  }   + T\pth{\frac{v}{\la}}\tBox_g f ,  
\end{aligned}
\end{equation}
which holds for any Lorentzian metric $g$, real function $T$ and scalar functions on the manifold $v$ and $f$. From \eqref{inverse 1} and \eqref{inverse 2} we can also obtain the expansion
\begin{equation}\label{expansion eikonal}
\begin{aligned}
\g_\la^{-1}(\d u_\A,\d u_\A) & = - \la \sum_\B \cos\pth{\frac{u_\B}{\lambda}}(F^{(1)}_\B)_{L_\A L_\A}
\\&\quad - \la^2  (\h_\la)_{L_\A L_\A}  - \la^2  \sum_\B \sin\pth{\frac{u_\B}{\la}} \pth{ (\F_\B)_{L_\A L_\A}  + (F^{(2,1)}_\B)_{L_\A L_\A}  } 
\\&\quad  - \la^2  \sum_\B \cos\pth{ \frac{2u_\B}{\la}} (F^{(2,2)}_\B)_{L_\A L_\A} 
\\&\quad + \la^2  \sum_{\B\neq\C,\pm} \cos\pth{\frac{u_\B\pm u_\C}{\la}} \pth{  \half (F^{(1)}_\B)^{\si}_{L_\A} (F^{(1)}_\C)_{L_\A\si}   - (F^{(2,\pm)}_{\B\C} )_{L_\A L_\A} }
\\&\quad + \la^2  \sum_{\B} \cos^2\pth{\frac{u_\B}{\la}}  (F^{(1)}_\B)^{\si}_{L_\A} (F^{(1)}_\B)_{L_\A\si}  + \la^3 (\g_\la^{\mu\nu})^{(\geq 3)} \dr_\mu u_\A\dr_\nu u_\A.
\end{aligned}
\end{equation}
Thanks to \eqref{expansion metric} we have 
\begin{equation}\label{exp box g}
\begin{aligned}
\tBox_{\g_\la}(\g_\la)_{\a\b} & = \tBox_{\g_\la}(\g_0)_{\a\b} + \la \sum_\A  \tBox_{\g_\la}\pth{\cos\pth{ \frac{u_\A}{\la} } (F^{(1)}_\A)_{\a\b}  }
\\&\quad + \la^2 \sum_\A \tBox_{\g_\la}\pth{\sin\pth{\frac{u_\A}{\la}} \pth{ (\F_\A)_{\a\b}  + (F^{(2,1)}_\A)_{\a\b} }} 
\\&\quad + \la^2 \sum_\A \tBox_{\g_\la}\pth{\cos\pth{ \frac{2u_\A}{\la}} (F^{(2,2)}_\A)_{\a\b} }
\\&\quad + \la^2 \sum_{\A\neq \B,\pm}   \tBox_{\g_\la}\pth{\cos\pth{\frac{u_\A\pm u_\B}{\la}} (F^{(2,\pm)}_{\A\B})_{\a\b} }
\\&\quad + \la^2 \tBox_{\g_\la}(\h_\la)_{\a\b} + \la^3 \sum_{\substack{u\in\mathcal{N}\\T\in\{\cos,\sin\}}}\tBox_{\g_\la}\pth{T\pth{\frac{u}{\la}} (\gg^{(3,h)}_{u,T})_{\a\b}} 
\\&\quad + \la^3 \sum_{\substack{v\in\mathcal{I}\\T\in\{\cos,\sin\}}} \tBox_{\g_\la}\pth{T\pth{\frac{v}{\la}} (\gg^{(3,e)}_{v,T})_{\a\b}}.
\end{aligned}
\end{equation}
We expand each term in this expression, using without mention \eqref{Box T f}, \eqref{expansion eikonal} and more generally \eqref{inverse 1} and \eqref{inverse 2}. We will also use without mention the assumptions \eqref{assumption}. For the first term, we simply have
\begin{align}
\tBox_{\g_\la}(\g_0)_{\a\b} & = \tBox_{\g_0}(\g_0)_{\a\b} + \la\cro{ (\g_\la^{-1})^{(\geq 1)} \dr^2\g_0}_{\a\b}. \label{wave proof 1}
\end{align}
For the second term, we first apply \eqref{Box T f}
\begin{align*}
\la \sum_\A  \tBox_{\g_\la}\pth{\cos\pth{ \frac{u_\A}{\la} } (F^{(1)}_\A)_{\a\b}  } & =  - \frac{1}{\la}\sum_\A \cos\pth{\frac{u_\A}{\la}} \g_\la^{-1}(\d u_\A, \d u_\A) (F^{(1)}_\A)_{\a\b}
\\&\quad - \sum_\A \sin\pth{\frac{u_\A}{\la}} \pth{ 2\g_\la^{\mu\nu}\dr_\mu u_\A \dr_\nu (F^{(1)}_\A)_{\a\b} + (\tBox_{\g_\la} u_\A)(F^{(1)}_\A)_{\a\b}  }  
\\&\quad + \la \sum_\A  \cos\pth{\frac{u_\A}{\la}}\tBox_{\g_\la} (F^{(1)}_\A)_{\a\b}.
\end{align*}
Using now \eqref{expansion eikonal}, \eqref{trigo 1}, \eqref{trigo 3} and \eqref{double sum 1} we obtain
\begin{equation}\label{wave proof 2}
\begin{aligned}
&\la \sum_\A  \tBox_{\g_\la}\pth{\cos\pth{ \frac{u_\A}{\la} } (F^{(1)}_\A)_{\a\b}  } 
\\& = - \sum_\A \sin\pth{\frac{u_\A}{\la}} \pth{ -2 L_\A (F^{(1)}_\A)_{\a\b} + (\tBox_{\g_0} u_\A)(F^{(1)}_\A)_{\a\b}  }  
\\&\quad + \frac{1}{4} \sum_{\A\neq\B,\pm} \cos\pth{\frac{u_\A\pm u_\B}{\la}} \pth{ (F^{(1)}_\B)_{L_\A L_\A}(F^{(1)}_\A)_{\a\b} + (F^{(1)}_\A)_{L_\B L_\B}(F^{(1)}_\B)_{\a\b}  }
\\&\quad +\la\sum_\A \cos\pth{\frac{u_\A}{\la}} \pth{  (\h_\la)_{L_\A L_\A}  (F^{(1)}_\A)_{\a\b}  + \cro{\pth{F^{(2,\pm)} + (F^{(1)})^2 }F^{(1)} + \dr^2 F^{(1)}}_{\a\b}}
\\&\quad + \frac{\la}{2} \sum_{\A} \sin\pth{\frac{2u_\A}{\la}} (\F_\A)_{L_\A L_\A} (F^{(1)}_\A)_{\a\b} 
\\&\quad + \la \sum_\A \sin\pth{\frac{2u_\A}{\la}} \cro{ \pth{ F^{(2,1)} + \dr F^{(1)} + F^{(1)} }F^{(1)} }_{\a\b}
\\&\quad + \la\sum_{\A} \cos\pth{\frac{u_\A}{\la}} \cos\pth{ \frac{2u_\A}{\la}} (F^{(2,2)}_\A)_{L_\A L_\A} (F^{(1)}_\A)_{\a\b}  
\\&\quad + \la \sum_{\substack{v\in\mathcal{I}\\T\in\{\cos,\sin\}}}T\pth{\frac{v}{\la}} \Big\{ \pth{ \F + F^{(2,1)} + F^{(2,2)} + F^{(2,\pm)} }F^{(1)}
\\&\hspace{6cm} + \pth{ (F^{(1)})^2 + F^{(1)} + \dr F^{(1)} }F^{(1)} \Big\}_{\a\b}
\\&\quad + \la^2 \cro{ (\g_\la^{-1})^{(\geq 2)} \dr^{\leq 2} F^{(1)} }_{\a\b}^\osc.
\end{aligned}
\end{equation}
For the third, fourth and fifth term in \eqref{exp box g} we simply apply \eqref{Box T f}. We obtain
\begin{equation}\label{wave proof 3}
\begin{aligned}
&\la^2 \sum_\A \tBox_{\g_\la}\pth{\sin\pth{\frac{u_\A}{\la}} \pth{ (\F_\A)_{\a\b}  + (F^{(2,1)}_\A)_{\a\b} }} 
\\& =   \la \sum_\A \cos\pth{ \frac{u_\A}{\la}} \pth{ -2L_\A  + \tBox_{\g_0} u_\A  } \pth{ (\F_\A)_{\a\b}  + (F^{(2,1)}_\A)_{\a\b} } 
\\&\quad +\la \sum_{\substack{v\in\mathcal{I}\\T\in\{\cos,\sin\}}}T\pth{\frac{v}{\la}} \cro{ \pth{ \F  + F^{(2,1)} } F^{(1)} }_{\a\b} + \la^2 \sum_\A \sin\pth{\frac{u_\A}{\la}}\tBox_{\g_\la} (\F_\A)_{\a\b} 
\\&\quad + \la^2  \cro{ (\g_\la^{-1})^{(\geq 1)} \dr^{\leq 1} \pth{ \F  + F^{(2,1)}  }+ \g_\la^{-1} \dr^2F^{(2,1)}}_{\a\b}^\osc,
\end{aligned}
\end{equation}
and
\begin{equation}\label{wave proof 4}
\begin{aligned}
& \la^2 \sum_\A \tBox_{\g_\la}\pth{\cos\pth{ \frac{2u_\A}{\la}} (F^{(2,2)}_\A)_{\a\b} }
\\& =  -2 \la \sum_\A \sin\pth{ \frac{2u_\A}{\la}} \pth{ -2L_\A  + \tBox_{\g_\la} u_\A  } (F^{(2,2)}_\A)_{\a\b}
\\&\quad + \la \sum_{\substack{v\in\mathcal{I}\\T\in\{\cos,\sin\}}}T\pth{\frac{v}{\la}} \cro{ F^{(2,2)} F^{(1)} }_{\a\b} + \la^2  \cro{ \g_\la^{-1} \dr^{\leq 2}F^{(2,2)}}_{\a\b}^\osc,
\end{aligned}
\end{equation}
and
\begin{equation}\label{wave proof 5}
\begin{aligned}
&\la^2 \sum_{\A\neq \B,\pm}   \tBox_{\g_\la}\pth{\cos\pth{\frac{u_\A\pm u_\B}{\la}} (F^{(2,\pm)}_{\A\B})_{\a\b} }
\\& = -  \sum_{\A\neq \B,\pm} \cos\pth{ \frac{u_\A\pm u_\B}{\la}} \g_0^{-1}(\d (u_\A\pm u_\B), \d (u_\A\pm u_\B)) (F^{(2,\pm)}_{\A\B})_{\a\b}
\\&\quad  + \la \sum_{\A\neq \B,\pm_1} \cos\pth{ \frac{u_\A}{\la}}  (F^{(1)}_\B)_{L_\A L_\A} (F^{(2,\pm_1)}_{\A\B})_{\a\b}
\\&\quad +  \la \sum_{\substack{v\in\mathcal{I}\\T\in\{\cos,\sin\}}}T\pth{\frac{v}{\la}} \cro{\dr^{\leq 1}F^{(2,\pm)} + F^{(1)} F^{(2,\pm)}}_{\a\b}  + \la^2 \cro{ \g_\la^{-1}\dr^{\leq 2} F^{(2,\pm)}}_{\a\b}^\osc.
\end{aligned}
\end{equation}
The sixth term is left as it is. For the seventh and eighth we apply \eqref{Box T f} and obtain
\begin{equation}\label{wave proof 7}
\begin{aligned}
 \la^3 \sum_{\substack{u\in\mathcal{N}\\T\in\{\cos,\sin\}}}\tBox_{\g_\la}\pth{T\pth{\frac{u}{\la}} (\gg^{(3,h)}_{u,T})_{\a\b}} & = \la^2  \cro{ \g_\la^{-1}\dr^{\leq 1}\gg^{(3,h)} }_{\a\b}^\osc 
 \\&\quad  + \la^3 \sum_{\substack{u\in\mathcal{N}\\T\in\{\cos,\sin\}}}T\pth{\frac{u}{\la}}\tBox_{\g_\la} (\gg^{(3,h)}_{u,T})_{\a\b},
\end{aligned}
\end{equation}
and
\begin{equation}\label{wave proof 8}
\begin{aligned}
 \la^3 \sum_{\substack{v\in\mathcal{I}\\T\in\{\cos,\sin\}}} \tBox_{\g_\la}\pth{T\pth{\frac{v}{\la}} (\gg^{(3,e)}_{v,T})_{\a\b}}& = - \la \sum_{\substack{v\in\mathcal{I}\\T\in\{\cos,\sin\}}}T\pth{ \frac{v}{\la}} \g_0^{-1}(\d v, \d v) (\gg^{(3,e)}_{v,T})_{\a\b}
\\&\quad +  \la^2\cro{ \g_\la^{-1}\dr^{\leq 1}\gg^{(3,e)} }_{\a\b}^\osc 
\\&\quad + \la^3 \sum_{\substack{v\in\mathcal{I}\\T\in\{\cos,\sin\}}}T\pth{\frac{v}{\la}}\tBox_{\g_\la} (\gg^{(3,e)}_{v,T})_{\a\b},
\end{aligned}
\end{equation}
where we used that $T''=-T$ if $T\in\{\cos,\sin\}$. Collecting \eqref{wave proof 1}-\eqref{wave proof 8} concludes the proof of Lemma \ref{lemma W}.

\subsection{Proof of Lemma \ref{lemma P}}\label{appendix proof lemma P}

We recall \eqref{P dg dg}, which gives the expression of $P_{\alpha\beta}(\g_\la)(\dr \g_\lambda ,\dr \g_\lambda)$. It is of the form $\g_\la^{-1}\g_\la^{-1}\dr\g_\la\dr\g_\la$ so that schematically we have
\begin{align*}
\pth{P_{\alpha\beta}(\g_\la)(\dr \g_\lambda ,\dr \g_\lambda)}^{(0)} & = \g_0^{-1}\g_0^{-1} \pth{\dr\g_\la}^{(0)} \pth{\dr\g_\la}^{(0)}.
\end{align*}
Thanks to \eqref{expansion metric} we have
\begin{align}\label{dg0}
\pth{ \dr_\a (\g_\la)_{\mu\nu}}^{(0)} & = \dr_\a (\g_0)_{\mu\nu} - \sum_\A \sin\pth{\frac{u_\A}{\la}}\dr_\a u_\A (F^{(1)}_\A)_{\mu\nu}.
\end{align}
With \eqref{P dg dg}, \eqref{trigo 2} and \eqref{double sum 1} this implies
\begin{equation}\label{P0 proof}
\begin{aligned}
&\pth{P_{\alpha\beta}(\g_\la)(\dr \g_\lambda ,\dr \g_\lambda)}^{(0)} 
\\& = P_{\alpha\beta}(\g_0)(\dr \g_0 ,\dr \g_0) - \frac{1}{4}  \sum_{\A} \dr_\alpha u_\A \dr_\beta u_\A  \left| F^{(1)}_\A \right|^2_{\g_0} 
\\&\quad  + \sum_\A \sin\pth{\frac{u_\A}{\la}} \bigg(- 2\Ga(\g_0)^\mu_{(\a \rho} \dr^\rho u_\A (F^{(1)}_\A)_{\beta)\mu} - (F^{(1)}_\A)^{\mu\nu} \dr_{(\a}u_\A \pth{ \dr_\mu (\g_0)_{\b)\nu}  - \frac{1}{2} \dr_{\b)} (\g_0)_{\mu\nu} }  \bigg)
\\&\quad + \frac{1}{4}  \sum_{\A} \cos\pth{\frac{2u_\A}{\la}} \dr_\alpha u_\A \dr_\beta u_\A  \left| F^{(1)}_\A \right|^2_{\g_0} 
\\&\quad + \frac{1}{4} \sum_{\A\neq\B,\pm} \cos\pth{\frac{u_\A\pm u_\B}{\la}}  \bigg( \pm\dr_{(\alpha}u_\A (F^{(1)}_\A)_{L_\B\si} (F^{(1)}_\B)_{\beta)}^\si  \pm\dr_{(\alpha}u_\B (F^{(1)}_\B)_{L_\A\si} (F^{(1)}_\A)_{\beta)}^\si  
\\&\hspace{5cm} \pm\frac{1}{2} \dr_{(\alpha} u_\A\dr_{\beta)} u_\B   \left| F^{(1)}_\A\cdot F^{(1)}_\B \right|_{\g_0}  \pm (F^{(1)}_\A)_{(\alpha L_\B} (F^{(1)}_\B)_{\beta) L_\A}
\\&\hspace{8.5cm}  \mp \g_0^{-1}(\d u_\A,\d u_\B)  (F^{(1)}_\A)_{(\alpha}^\nu (F^{(1)}_\B)_{\nu\beta)}   \bigg),
\end{aligned}
\end{equation}
where we also used \eqref{assumption}. Moreover we also have schematically
\begin{align*}
\pth{P_{\alpha\beta}(\g_\la)(\dr \g_\lambda ,\dr \g_\lambda)}^{(1)} & =\underbrace{ \g_0^{-1}\g_0^{-1} \pth{\dr\g_\la}^{(0)} \pth{\dr\g_\la}^{(1)}}_{\vcentcolon = A} + \underbrace{\pth{\g_\la^{-1}}^{(1)}\g_0^{-1} \pth{\dr\g_\la}^{(0)} \pth{\dr\g_\la}^{(0)}}_{\vcentcolon=B} .
\end{align*}
We start with $A_{\a\b}$. Thanks to \eqref{expansion metric} we have
\begin{equation}\label{dg1}
\begin{aligned}
\pth{ \dr_\a (\g_\la)_{\mu\nu}}^{(1)} & = \sum_\A \cos\pth{\frac{u_\A}{\la}} \dr_\a (F^{(1)}_\A)_{\mu\nu}  +  \sum_\A  \cos\left(\frac{u_\A}{\lambda} \right) \dr_\a u_\A \left( (\F_\A)_{\mu\nu} + (F^{(2,1)}_\A)_{\mu\nu} \right)
\\&\quad -2 \sum_\A  \sin\left( \frac{2u_\A}{\lambda}\right) \dr_\a u_\A (F^{(2,2)}_\A)_{\mu\nu} 
\\&\quad - \sum_{\A\neq \B,\pm}   \sin\left(\frac{u_\A\pm u_\B}{\lambda}\right) \dr_\a(u_\A \pm u_\B)  (F^{(2,\pm)}_{\A\B})_{\mu\nu} .
\end{aligned}
\end{equation}
Now, \eqref{P dg dg}, \eqref{dg0} and \eqref{dg1} imply 
\begin{equation}\label{P1 A proof}
\begin{aligned}
A_{\a\b} & =  2 \sum_\A  \cos\pth{\frac{u_\A}{\la}}\Big( \Ga(\g_0)^{\mu}_{(\a\rho}  \dr^\rho u_\A\pth{(\F_\A)_{\beta)\mu} +  (F^{(2,1)}_\A)_{\beta)\mu}}   +  \cro{ \dr F^{(1)} + F^{(1)} F^{(2,\pm)} }_{\a\b}  \Big)
\\&\quad - 4 \sum_\A  \sin\pth{ \frac{2u_\A}{\la}} \pth{ \Ga(\g_0)^{\mu}_{(\a\rho}  \dr^\rho u_\A (F^{(2,2)}_\A)_{\beta)\mu} + \cro{F^{(1)} \dr F^{(1)}}_{\a\b}}
\\&\quad  + \sum_{\substack{u\in\mathcal{N}\\T\in\{\cos,\sin\}}}T\pth{\frac{u}{\la}}\dr_{(\a}u\cro{\pth{1+F^{(1)}}\pth{\F + F^{(2,1)} +  F^{(2,2)}}}_{\b)} 
\\&\quad + \sum_{\substack{v\in\mathcal{I}\\T\in\{\cos,\sin\}}}T\pth{ \frac{v}{\la}} \cro{F^{(2,\pm)} + F^{(1)} \pth{\dr F^{(1)} + \F + F^{(2,1)} + F^{(2,2)} + F^{(2,\pm)} } }_{\a\b}.
\end{aligned}
\end{equation}
It remains to compute $B_{\a\b}$:
\begin{equation}\label{P1 B proof}
\begin{aligned}
B_{\a\b} & =\sum_\A  \cos\pth{\frac{u_\A}{\la}}\cro{F^{(1)}+(F^{(1)})^3}_{\a\b} + \sum_\A \sin\pth{\frac{2u_\A}{\la}}\cro{ (F^{(1)})^2 }_{\a\b} 
\\&\quad +  \sum_{\substack{v\in\mathcal{I}\\T\in\{\cos,\sin\}}}T\pth{ \frac{v}{\la}} \cro{ (F^{(1)})^2+(F^{(1)})^3 }_{\a\b} + \sum_{\substack{u\in\mathcal{N}\\T\in\{\cos,\sin\}}}T\pth{\frac{u}{\la}}\dr_{(\a}u\cro{(F^{(1)})^3}_{\b)}.
\end{aligned}
\end{equation}
Collecting \eqref{P0 proof}, \eqref{P1 A proof} and \eqref{P1 B proof} concludes the proof of Lemma \ref{lemma P}.

\subsection{Proof of Lemma \ref{lemma H}}\label{appendix proof lemma H}

We recall that 
\begin{align*}
H^\rho & = \g_\la^{\rho\sigma} \g_\la^{\mu\nu} \pth{ \dr_\mu (\g_\la)_{\sigma\nu} - \half \dr_\sigma (\g_\la)_{\mu\nu} }.
\end{align*}
We will first expand the derivatives and then use the expansion of the inverse (see \eqref{inverse 1}-\eqref{inverse 2}) to obtain $H^\rho$. Thanks to \eqref{expansion metric} we obtain
\begin{align*}
\dr_\mu (\g_\la)_{\sigma\nu} - \half \dr_\sigma (\g_\la)_{\mu\nu} & = D_{\mu\nu\si}^{(0)} + \la D_{\mu\nu\si}^{(1)} + \la^2 D_{\mu\nu\si}^{(2)} + \la^3 D_{\mu\nu\si}^{(3)},
\end{align*}
where
\begin{align*}
D_{\mu\nu\si}^{(0)} & = \dr_\mu (\g_0)_{\sigma\nu} - \half \dr_\sigma (\g_0)_{\mu\nu} -  \sum_\A  \sin\left( \frac{u_\A}{\lambda} \right) \pth{ \dr_\mu u_\A (F^{(1)}_\A)_{\sigma\nu} - \half \dr_\sigma u_\A (F^{(1)}_\A)_{\mu\nu}  }  ,
\end{align*}
\begin{align*}
D_{\mu\nu\si}^{(1)} & =   \sum_\A  \cos\left( \frac{u_\A}{\lambda} \right)\pth{ \dr_\mu (F^{(1)}_\A)_{\sigma\nu} - \half \dr_\sigma (F^{(1)}_\A)_{\mu\nu} }   
\\&\quad +   \sum_\A \cos\left(\frac{u_\A}{\lambda} \right)\pth{\dr_\mu u_\A(\F_\A + F^{(2,1)}_\A)_{\sigma\nu} - \half \dr_\sigma u_\A(\F_\A + F^{(2,1)}_\A)_{\mu\nu}   }
\\&\quad -2  \sum_\A  \sin\left( \frac{2u_\A}{\lambda}\right) \pth{\dr_\mu u_\A(F^{(2,2)}_\A)_{\sigma\nu} - \half \dr_\sigma u_\A(F^{(2,2)}_\A)_{\mu\nu}}  
\\&\quad - \sum_{\A\neq \B,\pm}   \sin\left(\frac{u_\A\pm u_\B}{\lambda}\right)\pth{ \dr_\mu (u_\A\pm u_\B) (F^{(2,\pm)}_{\A\B})_{\sigma\nu} - \half \dr_\sigma (u_\A\pm u_\B) (F^{(2,\pm)}_{\A\B})_{\mu\nu} },
\end{align*}
\begin{align*}
D_{\mu\nu\si}^{(2)} & =   \sum_\A  \sin\left(\frac{u_\A}{\lambda} \right) \pth{\dr_\mu (\F_\A)_{\sigma\nu} - \half \dr_\sigma (\F_\A)_{\mu\nu} } 
\\&\quad +  \sum_{\substack{u\in\mathcal{N}\\T\in\{\cos,\sin\}}} T\pth{\frac{u}{\la}} \cro{ \dr F^{(2,1)} + \dr F^{(2,2)} }_{\mu\nu\si} +  \sum_{\substack{v\in\mathcal{I}\\T\in\{\cos,\sin\}}}T\pth{\frac{v}{\la}} \cro{ \dr F^{(2,\pm)} }_{\mu\nu\si}
\\&\quad +   \dr_\mu (\h_\lambda)_{\sigma\nu} - \half \dr_\sigma (\h_\lambda)_{\mu\nu} 
 +  \sum_{\substack{u\in\mathcal{N}\\T\in\{\cos,\sin\}}}T'\pth{\frac{u}{\la}}\pth{ \dr_\mu u(\gg^{(3,h)}_{u,T})_{\sigma\nu} - \half \dr_\sigma u(\gg^{(3,h)}_{u,T})_{\mu\nu} } 
\\&\quad +  \sum_{\substack{v\in\mathcal{I}\\T\in\{\cos,\sin\}}}T'\pth{\frac{v}{\la}}\pth{ \dr_\mu v(\gg^{(3,e)}_{v,T})_{\sigma\nu} - \half \dr_\sigma v(\gg^{(3,e)}_{v,T})_{\mu\nu} } ,
\end{align*}
and
\begin{align*}
D_{\mu\nu\si}^{(3)} & =   \sum_{\substack{u\in\mathcal{N}\\T\in\{\cos,\sin\}}}T\pth{\frac{u}{\la}} \pth{\dr_\mu (\gg^{(3,h)}_{u,T})_{\sigma\nu} - \half \dr_\sigma (\gg^{(3,h)}_{u,T})_{\mu\nu}}
\\&\quad +  \sum_{\substack{v\in\mathcal{I}\\T\in\{\cos,\sin\}}}T\pth{\frac{v}{\la}} \pth{\dr_\mu (\gg^{(3,e)}_{v,T})_{\sigma\nu} - \half \dr_\sigma (\gg^{(3,e)}_{v,T})_{\mu\nu}}.
\end{align*}
We have 
\begin{align}\label{H rho somme}
H^\rho & = \g_\la^{\rho\sigma} \g_\la^{\mu\nu} \pth{ D_{\mu\nu\si}^{(0)} + \la D_{\mu\nu\si}^{(1)} + \la^2 D_{\mu\nu\si}^{(2)} + \la^3 D_{\mu\nu\si}^{(3)} } .
\end{align}
Using the background wave coordinate condition \eqref{wave condition BG}, \eqref{assumption}, \eqref{trigo 3} and \eqref{double sum 2} we first obtain
\begin{equation}\label{ggD0}
\begin{aligned}
&\g_\la^{\rho\sigma} \g_\la^{\mu\nu} D_{\mu\nu\si}^{(0)} 
\\& = - \la \sum_\A\cos\pth{\frac{u_\A}{\la}} (F^{(1)}_\A)^{\mu\nu} \g_0^{\rho\sigma} \pth{\dr_\mu (\g_0)_{\sigma\nu} - \half \dr_\sigma (\g_0)_{\mu\nu} } 
\\&\quad - \frac{\la}{4}  \sum_{\A}\sin\pth{\frac{2u_\A}{\la}}  \dr^\rho u_\A  \left| F^{(1)}_\A\right|^2_{\g_0} 
\\&\quad - \frac{\la}{4}  \sum_{\A\neq\B,\pm} \sin\pth{ \frac{u_\A\pm u_\B}{\la}} \bigg( (F^{(1)}_\B)^{\nu}_{L_\A} (F^{(1)}_\A)_{\nu}^\rho \pm (F^{(1)}_\A)^{\nu}_{L_\B} (F^{(1)}_\B)_{\nu}^\rho 
\\&\hspace{7cm} + \half \dr^\rho (u_\A\pm u_\B) \left| F^{(1)}_\A\cdot F^{(1)}_\B\right|_{\g_0}    \bigg)
\\&\quad - \la^2 \g_0^{\rho\sigma} \h_\la^{\mu\nu} \Bigg(\dr_\mu (\g_0)_{\sigma\nu} - \half \dr_\sigma (\g_0)_{\mu\nu}  -  \sum_\A  \sin\left( \frac{u_\A}{\lambda} \right) \pth{ \dr_\mu u_\A (F^{(1)}_\A)_{\sigma\nu} - \half \dr_\sigma u_\A (F^{(1)}_\A)_{\mu\nu}  }\Bigg)
\\&\quad + \la^2 \sum_{\substack{w\in \mathcal{N}\cup\mathcal{I}\\T\in\{1,\cos,\sin\}}}T\pth{ \frac{w}{\la}} \cro{  \pth{ \F + F^{(2,1)} + F^{(2,2)} + F^{(2,\pm)} + (F^{(1)})^2 }\pth{1+F^{(1)}} }^\rho
\\&\quad + \la^3  \cro{ \pth{\g_\la^{-1} (\g_\la)^{(\geq 3)} + (\g_\la)^{(\geq 1)} (\g_\la)^{(\geq 2)} } \pth{ 1 + F^{(1)}} }^{\rho,\osc}.
\end{aligned}
\end{equation}
For the remaining terms we obtain
\begin{equation}\label{ggD1}
\begin{aligned}
&\g_\la^{\rho\sigma} \g_\la^{\mu\nu} D_{\mu\nu\si}^{(1)}  
\\& =   \sum_\A  \cos\left( \frac{u_\A}{\lambda} \right)\g_0^{\rho\sigma} \bigg(\Pol{\F_\A + F^{(2,1)}_\A}{u_\A}_\si 
 +  \g_0^{\mu\nu} \pth{ \dr_\mu (F^{(1)}_\A)_{\sigma\nu} - \half \dr_\sigma (F^{(1)}_\A)_{\mu\nu} }   \bigg)
\\&\quad -2  \sum_\A  \sin\left( \frac{2u_\A}{\lambda}\right)\g_0^{\rho\sigma} \Pol{F^{(2,2)}_\A}{u_\A}_\si - \sum_{\A\neq \B,\pm}   \sin\left(\frac{u_\A\pm u_\B}{\lambda}\right)\g_0^{\rho\sigma} \Pol{F^{(2,\pm)}_{\A\B}}{u_\A\pm u_\B}_\si
\\&\quad +\la \sum_{\substack{w\in \mathcal{N}\cup\mathcal{I}\\T\in\{1,\cos,\sin\}}}T\pth{ \frac{w}{\la}} \cro{  F^{(1)} \pth{ \dr F^{(1)}  + \F + F^{(2,1)} + F^{(2,2)} + F^{(2,\pm)} } }^\rho
\\&\quad +\la^2 \Big\{ \pth{ (\g_\la^{-1})^{(\geq 2)} + (\g_\la^{-1})^{(\geq 1)}(\g_\la^{-1})^{(\geq 1)}} \pth{ \dr F^{(1)}  + \F + F^{(2,1)} + F^{(2,2)} + F^{(2,\pm)} } \Big\}^{\rho,\osc}
\end{aligned}
\end{equation}
and
\begin{equation}\label{ggD2}
\begin{aligned}
&\g_\la^{\rho\sigma} \g_\la^{\mu\nu} D_{\mu\nu\si}^{(2)} 
\\& =   \g_\la^{\rho\sigma} \g_\la^{\mu\nu} \pth{ \dr_\mu (\h_\lambda)_{\sigma\nu} - \half \dr_\sigma (\h_\lambda)_{\mu\nu}  +  \sum_\A  \sin\left(\frac{u_\A}{\lambda} \right) \pth{\dr_\mu (\F_\A)_{\sigma\nu} - \half \dr_\sigma (\F_\A)_{\mu\nu} }  }
\\&\quad +    \sum_{\substack{u\in\mathcal{N}\\T\in\{\cos,\sin\}}}T'\pth{\frac{u}{\la}}\g_0^{\rho\sigma}  \Pol{\gg^{(3,h)}_{u,T}}{u}_\si
\\&\quad +  \g_\la^{\rho\sigma} \g_\la^{\mu\nu}  \sum_{\substack{v\in\mathcal{I}\\T\in\{\cos,\sin\}}}T'\pth{\frac{v}{\la}}\pth{ \dr_\mu v(\gg^{(3,e)}_{v,T})_{\sigma\nu} - \half \dr_\sigma v(\gg^{(3,e)}_{v,T})_{\mu\nu} }
\\&\quad +  \sum_{\substack{u\in\mathcal{N}\\T\in\{\cos,\sin\}}} T\pth{\frac{u}{\la}} \cro{ \dr F^{(2,1)} + \dr F^{(2,2)} }^\rho +  \sum_{\substack{v\in\mathcal{I}\\T\in\{\cos,\sin\}}}T\pth{\frac{v}{\la}} \cro{ \dr F^{(2,\pm)} }^\rho
\\&\quad + \la   \cro{ \g_\la^{-1}(\g_\la)^{(\geq 1)} \pth{\dr F^{(2,1)} + \dr F^{(2,2)} + \dr F^{(2,\pm)} + \gg^{(3,h)}} }^{\rho,\osc} .
\end{aligned}
\end{equation}
The term $\g_\la^{\rho\sigma} \g_\la^{\mu\nu}D_{\mu\nu\si}^{(3)} $ is left as it is. Collecting \eqref{ggD0}-\eqref{ggD2} and plugging them into \eqref{H rho somme} concludes the proof of Lemma \ref{lemma H}.

\subsection{Proof of Proposition \ref{prop expression of the ricci}}\label{appendix proof prop ricci}

We start with $R^{(0)}_{\a\b}$. From \eqref{Ricci tensor GWC}, Lemmas \ref{lemma W} and \ref{lemma P} we have
\begin{align*}
2R^{(0)}_{\a\b} & = - \tBox_{\g_0}(\g_0)_{\a\b} - \sum_\A \sin\pth{\frac{u_\A}{\la}} (W^{(0,1)}_\A)_{\a\b}  -  \sum_{\A\neq\B,\pm} \cos\pth{\frac{u_\A\pm u_\B}{\la}}  (W^{(0,\pm)}_{\A\B})_{\a\b}
\\&\quad + P_{\alpha\beta}(\g_0)(\dr \g_0, \dr \g_0)  - \frac{1}{4}  \sum_{\A} \dr_\alpha u_\A \dr_\beta u_\A  \left| F^{(1)}_\A \right|^2_{\g_0} 
\\&\quad + \sum_\A \sin\pth{\frac{u_\A}{\la}} (P^{(0,1)}_\A)_{\a\b} + \sum_\A \cos\pth{\frac{2u_\A}{\la}} (P^{(0,2)}_\A)_{\a\b}
\\&\quad + \sum_{\A\neq \B,\pm}\cos\pth{\frac{u_\A\pm u_\B}{\la}} (P^{(0,\pm)}_{\A\B})_{\a\b} + \pth{ H^\rho \dr_\rho (\g_\la)_{\a\b} + (\g_\la)_{\rho(\a} \dr_{\b)}H^\rho }^{(0)}.
\end{align*}
We use Lemma \ref{lemma H} to compute the gauge terms:
\begin{align*}
&\pth{ H^\rho \dr_\rho (\g_\la)_{\a\b} + (\g_\la)_{\rho(\a} \dr_{\b)}H^\rho }^{(0)} 
\\& =   \sum_\A\pth{-\sin\pth{\frac{u_\A}{\la}} (\g_0)_{\rho(\a}  \dr_{\b)}u_\A (H^{(1,1)}_\A)^\rho +2 \cos\pth{\frac{2u_\A}{\la}} (\g_0)_{\rho(\a} \dr_{\b)} u_\A (H^{(1,2)}_\A)^\rho} 
\\&\quad + \sum_{\A\neq\B,\pm}\cos\pth{\frac{u_\A\pm u_\B}{\la}}(\g_0)_{\rho(\a} \dr_{\b)}(u_\A\pm u_\B)(H^{(1,\pm)}_{\A\B})^\rho .
\end{align*}
Therefore we obtain
\begin{align*}
2R^{(0)}_{\a\b} & = \sum_\A \pth{2 F_\A^2 -  \frac{1}{4} \left| F^{(1)}_\A \right|^2_{\g_0} } \dr_\mu u_\A \dr_\nu u_\A 
\\&\quad + \sum_\A \sin\pth{\frac{u_\A}{\la}}\pth{ -  (W^{(0,1)}_\A)_{\a\b} + (P^{(0,1)}_\A)_{\a\b} - (\g_0)_{\rho(\a}  \dr_{\b)}u_\A (H^{(1,1)}_\A)^\rho}
\\&\quad + \sum_\A \cos\pth{\frac{2u_\A}{\la}} \pth{ (P^{(0,2)}_\A)_{\a\b} + 2(\g_0)_{\rho(\a} \dr_{\b)} u_\A (H^{(1,2)}_\A)^\rho}
\\&\quad + \sum_{\A\neq\B,\pm} \cos\pth{\frac{u_\A\pm u_\B}{\la}}\Big( -  (W^{(0,\pm)}_{\A\B})_{\a\b} + (P^{(0,\pm)}_{\A\B})_{\a\b} 
+(\g_0)_{\rho(\a} \dr_{\b)}(u_\A\pm u_\B)(H^{(1,\pm)}_{\A\B})^\rho \Big),
\end{align*}
where we also used \eqref{eq BG coordinates}. Using now \eqref{W 0 1}, \eqref{P 0 1}, \eqref{H 1 1}, \eqref{P 0 2}, \eqref{H 1 2}, \eqref{W 0 pm} and \eqref{H 1 pm} we obtain \eqref{R 0}-\eqref{I 0 pm}.

\saut
We now compute $R^{(1)}_{\a\b}$. From \eqref{Ricci tensor GWC}, Lemmas \ref{lemma W} and \ref{lemma P} we have
\begin{equation}\label{R1 proof}
\begin{aligned}
2R^{(1)}_{\a\b} & =  \sum_\A\cos\pth{\frac{u_\A}{\la}}\pth{ - (W^{(1,1)}_\A)_{\a\b} + (P^{(1,1)}_\A)_{\a\b} } 
\\&\quad + \sum_\A \sin\pth{\frac{2u_\A}{\la}}\pth{ - (W^{(1,2)}_\A)_{\a\b} + (P^{(1,2)}_\A)_{\a\b} }
\\&\quad - \sum_\A \cos\pth{\frac{u_\A}{\lambda}} \cos\pth{ \frac{2u_\A}{\lambda} } (F^{(2,2)}_\A)_{L_\A L_\A} (F^{(1)}_\A)_{\alpha\beta} 
\\&\quad - \half \sum_\A \sin\pth{\frac{2u_\A}{\la}} (\F_\A)_{L_\A L_\A} (F^{(1)}_\A)_{\a\b} 
\\&\quad + \sum_{\substack{v\in\mathcal{I}\\T\in\{\cos,\sin\}}} T\pth{\frac{v}{\la}}\pth{ - (W^{(1)}_{v,T})_{\a\b} + (P^{(1)}_{v,T})_{\a\b} }
\\&\quad + \sum_{\substack{u\in\mathcal{N}\\T\in\{\cos,\sin\}}}T\pth{\frac{u}{\la}} \dr_{(\a}u (\hat{P}^{(1)}_{u,T})_{\b)}  + \pth{ H^\rho \dr_\rho (\g_\la)_{\a\b} + (\g_\la)_{\rho(\a} \dr_{\b)}H^\rho }^{(1)}.
\end{aligned}
\end{equation}
We use Lemma \ref{lemma H} (and in particular \eqref{H 1 1}-\eqref{H 2}) to compute the gauge terms. We have
\begin{align*}
&\pth{ (\g_\la)_{\rho(\a} \dr_{\b)}H^\rho }^{(1)} 
\\& =  \sum_\A\cos\pth{\frac{u_\A}{\la}}\pth{ (\g_0)_{\rho(\a} \dr_{\b)}\pth{ \g_0^{\rho\sigma} \Pol{\F_\A + F^{(2,1)}_\A}{u_\A}_\si} + \cro{\dr^{\leq 2}F^{(1)} + (F^{(1)})^2}^\rho } 
\\&\quad -2 \sum_\A\sin\pth{\frac{2u_\A}{\la}}\pth{ (\g_0)_{\rho(\a}\dr_{\b)}\pth{ \g_0^{\rho\sigma} \Pol{F^{(2,2)}_\A}{u_\A}_\si } + \cro{F^{(1)}\dr F^{(1)}}^\rho }
\\&\quad -  \sum_{\substack{u\in\mathcal{N}\\T\in\{\cos,\sin\}}}T\pth{\frac{u}{\la}} \dr_{(\a}u
\\&\hspace{2.5cm} \times \bigg( \Pol{\gg^{(3,h)}_{u,T}}{u}_{\b)} 
\\&\hspace{3.5cm} + \Big\{\dr F^{(2,1)} + \dr F^{(2,2)}
\\&\hspace{4.5cm} +  \pth{ \F + F^{(2,1)} + F^{(2,2)} + F^{(2,\pm)} + (F^{(1)})^2 + \dr F^{(1)} }\pth{1+F^{(1)}} \Big\}_{\b)} \bigg)
\\&\quad +  \sum_{\substack{v\in\mathcal{I}\\T\in\{\cos,\sin\}}}T\pth{ \frac{v}{\la}} \Big\{ \dr F^{(2,\pm)}
 +  \pth{ \F + F^{(2,1)} + F^{(2,2)} + F^{(2,\pm)} + (F^{(1)})^2 + \dr F^{(1)} }\pth{1+F^{(1)}} \Big\}_{\a\b}
\\&\quad +(\g_0)_{\rho(\a} \pth{\dr_{\b)}\Upsilon^\rho }^{(-1)} 
\end{align*}
and
\begin{align*}
&\pth{ H^\rho \dr_\rho (\g_\la)_{\a\b} }^{(1)} 
\\& =   \sum_\A \cos\pth{\frac{u_\A}{\la}} \pth{ \dr_\rho (\g_0)_{\a\b} \g_0^{\rho\sigma} \Pol{\F_\A + F^{(2,1)}_\A}{u_\A}_\si  + \cro{\dr^{\leq 1}F^{(1)}+(F^{(1)})^3}_{\a\b}}
\\&\quad -2 \sum_\A\sin\pth{\frac{2u_\A}{\la}} \pth{ \dr_\rho (\g_0)_{\a\b}  \g_0^{\rho\sigma} \Pol{F^{(2,2)}_\A}{u_\A}_\si  + \cro{F^{(1)}\pth{F^{(2,1)} + \dr^{\leq 1}F^{(1)}}}_{\a\b}}
\\&\quad + \half \sum_\A \sin\pth{\frac{2u_\A}{\la}} \Pol{\F_\A}{u_\A}_{L_\A} (F^{(1)}_\A)_{\a\b} 
 + 2 \sum_\A \sin\pth{\frac{u_\A}{\la}}\sin\pth{\frac{2u_\A}{\la}} (F^{(2,2)}_\A)_{L_\A L_\A} (F^{(1)}_\A)_{\a\b} 
\\&\quad + \sum_{\substack{v\in\mathcal{I}\\T\in\{\cos,\sin\}}}T\pth{ \frac{v}{\la}} \cro{ F^{(1)} \pth{F^{(1)} + \F + F^{(2,1)} + F^{(2,2)} + \dr^{\leq 1}F^{(1)} + (F^{(1)})^2 } }_{\a\b}.
\end{align*}
We plug these two computations into \eqref{R1 proof}, use \eqref{W 1 1}, \eqref{P 1 1}, \eqref{W 1 2}, \eqref{P 1 2}, \eqref{W 1 v T}, \eqref{P 1 v T} and \eqref{Phat 1 u T} and obtain \eqref{R 1}-\eqref{R 1 e}.

\saut
We now compute $R^{(\geq 2)}_{\a\b}$. From Lemmas \ref{lemma W} and \ref{lemma P} we have
\begin{align*}
2 R^{(\geq 2)}_{\a\b} & = - W^{(\geq 2)}_{\a\b} + P^{(\geq 2)}_{\a\b}  +\pth{ H^\rho \dr_\rho (\g_\la)_{\a\b} + (\g_\la)_{\rho(\a} \dr_{\b)}H^\rho }^{(\geq 2)}.
\end{align*}
We use \eqref{W 2} and \eqref{P 2} to obtain
\begin{align*}
2 R^{(\geq 2)}_{\a\b} & = -\tBox_{\g_\la}(\h_\la)_{\a\b}  - \sum_\A \sin\pth{\frac{u_\A}{\la}}\tBox_{\g_\la}(\F_\A)_{\a\b} 
\\&\quad -  \la \sum_{\substack{u\in\mathcal{N}\\T\in\{\cos,\sin\}}}T\pth{\frac{u}{\la}}\tBox_{\g_\la} (\gg^{(3,h)}_{u,T})_{\a\b}  - \la  \sum_{\substack{v\in\mathcal{I}\\T\in\{\cos,\sin\}}} T\left( \frac{v}{\lambda} \right)\tBox_{\g_\la} (\gg^{(3,e)}_{v,T})_{\a\b}
\\&\quad +  \Big\{\pth{ \g_\la^{-1}\g_\la^{-1}\dr\g_\la\dr\g_\la }^{(\geq 2)} + \g_\la^{-1}\big(\dr^{\leq 1}\gg^{(3,h)} + \dr^{\leq 1}\gg^{(3,e)} + \dr^{\leq 2} F^{(1)} + \dr^{\leq 1} \F 
\\&\hspace{7cm}+ \dr^{\leq 2}F^{(2,1)} + \dr^{\leq 2}F^{(2,2)} + \dr^{\leq 2} F^{(2,\pm)} \big)  \Big\}_{\a\b}^\osc 
\\&\quad  +\pth{ H^\rho \dr_\rho (\g_\la)_{\a\b} + (\g_\la)_{\rho(\a} \dr_{\b)}H^\rho }^{(\geq 2)},
\end{align*}
which gives \eqref{R 2}. This concludes the proof of Proposition \ref{prop expression of the ricci}.


\section{Proof of Proposition \ref{prop constraint main}}\label{appendix proof prop constraint main}

\subsection{Conformal formulation of the constraint equations}\label{section conformal}

We introduce the notations
\begin{align*}
\HH(g,k) & \vcentcolon= R(g) - |k|^2_g + (\tr_g k)^2,
\\ \MM(g,k) & \vcentcolon= \div_g k - \d \tr_g k,
\end{align*}
In order to solve the constraint equations, we take inspiration from \cite{Corvino2006} and look for $(g_\la,k_\la)$ under the form
\begin{equation}\label{conformal ansatz}
\begin{aligned}
g_\la & = \ffi^4_\la \ga_\la,
\\ k_\la & = \ffi^2_\la \pth{  \ka_\la  +  \L_{\ga_\la} X_\la},
\end{aligned}
\end{equation}
where 
\begin{align}\label{def LX}
\L_g X \vcentcolon =  \LL_{X}g -  \half \pth{\div_{g}X} g ,
\end{align}
with $\LL_Xg$ is the Lie derivative. A straightforward computation implies that the constraint equations for $(g_\la,k_\la)$ rewrite
\begin{align}
8 \De_{\ga_\la}\ffi_\la & = \HH\pth{\ga_\la,\ka_\la} + \RR_\HH\pth{ \ga_\la,\ka_\la,\ffi_\la,X_\la},\label{hamiltonian}
\\ (\vDe_{\ga_\la} X_\la)_i& = -  \MM(\ga_\la,\ka_\la)_i  +  \RR_\MM\pth{ \ga_\la,\ka_\la,\ffi_\la,X_\la}_i ,\label{momentum}
\end{align}
where
\begin{align*}
\vDe_{g} X_i \vcentcolon= \De_g X_i + R_{ij}(g)X^{j}
\end{align*}
with $\De_g$ now acting on vector fields and $R_{ij}(g)$ the Ricci tensor of $g$, and where the remainders are given by
\begin{align}
\RR_\HH\pth{ \ga_\la,\ka_\la,\ffi_\la,X_\la} & \vcentcolon= ( \ffi_\la-1)\HH\pth{\ga_\la,\ka_\la} \non
\\&\quad +\ffi_\la \bigg( - \left|  \LL_{X_\la} \ga_\la \right|^2_{\ga_\la} -2 \left| \ka_\la \cdot \LL_{X_\la} \ga_\la \right|_{\ga_\la} \label{reste H} 
\\&\hspace{3cm}  + \frac{3}{2} \pth{\div_{\ga_\la} X_\la}^2 +2 \pth{\tr_{\ga_\la}\ka_\la}\pth{ \div_{\ga_\la} X_\la}\bigg),\non
\\ \RR_\MM\pth{ \ga_\la,\ka_\la,\ffi_\la,X_\la}_i & \vcentcolon= - 4 \ffi_\la^{-1} \ga_\la^{ab} \pth{ \ka_\la+ \LL_{X_\la}\ga_\la }_{ai}\dr_b \ffi_\la + 2 \ffi_\la^{-1} \pth{\div_{\ga_\la}X_\la}\dr_i \ffi_\la. \label{reste M}
\end{align}
The parameters $\ga_\la$ and $\ka_\la$ of the conformal formulation are defined by
\begin{align}
\ga_\la & \vcentcolon= g_0 + \la \sum_\A \cos\pth{\frac{u_\A}{\la}} \bar{F}^{(1)}_\A + \la^2 \sum_{\A\neq \B,\pm}\cos\pth{\frac{u_\A\pm u_\B}{\la}}\ga^{(2,\pm)}_{\A\B},\label{ga la}
\\ \ka_\la & \vcentcolon= k_0 + \sum_\A \sin\pth{ \frac{u_\A}{\la}}  \half |\nab u_\A| \bar{F}^{(1)}_\A  + \la\sum_\A\pth{\cos\pth{\frac{u_\A}{\la}} \ka^{(1,1)}_\A + \sin\pth{\frac{2u_\A}{\la}} \ka^{(1,2)}_\A} \label{ka la}
\\&\quad + \la\sum_{\A\neq\B,\pm}\sin\pth{\frac{u_\A\pm u_\B}{\la}}\ka^{(1,\pm)}_{\A\B}.\nonumber
\end{align}
The unknowns $\ffi_\la$ and $X_\la$ of the conformal formulation are of the form
\begin{align}
\ffi_\la & = 1 + \la^2 \ffi^{(2)} + \la^2 \tffi_\la + \la^3 \ffi^{(3)}\label{def ffi la},
\\ X_\la & = \la^2 X^{(2)} + \la^2 \tX_\la + \la^3 X^{(3)}.\label{def X la}
\end{align}
where
\begin{align}
\ffi^{(2)} & =  \sum_\A \pth{ \sin\pth{\frac{u_\A}{\la}} \ffi^{(2,1)}_\A + \cos\pth{\frac{2u_\A}{\la}}\ffi^{(2,2)}_\A } +  \sum_{\A\neq \B,\pm} \cos\pth{\frac{u_\A\pm u_\B}{\la}}\ffi^{(2,\pm)}_{\A\B},\label{def ffi 2}
\\ X^{(2)} & =  \sum_\A \pth{ \sin\pth{\frac{u_\A}{\la}} X^{(2,1)}_\A + \cos\pth{\frac{2u_\A}{\la}}X^{(2,2)}_\A }  + \sum_{\A\neq \B,\pm} \cos\pth{\frac{u_\A\pm u_\B}{\la}}X^{(2,\pm)}_{\A\B}.\label{def X 2}
\end{align}
In \eqref{def ffi la}-\eqref{def X la}, $\tffi_\la$ and $\tX_\la$ are non-oscillating remainders, while $\ffi^{(3)}$ and $X^{(3)}$ are oscillating but we don't need to be precise on their oscillating behaviour.

\subsection{The constraint hierarchy}\label{section constraint hierarchy}

We first compute useful expansions related to $\ga_\la$. The inverse of $\ga_\la$ satisfies 
\begin{align}\label{ga-1}
\ga_\la^{ij} = g_0^{ij} + \la(\ga_\la^{ij})^{(1)} + \la^2(\ga_\la^{ij})^{(2)} + \GO{\la^3}
\end{align}
with
\begin{align}
(\ga_\la^{ij})^{(1)} & = - \sum_\A \cos\pth{\frac{u_\A}{\la}} (\bar{F}^{(1)}_\A)^{ij},\label{ga-1 1}
\\ (\ga_\la^{ij})^{(2)}  & = -  \sum_{\A\neq \B,\pm}\cos\pth{\frac{u_\A\pm u_\B}{\la}} (\ga^{(2,\pm)}_{\A\B})^{ij} -  \sum_{\A,\B} \cos\pth{\frac{u_\A}{\la}} \cos\pth{\frac{u_\B}{\la}} (\bar{F}^{(1)}_\B)^{ik} (\bar{F}^{(1)}_\A)_{k}^j. \label{ga-1 2}
\end{align}
The Christofel symbols of $\ga_\la$ satisfy 
\begin{align}
\Ga(\ga_\la)^k_{ij} & =  \Ga(g_0)^k_{ij} + (\tilde{\Ga}^{(0)})^k_{ij} + \la (\Ga^{(1)})^k_{ij}  + \GO{\la^2}\label{christofel symbols}
\end{align}
with
\begin{align}
(\tilde{\Ga}^{(0)})^k_{ij} & = -\half\sum_\A  \sin\pth{\frac{u_\A}{\la}} \pth{\dr_{(i} u_\A (\bar{F}^{(1)}_\A)_{j)}^k - \dr^k  u_\A (\bar{F}^{(1)}_\A)_{ij} }, \label{Gatilde 0}
\\ (\Ga^{(1)})^k_{ij} & = \sum_\A \cos\pth{\frac{u_\A}{\la}} (\bar{Q}_\A)^k_{ij}  + \half \sum_\A \cos\pth{\frac{u_\A}{\la}} g_0^{k\ell} \pth{ \dr_{(i} u_\A (\ga^{(2,1)}_\A)_{j)\ell} - \dr_\ell u_\A (\ga^{(2,1)}_\A)_{ij} }\non
\\&\quad - \half  \sum_{\A\neq \B,\pm}\sin\pth{\frac{u_\A\pm u_\B}{\la}}g_0^{k\ell}\Big( \dr_{(i} (u_\A\pm u_\B)(\ga^{(2,\pm)}_{\A\B})_{j)\ell} \label{Ga 1}
 - \dr_\ell (u_\A\pm u_\B)(\ga^{(2,\pm)}_{\A\B})_{ij}\Big) 
\\&\quad + \half \sum_{\A,\B} \cos\pth{\frac{u_\A}{\la}}\sin\pth{\frac{u_\B}{\la}} (\bar{F}^{(1)}_\A)^{k\ell} \pth{ \dr_{(i} u_\B (\bar{F}^{(1)}_\B)_{j)\ell} - \dr_\ell u_\B (\bar{F}^{(1)}_\B)_{ij}},\non
\end{align}
and
\begin{align*}
(\bar{Q}_\A)^k_{ij} & =  \half  \pth{ g_0^{k\ell}\pth{ \dr_{(i} (\bar{F}^{(1)}_\A)_{j)\ell} - \dr_\ell (\bar{F}^{(1)}_\A)_{ij}} - (\bar{F}^{(1)}_\A)^{k\ell}\pth{ \dr_{(i} (g_0)_{j)\ell} - \dr_\ell (g_0)_{ij}} }.
\end{align*}
Note that \eqref{assum F1} implies 
\begin{equation}\label{barQ}
\begin{aligned}
(\bar{Q}_\A)^k_{kj} & = 0, 
\\ g_0^{ij}\dr_k u_\A (\bar{Q}_\A)^k_{ij} & =(\bar{F}^{(1)}_\A)_{ij}  \pth{ -   \dr^i \dr^j u_\A + \half\dr^\ell u_\A\dr_\ell g_0^{ij}}.
\end{aligned}
\end{equation}
We are now ready to compute the main terms in the equations \eqref{hamiltonian}-\eqref{momentum}.

\begin{lemma}\label{lem constraint expansion H}
We have 
\begin{align*}
 \HH(\ga_\la,\ka_\la) & = \pth{ \HH(\ga_\la,\ka_\la)}^{(0)} + \la \pth{ \HH(\ga_\la,\ka_\la)}^{(1)} + \la^2 \pth{ \HH(\ga_\la,\ka_\la)}^{(\geq 2)},
\end{align*}
where 
\begin{equation}\label{H 0}
\begin{aligned}
&\pth{ \HH(\ga_\la,\ka_\la)}^{(0)} 
\\& =  \sum_\A \sin\pth{\frac{u_\A}{\la}}  (\bar{F}^{(1)}_\A)_{ij}  \pth{ \dr^i \dr^j u_\A - \half\dr^\ell u_\A\dr_\ell g_0^{ij} - |\nab_{g_0} u_\A|_{g_0}k_0^{ij}}  
\\&\quad - 6 \sum_{\A}\cos\pth{\frac{2u_\A}{\la}}  |\nab_{g_0} u_\A|_{g_0}^2 F_\A^2 
\\&\quad + \sum_{\A\neq \B,\pm}\cos\pth{\frac{u_\A\pm u_\B}{\la}}\Big( |\nab_{g_0}(u_\A\pm u_\B)|^2_{g_0}\tr_{g_0}\ga^{(2,\pm)}_{\A\B} -(\ga^{(2,\pm)}_{\A\B})_{\nab_{g_0}(u_\A\pm u_\B)\nab_{g_0}(u_\A\pm u_\B)} \Big)
\\&\quad + \frac{1}{8}\sum_{\A\neq\B,\pm} \cos\pth{\frac{u_\A\pm u_\B}{\la}} \bigg( \pm 2 (\bar{F}^{(1)}_\A)_{i \nab_{g_0} u_\B} (\bar{F}^{(1)}_\B)_{\nab_{g_0} u_\A}^i 
\\&\hspace{5.5cm} + \left| \bar{F}^{(1)}_\A \cdot \bar{F}^{(1)}_\B \right|_{g_0} \Big( -2 |\nab_{g_0} u_\A|_{g_0}^2 -2 |\nab_{g_0} u_\B|_{g_0}^2 
\\&\hspace{9cm} \mp 3 |\nab_{g_0} u_\A\cdot \nab_{g_0} u_\B|_{g_0} 
\\&\hspace{9cm}\pm |\nab_{g_0} u_\A|_{g_0} |\nab_{g_0} u_\B|_{g_0} \Big)  \bigg).
\end{aligned}
\end{equation}
Moreover, the higher order terms satisfy
\begin{align}
\pth{ \HH(\ga_\la,\ka_\la)}^{(1)} & = \sum_{\substack{w\in\mathcal{N}\cup\mathcal{I}\\T\in\{\cos,\sin\}}} T\pth{\frac{w}{\la}} \cro{ \pth{ \ga_\la^{-1}\dr^2\ga_\la+ \ga_\la^{-1}\ga_\la^{-1}\pth{(\dr\ga_\la)^2 + (\ka_\la)^2} }^{(1)} }      ,\label{H 1}
\\\pth{ \HH(\ga_\la,\ka_\la)}^{(\geq 2)} & = \cro{ \pth{ \ga_\la^{-1}\dr^2\ga_\la+ \ga_\la^{-1}\ga_\la^{-1}\pth{(\dr\ga_\la)^2 + (\ka_\la)^2} }^{(\geq 2)} }^\osc.\label{H 2}
\end{align} 
\end{lemma}

\begin{proof}
We start with the scalar curvature, using its expression in coordinates:
\begin{align*}
R(\ga) & = \ga^{ij} \pth{ \dr_k \Ga(\ga)^k_{ij} - \dr_i  \Ga(\ga)^k_{jk} +  \Ga(\ga)^k_{k\ell} \Ga(\ga)^\ell_{ij} -  \Ga(\ga)^k_{i\ell} \Ga(\ga)^\ell_{jk} }.
\end{align*}
Using \eqref{assum F1} we first obtain
\begin{align*}
\pth{ R(\ga_\la) }^{(-1)} & = g_0^{ij} \pth{ \dr_k (\tilde{\Ga}^{(0)})^k_{ij} - \dr_i  (\tilde{\Ga}^{(0)})^k_{jk} }^{(-1)}.
\\& =- \sum_\A  \cos\pth{\frac{u_\A}{\la}}\pth{ (\bar{F}^{(1)}_\A)_{\nab_{g_0} u_\A \nab_{g_0} u_\A} -|\nab_{g_0} u_\A|_{g_0}^2 \tr_{g_0}\bar{F}^{(1)}_\A  } 
\\& = 0.
\end{align*}
We have $\pth{ R(\ga_\la) }^{(0)} = I + II + III$ with
\begin{align*}
I & \vcentcolon = g_0^{ij} \pth{  \dr_k \Ga(\ga_\la)^k_{ij} - \dr_i  \Ga(\ga_\la)^k_{jk}  }^{(0)},
\\ II & \vcentcolon = \pth{\ga_\la^{ij}}^{(1)} \pth{  \dr_k (\tilde{\Ga}^{(0)})^k_{ij} - \dr_i (\tilde{\Ga}^{(0)})^k_{jk}  }^{(-1)},
\\ III & \vcentcolon = g_0^{ij} \pth{\Ga(\ga_\la)^k_{k\ell} \Ga(\ga_\la)^\ell_{ij} -  \Ga(\ga_\la)^k_{i\ell} \Ga(\ga_\la)^\ell_{jk}}^{(0)}.
\end{align*}
We start with $I$. Using $(\tilde{\Ga}^{(0)})^k_{jk} =0$ and $g_0^{ij}(\tilde{\Ga}^{(0)})^k_{ij}=0$ (which follow from \eqref{Gatilde 0} and \eqref{assum F1}) we have
\begin{equation}\label{R0 I}
\begin{aligned}
I & = g_0^{ij} \pth{  \dr_k \Ga(g_0)^k_{ij} - \dr_i  \Ga(g_0)^k_{jk}  }  -  (\tilde{\Ga}^{(0)})^k_{ij} \dr_k g_0^{ij}  + g_0^{ij} \pth{  \dr_k (\Ga^{(1)})^k_{ij} - \dr_i (\Ga^{(1)})^k_{jk}  }^{(-1)}
\\& =  g_0^{ij} \pth{  \dr_k \Ga(g_0)^k_{ij} - \dr_i  \Ga(g_0)^k_{jk}  }  -  (\tilde{\Ga}^{(0)})^k_{ij} \dr_k g_0^{ij}
\\&\quad + \sum_\A \sin\pth{\frac{u_\A}{\la}}  (\bar{F}^{(1)}_\A)_{ij}  \pth{ \dr^i \dr^j u_\A - \half\dr^\ell u_\A\dr_\ell g_0^{ij}}  
\\&\quad + \sum_{\A}\sin\pth{\frac{u_\A}{\la}}\pth{ |\nab u_\A|^2_{g_0}\tr_{g_0}\ga^{(2,1)}_{\A} -(\ga^{(2,1)}_{\A})_{\nab_{g_0} u_\A \nab_{g_0} u_\A} }
\\&\quad + \sum_{\A\neq \B,\pm}\cos\pth{\frac{u_\A\pm u_\B}{\la}}\Big( |\nab_{g_0}(u_\A\pm u_\B)|^2_{g_0}\tr_{g_0}\ga^{(2,\pm)}_{\A\B} -(\ga^{(2,\pm)}_{\A\B})_{\nab_{g_0}(u_\A\pm u_\B)\nab_{g_0}(u_\A\pm u_\B)} \Big)
\\&\quad - 4 \sum_{\A}\cos\pth{\frac{2u_\A}{\la}}  |\nab_{g_0} u_\A|_{g_0}^2 F_\A^2 
\\&\quad -  \frac{1}{8} \sum_{\A\neq\B,\pm}\cos\pth{\frac{u_\A\pm u_\B}{\la}}  |\nab_{g_0}(u_\A\pm u_\B)|^2_{g_0}\left| \bar{F}^{(1)}_\A \cdot \bar{F}^{(1)}_\B \right|_{g_0} ,
\end{aligned}
\end{equation}
where we have also used \eqref{Ga 1} and \eqref{barQ}. We now compute $II$, using \eqref{ga-1 1} and \eqref{Gatilde 0}:
\begin{equation}\label{R0 II}
\begin{aligned}
II & = -2 \sum_{\A}  |\nab_{g_0} u_\A|_{g_0}^2F_\A^2 -2 \sum_{\A} \cos\pth{\frac{2u_\A}{\la}}  |\nab_{g_0} u_\A|_{g_0}^2F_\A^2  
\\&\quad  -\frac{1}{8}  \sum_{\A\neq\B,\pm}  \cos\pth{\frac{u_\A\pm u_\B}{\la}} \pth{ |\nab_{g_0} u_\A|_{g_0}^2 + |\nab_{g_0} u_\B|_{g_0}^2 } \left| \bar{F}^{(1)}_\A \cdot \bar{F}^{(1)}_\B \right|_{g_0} .
\end{aligned}
\end{equation}
We now compute $III$ using \eqref{Gatilde 0}:
\begin{equation}\label{R0 III}
\begin{aligned}
III & = g_0^{ij} \pth{\Ga(g_0)^k_{k\ell} \Ga(g_0)^\ell_{ij} -  \Ga(g_0)^k_{i\ell} \Ga(g_0)^\ell_{jk}} - 2g_0^{ij} (\tilde{\Ga}^{(0)})^\ell_{ik} \Ga(g_0)^k_{j\ell}   - g_0^{ij}  (\tilde{\Ga}^{(0)})^k_{i\ell} (\tilde{\Ga}^{(0)})^\ell_{jk}
\\& = g_0^{ij} \pth{\Ga(g_0)^k_{k\ell} \Ga(g_0)^\ell_{ij} -  \Ga(g_0)^k_{i\ell} \Ga(g_0)^\ell_{jk}} - 2g_0^{ij} (\tilde{\Ga}^{(0)})^\ell_{ik} \Ga(g_0)^k_{j\ell} 
\\&\quad + \sum_{\A}  |\nab_{g_0} u_\A |_{g_0}^2 F_\A^2 - \sum_{\A} \cos\pth{\frac{2u_\A}{\la}}  |\nab_{g_0} u_\A |_{g_0}^2 F_\A^2 
\\&\quad + \frac{1}{8}\sum_{\A\neq\B,\pm} \cos\pth{\frac{u_\A\pm u_\B}{\la}} \bigg( \pm 2 (\bar{F}^{(1)}_\A)_{i \nab_{g_0} u_\B} (\bar{F}^{(1)}_\B)_{\nab_{g_0} u_\A}^i 
\\&\hspace{7cm} \mp  |\nab_{g_0} u_\A \cdot \nab_{g_0} u_\B |_{g_0}  \left| \bar{F}^{(1)}_\A \cdot \bar{F}^{(1)}_\B \right|_{g_0} \bigg).
\end{aligned}
\end{equation}
We now compute the $\la^0$ contribution of the quadratic terms in $\ka_\la$. From \eqref{assum F1} we have $\tr_{\ga_\la} \ka_\la = \tr_{g_0}k_0 + \GO{\la}$ so that 
\begin{align}\label{tr2 0}
\pth{ \pth{\tr_{\ga_\la} \ka_\la}^2}^{(0)} = (\tr_{g_0}k_0)^2
\end{align}
and
\begin{equation}\label{k2 0}
\begin{aligned}
\pth{ |\ka_\la|^2_{\ga_\la} }^{(0)} & = |k_0|^2_{g_0} + \sum_\A \sin\pth{ \frac{u_\A}{\la}}|\nab_{g_0} u_\A|_{g_0} \left| \bar{F}^{(1)}_\A \cdot k_0 \right|_{g_0}  
\\&\quad +  \sum_{\A} |\nab_{g_0} u_\A|_{g_0}^2  F_\A^2  -  \sum_{\A} \cos\pth{\frac{2u_\A}{\la}} |\nab_{g_0} u_\A|_{g_0}^2  F_\A^2
\\&\quad +  \frac{1}{8} \sum_{\A\neq\B,\pm}\mp\cos\pth{ \frac{u_\A\pm u_\B}{\la}}  |\nab_{g_0} u_\A|_{g_0} |\nab_{g_0} u_\B|_{g_0} \left| \bar{F}^{(1)}_\A \cdot \bar{F}^{(1)}_\B  \right|_{g_0}.
\end{aligned}
\end{equation}
Since the quadratic terms in $\ka_\la$ don't contribute to $\la^{-1}$ order we have proved that \[\pth{ \HH(\ga_\la,\ka_\la)}^{(-1)} =0.\] Moreover we have also proved that $\pth{ \HH(\ga_\la,\ka_\la)}^{(0)}$ is indeed given by \eqref{H 0}, putting together \eqref{R0 I}-\eqref{k2 0} and using in particular \eqref{BG hamiltonian} and \eqref{eikonal init} to cancel the resonant term. It only remains to prove \eqref{H 1}, since \eqref{H 2} follows from the definition of the Hamiltonian constraint operator $\HH$. In order to prove \eqref{H 1}, we only need to check that $\pth{ \HH(\ga_\la,\ka_\la)}^{(1)}$ is purely oscillating, i.e that is does not contain any resonant term. We use a schematic expression of $\pth{ \HH(\ga_\la,\ka_\la)}^{(1)}$ which follows from the definition of $\HH$:
\begin{equation}\label{H 1 rough}
\begin{aligned}
\pth{ \HH(\ga_\la,\ka_\la)}^{(1)} & = \pth{\ga_\la^{-1}}^{(0)}\pth{\dr^2\ga_\la}^{(1)} + \pth{\ga_\la^{-1}}^{(1)}\pth{\dr^2\ga_\la}^{(0)} + \pth{\ga_\la^{-1}}^{(2)}\pth{\dr^2\ga_\la}^{(-1)}
\\&\quad + \pth{\Ga(\ga_\la)}^{(0)}\pth{ \Ga(\ga_\la)}^{(1)} + \pth{ \ga_\la^{-1}\ka_\la}^{(0)}  \pth{ \ga_\la^{-1}\ka_\la}^{(1)} .
\end{aligned}
\end{equation}
We list every oscillating functions appearing in the five terms of the schematic expression \eqref{H 1 rough}:
\begin{itemize}
\item Thanks to \eqref{ga-1}, $\pth{\ga_\la^{-1}}^{(0)}\pth{\dr^2\ga_\la}^{(1)}$ contains $\cos\pth{\frac{u_\A}{\la}}$ and $\sin\pth{ \frac{u_\A\pm u_\B}{\la}}$.
\item Thanks to \eqref{ga-1 1}-\eqref{ga-1 2}, $ \pth{\ga_\la^{-1}}^{(1)}\pth{\dr^2\ga_\la}^{(0)}$ and $ \pth{\ga_\la^{-1}}^{(2)}\pth{\dr^2\ga_\la}^{(-1)}$ contain $\cos\pth{\frac{u_\A}{\la}}$, $\sin\pth{\frac{2u_\A}{\la}}$, $\sin\pth{ \frac{u_\A\pm u_\B}{\la}}$, $\cos\pth{ \frac{u_\A\pm2 u_\B}{\la}}$ and $\cos\pth{ \frac{u_\A\pm_1 u_\B \pm_2 u_\C}{\la}}$.
\item Thanks to \eqref{Gatilde 0}-\eqref{Ga 1}, $\pth{\Ga(\ga_\la)}^{(0)}\pth{ \Ga(\ga_\la)}^{(1)}$ contains $\cos\pth{\frac{u_\A}{\la}}$, $\sin\pth{\frac{2u_\A}{\la}}$, $\sin\pth{\frac{u_\A\pm u_\B}{\la}}$, $\cos\pth{ \frac{u_\A\pm2 u_\B}{\la}}$ and $\cos\pth{ \frac{u_\A\pm_1 u_\B \pm_2 u_\C}{\la}}$.
\item Thanks to \eqref{ga la}-\eqref{ka la} we see that the oscillating behaviour of $\ga_\la^{-1}\ka_\la$ is the same as the oscillating behaviour of $\Ga(\ga_\la)$ so that $\pth{ \ga_\la^{-1}\ka_\la}^{(0)}  \pth{ \ga_\la^{-1}\ka_\la}^{(1)}$ behaves as $\pth{\Ga(\ga_\la)}^{(0)}\pth{ \Ga(\ga_\la)}^{(1)}$.
\end{itemize}
This concludes the proof of the lemma.
\end{proof}

\begin{lemma}\label{lem constraint expansion M}
We have 
\begin{align*}
 \MM(\ga_\la,\ka_\la) & =  \pth{ \MM(\ga_\la,\ka_\la)}^{(0)} + \la \pth{ \MM(\ga_\la,\ka_\la)}^{(1)} + \la^2 \pth{ \MM(\ga_\la,\ka_\la)}^{(\geq 2)},
\end{align*}
where 
\begin{equation}\label{M 0}
\begin{aligned}
&\pth{ \MM(\ga_\la,\ka_\la)_i}^{(0)} 
\\& =  \sum_\A\sin\pth{ \frac{u_\A}{\la}} \bigg(  \dr_i u_\A \tr_{g_0}\ka^{(1,1)}_\A - (\ka^{(1,1)}_\A)_{\nab_{g_0} u_\A i} 
\\&\quad\hspace{2.5cm} + \half  \dr^b \pth{|\nab_{g_0} u_\A|_{g_0} \bar{F}^{(1)}_\A}_{bi} +\frac{1}{4} |\nab_{g_0} u_\A|_{g_0} (\bar{F}^{(1)}_\A)_{bc} \dr_i g_0^{bc} 
\\&\quad\hspace{2.5cm} +  \half |\nab_{g_0} u_\A|_{g_0} \pth{ \dr_a g_0^{ca} + \half g_0^{ab}g_0^{cd}\dr_d (g_0)_{ab} } (\bar{F}^{(1)}_\A)_{ic}  - \half \dr_i u_\A (\bar{F}^{(1)}_\A)_{bc} (k_0)^{bc} \bigg)
\\&\quad + \sum_\A\cos\pth{\frac{2u_\A}{\la}} \pth{ 2 (\ka^{(1,2)}_\A)_{\nab_{g_0} u_\A i} - 2 \dr_i u_\A \tr_{g_0}\ka^{(1,2)}_\A + 3 |\nab_{g_0} u_\A|_{g_0}  \dr_i u_\A F_\A^2}
\\&\quad +\sum_{\A\neq\B,\pm} \cos\pth{\frac{u_\A\pm u_\B}{\la}}\pth{ (\ka^{(1,\pm)}_{\A\B})_{(\nab_{g_0} u_\A \pm \nab_{g_0} u_\B)i} - \dr_i(u_\A\pm u_\B)\tr_{g_0}  \ka^{(1,\pm)}_{\A\B}}
\\&\quad - \half \sum_{\A\neq\B} \cos\pth{\frac{u_\A}{\la}}  \cos\pth{ \frac{u_\B}{\la}} |\nab_{g_0} u_\B|_{g_0} \bigg((\bar{F}^{(1)}_\A)^{b}_{\nab_{g_0} u_\B} (\bar{F}^{(1)}_\B)_{bi}  - \dr_i u_\B \left| \bar{F}^{(1)}_\A \cdot \bar{F}^{(1)}_\B\right|_{g_0}\bigg)
\\&\quad -\frac{1}{4} \sum_{\A\neq\B} \sin\pth{\frac{u_\A}{\la}} \sin\pth{ \frac{u_\B}{\la}}   |\nab_{g_0} u_\B|_{g_0}   \dr_i u_\A \left| \bar{F}^{(1)}_\A \cdot \bar{F}^{(1)}_\B \right|_{g_0}.
\end{aligned}
\end{equation}
Moreover, the higher order terms satisfy
\begin{align}
\pth{ \MM(\ga_\la,\ka_\la)}^{(1)} & = \sum_{\substack{w\in\mathcal{N}\cup\mathcal{I}\\T\in\{\cos,\sin\}}} T\pth{\frac{w}{\la}}   \cro{\pth{ \ga_\la^{-1}\dr\ka_\la + \ga_\la^{-1}\ga_\la^{-1} \ka_\la\dr\ga_\la }^{(1)}} \label{M 1}
\\ \pth{ \MM(\ga_\la,\ka_\la)}^{(\geq 2)} & = \cro{\pth{ \ga_\la^{-1}\dr\ka_\la + \ga_\la^{-1}\ga_\la^{-1} \ka_\la\dr\ga_\la }^{(\geq 2)}}^\osc.  \label{M 2}
\end{align} 
\end{lemma}

\begin{proof}
In coordinates we have
\begin{align*}
\MM(g,k)_i & =  g^{ab}\pth{ \dr_a k_{bi} - \dr_i  k_{ab}} + \pth{ \dr_a g^{ca} + \half g^{ab}g^{cd}\dr_d g_{ab} } k_{ic} - \half \dr_i g^{bc} k_{bc}   .
\end{align*}
Using again \eqref{assum F1} we easily obtain $\pth{ \MM(\ga_\la,\ka_\la)_i  }^{(-1)} = 0$. We now have $\pth{ \MM(\ga_\la,\ka_\la)_i}^{(0)}=I+II+III$ where
\begin{align*}
I & \vcentcolon = g_0^{ab}\pth{ \dr_a (\ka_\la)_{bi} - \dr_i  (\ka_\la)_{ab}}^{(0)},
\\ II & \vcentcolon = (\ga_\la^{ab})^{(1)} \pth{ \dr_a (\ka_\la)_{bi} - \dr_i  (\ka_\la)_{ab}}^{(-1)},
\\ III & \vcentcolon =  \pth{ \dr_a \ga_\la^{ca} + \half g_0^{ab}g_0^{cd}\dr_d (\ga_\la)_{ab} }^{(0)} \pth{(\ka_\la)_{ic}}^{(0)} - \half \pth{\dr_i \ga_\la^{bc}}^{(0)} \pth{(\ka_\la)_{bc}}^{(0)} .
\end{align*}
We start with $I$:
\begin{equation}\label{M0 I}
\begin{aligned}
I & = g_0^{ab}\pth{ \dr_a (k_0)_{bi} - \dr_i  (k_0)_{ab}}
\\&\quad + \sum_\A\sin\pth{ \frac{u_\A}{\la}} \bigg( \dr_i u_\A \tr_{g_0}\ka^{(1,1)}_\A - (\ka^{(1,1)}_\A)_{\nab_{g_0} u_\A i} 
\\&\hspace{4cm} + \half \pth{ \dr^b \pth{|\nab_{g_0} u_\A|_{g_0} \bar{F}^{(1)}_\A}_{bi} + |\nab_{g_0} u_\A|_{g_0} (\bar{F}^{(1)}_\A)_{ab} \dr_i g_0^{ab}  } \bigg)
\\&\quad + 2\sum_\A\cos\pth{\frac{2u_\A}{\la}} \pth{  (\ka^{(1,2)}_\A)_{\nab_{g_0} u_\A i} - \dr_i u_\A \tr_{g_0}\ka^{(1,2)}_\A}
\\&\quad +\sum_{\A\neq\B,\pm} \cos\pth{\frac{u_\A\pm u_\B}{\la}}\pth{ (\ka^{(1,\pm)}_{\A\B})_{(\nab_{g_0} u_\A \pm \nab_{g_0} u_\B)i} - \dr_i(u_\A\pm u_\B)\tr_{g_0}  \ka^{(1,\pm)}_{\A\B}}.
\end{aligned}
\end{equation}
We now compute $II$ and $III$ using \eqref{ga-1 1}:
\begin{align}
II & =  2 \sum_{\A} |\nab_{g_0} u_\A|_{g_0}  \dr_i u_\A F_\A^2 +  2 \sum_{\A} \cos\pth{\frac{2u_\A}{\la}} |\nab_{g_0} u_\A|_{g_0}  \dr_i u_\A F_\A^2 \label{M0 III}
\\&\quad - \frac{1}{4} \sum_{\A\neq\B,\pm} \cos\pth{\frac{u_\A\pm u_\B}{\la}}  |\nab_{g_0} u_\B|_{g_0} \pth{(\bar{F}^{(1)}_\A)^{b}_{\nab_{g_0} u_\B} (\bar{F}^{(1)}_\B)_{bi} - \dr_i u_\B \left| \bar{F}^{(1)}_\A \cdot \bar{F}^{(1)}_\B\right|_{g_0}}. \non
\\ III & =  \pth{ \dr_a g_0^{ca} + \half g_0^{ab}g_0^{cd}\dr_d (g_0)_{ab} } (k_0)_{ic} - \half \dr_i g_0^{bc} (k_0)_{bc} -  \sum_{\A}  |\nab_{g_0} u_\A|_{g_0}   \dr_i u_\A F_\A^2 \non
\\&\quad +  \half \sum_\A \sin\pth{ \frac{u_\A}{\la}} \bigg( |\nab_{g_0} u_\A|_{g_0} \pth{ \dr_a g_0^{ca} + \half g_0^{ab}g_0^{cd}\dr_d (g_0)_{ab} } (\bar{F}^{(1)}_\A)_{ic} \non
\\&\hspace{5cm} - \half  |\nab_{g_0} u_\A|_{g_0} (\bar{F}^{(1)}_\A)_{bc} \dr_i g_0^{bc}  - \dr_i u_\A (\bar{F}^{(1)}_\A)_{bc} (k_0)^{bc} \bigg) \label{M0 III}
\\&\quad  +   \sum_{\A}  \cos\pth{\frac{2u_\A}{\la}}   |\nab_{g_0} u_\A|_{g_0}   \dr_i u_\A F_\A^2 \non
\\&\quad -\frac{1}{4} \sum_{\A\neq\B} \sin\pth{\frac{u_\A}{\la}} \sin\pth{ \frac{u_\B}{\la}}   |\nab_{g_0} u_\B|_{g_0}   \dr_i u_\A \left| \bar{F}^{(1)}_\A \cdot \bar{F}^{(1)}_\B \right|_{g_0}. \non
\end{align}
Collecting \eqref{M0 I}-\eqref{M0 III} and using \eqref{BG momentum} and \eqref{eikonal init} to cancel the non-oscillating term we have proved that $\pth{ \MM(\ga_\la,\ka_\la)_i}^{(0)}$ is indeed given by \eqref{M 0}. It only remains to prove \eqref{M 1}, since \eqref{M 2} follows from the definition of the momentum constraint operator $\MM$. This actually follows from \eqref{H 1} since $\ka_\la$ has the same oscillating behaviour than a derivative of $\ga_\la$. In particular $\pth{ \MM(\ga_\la,\ka_\la)}^{(1)}$ and $\pth{ \HH(\ga_\la,\ka_\la)}^{(1)}$ share the same oscillating behaviour. This concludes the proof of the lemma.
\end{proof}

We now estimate the remainders in \eqref{hamiltonian}-\eqref{momentum}.

\begin{lemma}
We have
\begin{align*}
\RR_\HH\pth{ \ga_\la,\ka_\la,\ffi_\la,X_\la} & = \la\pth{\RR_\HH\pth{ \ga_\la,\ka_\la,\ffi_\la,X_\la}}^{(1)} + \la^2\pth{\RR_\HH\pth{ \ga_\la,\ka_\la,\ffi_\la,X_\la}}^{(\geq 2)},
\\ \RR_\MM\pth{ \ga_\la,\ka_\la,\ffi_\la,X_\la} & = \la\pth{\RR_\MM\pth{ \ga_\la,\ka_\la,\ffi_\la,X_\la}}^{(1)} + \la^2\pth{\RR_\MM\pth{ \ga_\la,\ka_\la,\ffi_\la,X_\la}}^{(\geq 2)},
\end{align*}
where
\begin{align}
\pth{\RR_\HH\pth{ \ga_\la,\ka_\la,\ffi_\la,X_\la}}^{(1)} & = \sum_{\substack{w\in\mathcal{N}\cup\mathcal{I}\\T\in\{\cos,\sin\}}} T\pth{\frac{w}{\la}} \cro{ \pth{1+\bar{F}^{(1)}} \pth{X^{(2,1)} + X^{(2,2)} + X^{(2,\pm)}} }  \label{reste H 1},
\\ \pth{\RR_\HH\pth{ \ga_\la,\ka_\la,\ffi_\la,X_\la}}^{(\geq 2)} & =  \pth{\ffi^{(2)} + \tffi_\la + \la\ffi^{(3)} }\HH\pth{\ga_\la,\ka_\la}  \label{reste H 2}
\\&\hspace{0.5cm} +\Big\{ \pth{ \ffi_\la (\ga_\la^{-1})^2 \pth{ \dr X_\la + \ga_\la^{-1} \dr\ga_\la X_\la} ( \ka_\la + \dr X_\la + \ga_\la^{-1} \dr\ga_\la X_\la)  }^{(\geq 2)}\Big\}^\osc \non,
\end{align}
and
\begin{align}
\pth{\RR_\MM\pth{ \ga_\la,\ka_\la,\ffi_\la,X_\la}}^{(1)} & = \sum_{\substack{w\in\mathcal{N}\cup\mathcal{I}\\T\in\{\cos,\sin\}}} T\pth{\frac{w}{\la}} \cro{ \pth{1+\bar{F}^{(1)}} \pth{\ffi^{(2,1)} + \ffi^{(2,2)} + \ffi^{(2,\pm)}} },  \label{reste M 1}
\\ \pth{\RR_\MM\pth{ \ga_\la,\ka_\la,\ffi_\la,X_\la}}^{(\geq 2)} & = \cro{ \pth{\ffi_\la^{-1} \ga_\la^{-1} \dr \ffi_\la (\ka_\la + \dr X_\la + \ga_\la^{-1} \dr\ga_\la X_\la)}^{(\geq 2)} }^\osc. \label{reste M 2}
\end{align}
\end{lemma}

\begin{proof}
We start with $\RR_\HH\pth{ \ga_\la,\ka_\la,\ffi_\la,X_\la} $ defined in \eqref{reste H}. Since $\ffi_\la-1=\GO{\la^2}$ (see \eqref{def ffi la}) and $\HH\pth{\ga_\la,\ka_\la}=\GO{1}$ (see Lemma \ref{lem constraint expansion H}), $( \ffi_\la-1)\HH\pth{\ga_\la,\ka_\la}$ does not contribute to $\pth{\RR_\HH\pth{ \ga_\la,\ka_\la,\ffi_\la,X_\la}}^{(1)} $. We find that $\pth{\RR_\HH\pth{ \ga_\la,\ka_\la,\ffi_\la,X_\la}}^{(1)} $ is of the schematic form
\begin{align*}
\pth{\RR_\HH\pth{ \ga_\la,\ka_\la,\ffi_\la,X_\la}}^{(1)} & = \pth{ \ffi_\la (\ga_\la^{-1})^2 \pth{ \dr X_\la + \ga_\la^{-1} \dr\ga_\la X_\la} ( \ka_\la + \dr X_\la + \ga_\la^{-1} \dr\ga_\la X_\la)  }^{(1)}
\\& = (g_0^{-1})^2 \pth{ \dr X_\la }^{(1)} ( \ka_\la)^{(0)}
\end{align*}
where we also used $\ga_\la^{-1} \dr\ga_\la X_\la=\GO{\la^2}$ and $\dr X_\la=\GO{\la}$. Using now $\pth{ \dr X_\la }^{(1)}= \pth{ \dr X^{(2)} }^{(-1)}$, \eqref{def X 2} and \eqref{ka la} we finally obtain \eqref{reste H 1}. Moreover, \eqref{reste H 2} can be deduced from the schematic expression of $\RR_\HH\pth{ \ga_\la,\ka_\la,\ffi_\la,X_\la} $. We now turn to $\RR_\MM\pth{ \ga_\la,\ka_\la,\ffi_\la,X_\la} $ defined in \eqref{reste M}.
\begin{align*}
\pth{\RR_\MM\pth{ \ga_\la,\ka_\la,\ffi_\la,X_\la}}^{(1)} & = \pth{\ffi_\la^{-1} \ga_\la^{-1} \dr \ffi_\la (\ka_\la + \dr X_\la + \ga_\la^{-1} \dr\ga_\la X_\la) }^{(1)}
\\& = g_0^{-1} \pth{ \dr \ffi_\la}^{(1)} (\ka_\la)^{(0)}
\end{align*}
where we used $\dr\ffi_\la = \GO{\la}$. Using now $\pth{ \dr \ffi_\la }^{(1)}= \pth{ \dr \ffi^{(2)} }^{(-1)}$, \eqref{def ffi 2} and \eqref{ka la} we finally obtain \eqref{reste M 1}. Moreover, \eqref{reste H 2} can be deduced from the schematic expression of $\RR_\HH\pth{ \ga_\la,\ka_\la,\ffi_\la,X_\la} $. This concludes the proof of the lemma.
\end{proof}

Finally we expand the elliptic operators appearing on the RHS of \eqref{hamiltonian}-\eqref{momentum}.

\begin{lemma}\label{lem exp elliptic}
Let $f$ and $v$ be scalar functions and $Z$ a vector field on $\Si_0$ and let $T\in\{\cos,\sin\}$.
\begin{itemize}
\item[(i)] We have
\begin{align}\label{De ga la f}
\De_{\ga_\la} f & = \De_e f  + \cro{ (\ga_\la^{-1}-e^{-1})\dr^2 f + \ga_\la^{-1}\Ga(\ga_\la) \dr f }^\osc.
\end{align}
and
\begin{align}
\pth{\De_{\ga_\la}\pth{ T\pth{\frac{v}{\la}}f}}^{(-2)} & = T''\pth{\frac{v}{\la}} |\nab_{g_0}v|^2_{g_0} f \label{De ga la -2},
\\ \pth{\De_{\ga_\la}\pth{ T\pth{\frac{v}{\la}}f}}^{(-1)} & = T'\pth{\frac{v}{\la}}  \cro{ \dr^{\leq 1} f   } + \sum_\A T\pth{\frac{v}{\la}}\cos\pth{\frac{u_\A}{\la}} \cro{\bar{F}^{(1)} f} ,\label{De ga la -1}
\\ \pth{\De_{\ga_\la}\pth{ T\pth{\frac{v}{\la}}f}}^{(\geq 0)} & = \cro{\pth{\ga_\la^{-1}\dr^2 f + \ga_\la^{-1}\ga_\la^{-1}\dr\ga_\la\dr f}^{(\geq 0)}}^\osc.\label{De ga la 0}
\end{align}
\item[(ii)] We have
\begin{align}
\pth{\vDe_{\ga_\la} Z_i}^{(-1)} & = \sum_\A  \cos\pth{\frac{u_\A}{\la}}|\nab_{g_0} u_\A|^2_{g_0} (\bar{F}^{(1)}_\A)_{ij} Z^j , \label{vDe Z -1}
\\ \pth{\vDe_{\ga_\la} Z_i}^{(\geq 0)} & =  \De_eZ_i  + \Big\{(\ga_\la^{-1}-e^{-1})\dr^2 Z + \ga_\la^{-1}\Ga(\ga_\la) \dr Z \label{vDe Z 0}
\\&\hspace{2cm} + \pth{\pth{\ga_\la^{-1}\dr\Ga(\ga_\la)}^{(\geq 0)} + \ga_\la^{-1} \Ga(\ga_\la)\Ga(\ga_\la)} Z \Big\}^\osc_i. \non
\end{align}
and
\begin{align}
\pth{\vDe_{\ga_\la}\pth{T\pth{\frac{v}{\la}}Z}_i }^{(-2)} & = T''\pth{\frac{v}{\la}} |\nab_{g_0}v|^2_{g_0} Z_i \label{vDe osc -2}
\\ \pth{\vDe_{\ga_\la}\pth{T\pth{\frac{v}{\la}}Z}_i }^{(-1)} & =  T'\pth{\frac{v}{\la}}\cro{ \dr^{\leq 1} Z }_i + \sum_\A \cos\pth{\frac{u_\A}{\la}} T\pth{\frac{v}{\la}}\cro{ Z \bar{F}^{(1)} }_i \label{vDe osc -1}
\end{align}
and
\begin{equation}\label{vDe osc 0}
\begin{aligned}
&\pth{\vDe_{\ga_\la}\pth{T\pth{\frac{v}{\la}}Z}_i }^{(\geq 0)} 
\\& =   \Big\{ \pth{ \pth{\ga_\la^{-1}}^{(\geq 2)} + \pth{ \ga_\la^{-1} \Ga(\ga_\la)}^{(\geq 1)} + \pth{\ga_\la^{-1} \dr \Ga(\ga_\la)}^{(\geq 0)} + \ga_\la^{-1}\Ga(\ga_\la)\Ga(\ga_\la)} Z
\\&\hspace{5cm} +\pth{ \pth{\ga_\la^{-1}}^{(\geq 1)} + \ga_\la^{-1} \Ga(\ga_\la)}\dr^{\leq 1} Z  + \ga_\la^{-1}\dr^2 Z \Big\}_i^{\osc} .
\end{aligned}
\end{equation}
\end{itemize}
\end{lemma}

\begin{proof}
The identity \eqref{De ga la f} follows directly from the expression of the Laplace-Beltrami operator in coordinates. The following computation is the elliptic equivalent of \eqref{Box T f}:
\begin{align*}
\De_{\ga_\la}\pth{ T\pth{\frac{v}{\la}}f} & =  \frac{1}{\la^2} T''\pth{\frac{v}{\la}} \ga_\la^{-1}(\d v,\d v) f +  \frac{1}{\la}T'\pth{\frac{v}{\la}}  \pth{ 2\ga_\la^{ij}\dr_j v \dr_i f + (\De_{\ga_\la} v) f   }  +T\pth{\frac{v}{\la}}\De_{\ga_\la}f.
\end{align*}
We use the expansion of the inverse of $\ga_\la$ given by \eqref{ga-1} to obtain \eqref{De ga la -2}-\eqref{De ga la 0}. We now turn to the elliptic operator $\vDe_{\ga_\la}$ acting on vector fields. Thanks to the expression of the Ricci tensor in coordinates we find
\begin{equation}\label{vDe coordinates}
\begin{aligned}
\vDe_{g}Z_i & = g^{k\ell}  \dr_k \dr_\ell Z_i  - 2g^{k\ell} \Ga(g)_{\ell i}^a \dr_k Z_a - g^{k\ell} Z_a \pth{ \dr_k  \Ga(g)_{\ell i}^a - \Ga(g)_{ki}^{b}\Ga(g)_{\ell b}^a  }
\\&\quad + Z^{j}\dr_a \Ga(g)^a_{ij} - Z^{j} \dr_j \Ga(g)^a_{ai} + Z^{j}\Ga(g)^a_{ab} \Ga(g)^b_{ij} - Z^{j}\Ga(g)^a_{ib} \Ga(g)^b_{aj}.
\end{aligned}
\end{equation}
This already gives
\begin{align*}
\pth{ \vDe_{\ga_\la}Z_i }^{(-1)} & =  Z_b\pth{ - g_0^{k\ell}  \dr_k (\tilde{\Ga}^{(0)})_{\ell i}^b + g_0^{jb}\pth{\dr_a (\tilde{\Ga}^{(0)})^a_{ij} - \dr_j (\tilde{\Ga}^{(0)})^a_{ai}}}^{(-1)}
\\& = \sum_\A  \cos\pth{\frac{u_\A}{\la}}|\nab u_\A|^2 (\bar{F}^{(1)}_\A)_{ij} Z^j.
\end{align*}
Moreover this also gives \eqref{vDe Z 0}. Thanks to \eqref{vDe coordinates} we finally obtain
\begin{align*}
\vDe_{\ga_\la}\pth{T\pth{\frac{v}{\la}}Z}_i & = \frac{1}{\la^2} \ga_\la^{k\ell}  T''\pth{\frac{v}{\la}}\dr_k v \dr_\ell v Z_i 
\\&\quad +  \frac{1}{\la}T'\pth{\frac{v}{\la}}\pth{2 \ga_\la^{k\ell} \dr_\ell v \dr_k Z_i + \ga_\la^{k\ell}  \dr_k\dr_\ell v Z_i - 2 \ga_\la^{k\ell} \Ga(\ga_\la)_{\ell i}^a \dr_k v Z_a }
\\&\quad +  T\pth{\frac{v}{\la}}\pth{Z^{j}\dr_a \Ga(\ga_\la)^a_{ij} - Z^{j} \dr_j \Ga(\ga_\la)^a_{ai}  - \ga_\la^{k\ell} Z_a  \dr_k  \Ga(\ga_\la)_{\ell i}^a  }
\\&\quad + T\pth{\frac{v}{\la}}\pth{\ga_\la^{k\ell} Z_a \Ga(\ga_\la)_{ki}^{b}\Ga(\ga_\la)_{\ell b}^a  + Z^{j}\Ga(\ga_\la)^a_{ab} \Ga(\ga_\la)^b_{ij} -Z^{j}\Ga(\ga_\la)^a_{ib} \Ga(\ga_\la)^b_{aj}}
\\&\quad + T\pth{\frac{v}{\la}}\pth{ - 2\ga_\la^{k\ell} \Ga(\ga_\la)_{\ell i}^a \dr_kZ_a  + \ga_\la^{k\ell} \dr_k\dr_\ell Z_i }.
\end{align*}
We use the expansion of the inverse of $\ga_\la$ given by \eqref{ga-1} and of the Christofel symbols given by \eqref{christofel symbols} to deduce \eqref{vDe osc -2}-\eqref{vDe osc 0}. This concludes the proof of the lemma.
\end{proof}

We are now ready to define the constraint hierarchy, i.e to deduce from Lemmas \ref{lem constraint expansion H}-\ref{lem exp elliptic} a hierarchy of equations for the various terms in $\ffi_\la$ and $X_\la$ such that they solve \eqref{hamiltonian}-\eqref{momentum}. The hierarchy associated to the Hamiltonian constraint is
\begin{align}
8 \pth{ \De_{\ga_\la}\ffi^{(2)}}^{(-2)} & = \pth{ \HH(\ga_\la,\ka_\la) }^{(0)},\label{eq ffi2}
\\ 8 \pth{ \De_{\ga_\la}\ffi^{(3)}}^{(-2)} & =  \pth{ \HH(\ga_\la,\ka_\la) }^{(1)} + \pth{\RR_\HH\pth{ \ga_\la,\ka_\la,\ffi_\la,X_\la}}^{(1)} - 8 \pth{ \De_{\ga_\la}\ffi^{(2)}}^{(-1)},\label{eq ffi3}
\\ 8 \De_{\ga_\la}\tffi_\la & = \pth{\HH\pth{\ga_\la,\ka_\la}}^{(\geq 2)} + \pth{\RR_\HH\pth{ \ga_\la,\ka_\la,\ffi_\la,X_\la}}^{(\geq 2)} - 8 \pth{ \De_{\ga_\la}\ffi^{(2)}}^{(\geq 0)}  \label{eq tffi}
\\&\quad - 8 \pth{ \De_{\ga_\la}\ffi^{(3)}}^{(\geq -1)}.\non
\end{align}
The hierarchy associated to the momentum constraint is
\begin{align}
\pth{\vDe_{\ga_\la}X^{(2)} }^{(-2)} & = - \pth{ \MM(\ga_\la,\ka_\la) }^{(0)}, \label{eq X2}
\\ \pth{\vDe_{\ga_\la}X^{(3)} }^{(-2)} & = - \pth{ \MM(\ga_\la,\ka_\la) }^{(1)} + \pth{\RR_\MM\pth{ \ga_\la,\ka_\la,\ffi_\la,X_\la}}^{(1)} - \pth{\vDe_{\ga_\la}X^{(2)} }^{(-1)} \label{eq X3}
\\&\quad - \pth{\vDe_{\ga_\la}\tX_\la}^{(-1)}, \non
\\ \pth{\vDe_{\ga_\la}\tX_\la}^{(\geq 0)} & =  -\pth{\MM(\ga_\la,\ka_\la)}^{(\geq 2)}  + \pth{\RR_\MM\pth{ \ga_\la,\ka_\la,\ffi_\la,X_\la} }^{(\geq 2)} - \pth{\vDe_{\ga_\la}X^{(2)} }^{(\geq 0)} \label{eq tX}
\\&\quad  - \pth{\vDe_{\ga_\la}X^{(3)} }^{(\geq -1)} .\non
\end{align}

\begin{remark}\label{remark couplage}
The momentum hierarchy \eqref{eq X2}-\eqref{eq tX} differs from the Hamiltonian hierarchy \eqref{eq ffi2}-\eqref{eq tffi} because of the term $\pth{\vDe_{\ga_\la}\tX_\la}^{(-1)}$ in \eqref{eq X3}. The presence of this term is due to the fact that the operator $\vDe_{\ga_\la}$ acting on vector fields contains second derivatives of $\ga_\la$ and thus looses one power of $\la$ even when applied to a non-oscillating vector field such as $\tX_\la$. This is not the case for the operator $\De_{\ga_\la}$ acting on functions since it only contains up to first derivatives of $\ga_\la$.
\end{remark}

\subsection{Solving the hierarchy}\label{section solving constraint}

In this section we construct a solution $\pth{\ffi^{(2)}, \ffi^{(3)}, \tffi, X^{(2)}, X^{(3)}, \tX}$ of the equations \eqref{eq ffi2}-\eqref{eq tX}. 

\paragraph{The $\la^0$ equations.} 

We start by defining $\ffi^{(2)}$ and $X^{(2)}$ such that they solve \eqref{eq ffi2} and \eqref{eq X2}. On the one hand, \eqref{def ffi 2}-\eqref{def X 2} together with \eqref{De ga la -2} and \eqref{vDe osc -2} imply
\begin{align*}
8 \pth{ \De_{\ga_\la}\ffi^{(2)}}^{(-2)} & = - 8 \sum_\A  \sin\pth{\frac{u_\A}{\la}} |\nab_{g_0}u_\A|^2_{g_0} \ffi^{(2,1)}_\A -32 \sum_\A   \cos\pth{\frac{2u_\A}{\la}} |\nab_{g_0}u_\A|^2_{g_0} \ffi^{(2,2)}_\A
\\&\quad - 8 \sum_{\A\neq \B,\pm} \cos\pth{\frac{u_\A\pm u_\B}{\la}} |\nab_{g_0}(u_\A\pm u_\B)|^2_{g_0} \ffi^{(2,\pm)}_{\A\B},
\\ \pth{\vDe_{\ga_\la}X^{(2)} }^{(-2)}_i & = -\sum_\A  \sin\pth{\frac{u_\A}{\la}} |\nab_{g_0}u_\A|^2_{g_0} (X^{(2,1)}_\A)_i  -4 \sum_\A \cos\pth{\frac{2u_\A}{\la}}|\nab_{g_0}u_\A|^2_{g_0}(X^{(2,2)}_\A)_i  
\\&\quad - \sum_{\A\neq \B,\pm} \cos\pth{\frac{u_\A\pm u_\B}{\la}}|\nab_{g_0}(u_\A\pm u_\B)|^2_{g_0}(X^{(2,\pm)}_{\A\B})_i.
\end{align*}
On the other hand, \eqref{H 0} and \eqref{M 0} give the expression of the RHS of \eqref{eq ffi2} and \eqref{eq X2}. Therefore, solving \eqref{eq ffi2} and \eqref{eq X2} is equivalent to defining $\ffi^{(2,1)}_\A$, $ \ffi^{(2,2)}_\A$, $\ffi^{(2,\pm)}_{\A\B}$ by
\begin{align}
 \ffi^{(2,1)}_\A & \vcentcolon = \wc{\ffi}^{(2,1)}_{\A} , \label{ffi 2 1 def}
\\  \ffi^{(2,2)}_\A & \vcentcolon =  \frac{3}{16}  F_\A^2 , \label{ffi 2 2 def}
\\   \ffi^{(2,\pm)}_{\A\B} & \vcentcolon = -\frac{1}{8} \pth{\tr_{g_0}\ga^{(2,\pm)}_{\A\B} -(\ga^{(2,\pm)}_{\A\B})_{N^{(\pm)}_{\A\B}N^{(\pm)}_{\A\B}}} +  \wc{\ffi}^{(2,\pm)}_{\A\B}, \label{ffi 2 pm def}
\end{align}
and $(X^{(2,1)}_\A)_i$, $(X^{(2,2)}_\A)_i  $,  $(X^{(2,\pm)}_{\A\B})_i$ by
\begin{align}
(X^{(2,1)}_\A)_i & \vcentcolon = \frac{1}{ |\nab_{g_0}u_\A|^2_{g_0}} \pth{ \dr_i u_\A \tr_{g_0}\ka^{(1,1)}_\A - (\ka^{(1,1)}_\A)_{\nab_{g_0} u_\A i}  +  (\wc{X}^{(2,1)}_\A)_i }, \label{X 2 1 def}
\\ (X^{(2,2)}_\A)_i  & \vcentcolon = \frac{1}{4 |\nab_{g_0}u_\A|^2_{g_0}} \pth{ 2 (\ka^{(1,2)}_\A)_{\nab_{g_0} u_\A i} - 2 \dr_i u_\A \tr_{g_0}\ka^{(1,2)}_\A + 3 |\nab_{g_0} u_\A|_{g_0}  \dr_i u_\A F_\A^2},
\\ (X^{(2,\pm)}_{\A\B})_i & \vcentcolon = \frac{1}{|\nab_{g_0}(u_\A\pm u_\B)|^2_{g_0}}\pth{ (\ka^{(1,\pm)}_{\A\B})_{(\nab_{g_0} u_\A \pm \nab_{g_0} u_\B)i} - \dr_i(u_\A\pm u_\B)\tr_{g_0}  \ka^{(1,\pm)}_{\A\B} + (\wc{X}^{(2,\pm)}_{\A\B})_i } . \label{X 2 pm def}
\end{align}
where we recognized the expression of $\wc{\ffi}^{(2,1)}_{\A}$, $ \wc{\ffi}^{(2,\pm)}_{\A\B}$, $\wc{X}^{(2,1)}_\A$ and $\wc{X}^{(2,\pm)}_{\A\B}$ given by \eqref{wcffi 21}-\eqref{wcX 2pm}.

\paragraph{The $\la^1$ equations.} We now define $\ffi^{(3)}$ and $X^{(3)}$ such that they solve \eqref{eq ffi3} and \eqref{eq X3}. Thanks to \eqref{H 1}, \eqref{reste H 1}, \eqref{def ffi 2} and \eqref{De ga la -1} the equation \eqref{eq ffi3} rewrites
\begin{align}
\pth{ \De_{\ga_\la}\ffi^{(3)}}^{(-2)} & = \sum_{\substack{w\in\mathcal{N}\cup\mathcal{I}\\T\in\{\cos,\sin\}}} T\pth{\frac{w}{\la}} S^\HH_{w,T} , \label{eq ffi3 bis}
\end{align}
where each $S^\HH_{w,T}$ is of the form
\begin{align*}
S^\HH_{w,T} & = \bigg\{ \pth{ \ga_\la^{-1}\dr^2\ga_\la+ \ga_\la^{-1}\ga_\la^{-1}\pth{(\dr\ga_\la)^2 + (\ka_\la)^2} }^{(1)}  +  \dr\pth{  \ffi^{(2,1)}  +  \ffi^{(2,2)} + \ffi^{(2,\pm)} }
\\&\hspace{2cm} +  \pth{1+\bar{F}^{(1)}} \pth{X^{(2,1)} + X^{(2,2)} + X^{(2,\pm)} + \ffi^{(2,1)}  +  \ffi^{(2,2)} + \ffi^{(2,\pm)}}    \bigg\}.
\end{align*}
We now define $\ffi^{(3)}$ by
\begin{align}
\ffi^{(3)} & \vcentcolon = - \sum_{\substack{w\in\mathcal{N}\cup\mathcal{I}\\T\in\{\cos,\sin\}}} T\pth{\frac{w}{\la}}  \frac{S^\HH_{w,T}}{|\nab_{g_0}w|^2_{g_0} } . \label{def ffi 3}
\end{align}
Using \eqref{De ga la -2} and $T''=-T$ for $T\in\{\cos,\sin\}$ one can check that $\ffi^{(3)}$ solves \eqref{eq ffi3 bis} and thus \eqref{eq ffi3}. We now turn to \eqref{eq X3}, which thanks to \eqref{M 1}, \eqref{reste M 1}, \eqref{def X 2}, \eqref{vDe Z -1} and \eqref{vDe osc -1} rewrites
\begin{align}\label{eq X3 bis}
\pth{\vDe_{\ga_\la}X^{(3)} }^{(-2)}_i & = \sum_{\substack{w\in\mathcal{N}\cup\mathcal{I}\\T\in\{\cos,\sin\}}} T\pth{\frac{w}{\la}}  (S^\MM_{w,T})_i - \sum_\A  \cos\pth{\frac{u_\A}{\la}}|\nab_{g_0} u_\A|^2_{g_0} (\bar{F}^{(1)}_\A)_{ij} \tX_\la^j ,
\end{align}
where each $S^\MM_{w,T}$ is of the form
\begin{align*}
S^\MM_{w,T} & =  \bigg\{ \pth{ \ga_\la^{-1}\dr\ka_\la + \ga_\la^{-1}\ga_\la^{-1} \ka_\la\dr\ga_\la }^{(1)}+  \dr\pth{  X^{(2,1)}  +  X^{(2,2)} + X^{(2,\pm)}} 
\\&\hspace{1cm} +  \pth{1+\bar{F}^{(1)}} \pth{X^{(2,1)} + X^{(2,2)} + X^{(2,\pm)}+\ffi^{(2,1)} + \ffi^{(2,2)} + \ffi^{(2,\pm)}}  \bigg\}.
\end{align*}
In order to solve this equation, we further decompose $X^{(3)}$ into 
\begin{align}\label{def X 3}
X^{(3)}_i & \vcentcolon = - \sum_{\substack{w\in\mathcal{N}\cup\mathcal{I}\\T\in\{\cos,\sin\}}} T\pth{\frac{w}{\la}} \frac{ (S^\MM_{w,T})_i }{ |\nab_{g_0}w|^2_{g_0}} + \sum_\A \cos\pth{\frac{u_\A}{\la}} (\bar{F}^{(1)}_\A)_{ij} \tX_\la^j .
\end{align}
Using \eqref{vDe osc -2} and $T''=-T$ for $T\in\{\cos,\sin\}$ one can check that $X^{(3)}$ solves \eqref{eq X3 bis} and thus \eqref{eq X3}. Following Remark \ref{remark couplage}, note that $X^{(3)}$ depends on the remainder $\tX_\la$.

\paragraph{The $\la^2$ equations.} We conclude the resolution of \eqref{hamiltonian}-\eqref{momentum} by solving for $\tffi_\la$ and $\tX_\la$. This step is identical to the singlephase case considered in \cite{Touati2023b} since the multiphase aspect of the current construction does not change the nature of the equations for the remainders. We point out two harmless difference:
\begin{itemize}
\item The conformal formulation here is slightly different from the one used in \cite{Touati2023b}. In particular the elliptic operator for the vector equation is different but as in \cite{Touati2023b} we will benefit from \eqref{vDe Z 0} and invert $\De_e$.
\item The free tensors $\ga^{(2,\pm)}_{\A\B}$, $\ka^{(1,1)}_\A$, $\ka^{(1,2)}_\A$ and $\ka^{(1,\pm)}_{\A\B}$ are not strictly speaking present in \cite{Touati2023b} but they can easily be estimated with the assumption \eqref{estim ga kappa param}.
\end{itemize}
Therefore, following the same steps as Section 6 of \cite{Touati2023b} we obtain the existence of a unique couple $\pth{\tffi_\la,\tX_\la}\in\pth{H^{N-3}_\de}^2$ solving \eqref{hamiltonian}-\eqref{momentum} and satisfying moreover
\begin{align}\label{estim tffi tX}
\l \tffi_\la \r_{H^{k+2}_\de} + \l \tX_\la \r_{H^{k+2}_\de} \lesssim \frac{\e}{\la^k}
\end{align}
for $k\in\llbracket 0,N-5\rrbracket$.

\subsection{Final form of the solution}

In the previous section we have a constructed a solution $(\ffi_\la,X_\la)$ to the system \eqref{hamiltonian}-\eqref{momentum}. Therefore we have constructed a solution of the constraint equations $(g_\la,k_\la)$ of the form \eqref{conformal ansatz}. To conclude the proof of Proposition \ref{prop constraint main} we need to recover the expressions \eqref{g la}-\eqref{k la} as well as the estimates \eqref{estim reste cons 1}-\eqref{estim reste cons 2}.

\paragraph{The induced metric.} We have $g_\la=\ffi_\la^4\ga_\la$, so that \eqref{ga la} and \eqref{def ffi la} implies
\begin{align*}
g_\la & = g_0 + \la \sum_\A \cos\pth{\frac{u_\A}{\la}} \bar{F}^{(1)}_\A + 4\la^2\pth{ \ffi^{(2)} + \tffi_\la}g_0 + \la^2 \sum_{\A\neq \B,\pm}\cos\pth{\frac{u_\A\pm u_\B}{\la}}\ga^{(2,\pm)}_{\A\B}
\\&\quad + \la^3 \cro{ \pth{\ffi^{(2)} + \tffi_\la + \pth{\ffi_\la^4}^{(\geq 3)}}\pth{\bar{F}^{(1)}+\ga^{(2,\pm)}}  }^\osc.
\end{align*}
Using now \eqref{def ffi 2} we obtain
\begin{align*}
g_\la & = g_0 + \la \sum_\A \cos\pth{\frac{u_\A}{\la}} \bar{F}^{(1)}_\A + \la^2 \sum_\A  \sin\pth{\frac{u_\A}{\la}} 4\ffi^{(2,1)}_\A g_0 + \la^2 \sum_\A \cos\pth{\frac{2u_\A}{\la}}4\ffi^{(2,2)}_\A g_0 
\\&\quad + \la^2 \sum_{\A\neq \B,\pm} \cos\pth{\frac{u_\A\pm u_\B}{\la}}\pth{ \ga^{(2,\pm)}_{\A\B} + 4\ffi^{(2,\pm)}_{\A\B}g_0 }
\\&\quad + \la^2 \tffi_\la g_0 + \la^3 \cro{ \pth{\ffi^{(2)} + \tffi_\la + \pth{\ffi_\la^4}^{(\geq 3)}}\pth{\bar{F}^{(1)}+\ga^{(2,\pm)}}  }^\osc.
\end{align*}
Using now \eqref{ffi 2 pm def} we obtain
\begin{align*}
\ga^{(2,\pm)}_{\A\B} + 4\ffi^{(2,\pm)}_{\A\B}g_0 & = \bar{\PP}^{[1]}_{u_\A\pm u_\B}\pth{\ga^{(2,\pm)}_{\A\B}} + 4 \wc{\ffi}^{(2,\pm)}_{\A\B} g_0,
\end{align*}
where we recognized the operator $\bar{\PP}^{[1]}_{u_\A\pm u_\B}$ defined by \eqref{def P1(T)}. Using also \eqref{ffi 2 1 def} and \eqref{ffi 2 2 def}, and defining 
\begin{align*}
h_\la & \vcentcolon = \tffi_\la g_0 + \la \cro{ \pth{\ffi^{(2)} + \tffi_\la + \pth{\ffi_\la^4}^{(\geq 3)}}\pth{\bar{F}^{(1)}+\ga^{(2,\pm)}}  }^\osc
\end{align*}
allows us to recover the expression \eqref{g la} for $g_\la$. The estimate \eqref{estim reste cons 1} follows from \eqref{estim seed}, \eqref{estim ga kappa param}, the definition of $\ffi^{(2)}$ and $\ffi^{(3)}$ and \eqref{estim tffi tX}.

\paragraph{The second fundamental form.} We have
\begin{align*}
k_\la & = \ffi^2_\la\pth{  \ka_\la  +  \L_{\ga_\la} X_\la} .
\end{align*}
We first expand $\L_{\ga_\la}X_\la$. Thanks to \eqref{def LX}, \eqref{def X la} and \eqref{def X 2} we have
\begin{align*}
\L_{\ga_\la}X_\la & =   \la \sum_\A \cos\pth{\frac{u_\A}{\la}} \nab_{g_0}u_\A \tilde{\otimes} X^{(2,1)}_\A -2\la\sum_\A \sin\pth{\frac{2u_\A}{\la}}\nab_{g_0}u_\A \tilde{\otimes} X^{(2,2)}_\A
\\&\quad  - \la \sum_{\A\neq \B,\pm} \sin\pth{\frac{u_\A\pm u_\B}{\la}} \nab_{g_0}(u_\A\pm u_\B) \tilde{\otimes} X^{(2,1)}_\A 
\\&\quad + \la^2 \cro{\pth{ \pth{1+\ga_\la^{-1}\ga_\la}\pth{\dr X_\la + \ga_\la^{-1} X_\la\dr\ga_\la} }^{(\geq 2)}}^\osc,
\end{align*}
where we also used the notation introduced in \eqref{def produit}. Using now also \eqref{ka la} we obtain
\begin{align*}
k_\la & = k_0 + \sum_\A \sin\pth{ \frac{u_\A}{\la}}  \half |\nab u_\A| \bar{F}^{(1)}_\A   + \la\sum_\A\cos\pth{\frac{u_\A}{\la}}\pth{ \ka^{(1,1)}_\A  + \nab_{g_0}u_\A \tilde{\otimes} X^{(2,1)}_\A}
\\&\quad + \la\sum_\A \sin\pth{\frac{2u_\A}{\la}} \pth{ \ka^{(1,2)}_\A -2 \nab_{g_0}u_\A \tilde{\otimes} X^{(2,2)}_\A} 
\\&\quad + \la\sum_{\A\neq\B,\pm}\sin\pth{\frac{u_\A\pm u_\B}{\la}}\pth{ \ka^{(1,\pm)}_{\A\B}  - \nab_{g_0}(u_\A\pm u_\B) \tilde{\otimes} X^{(2,1)}_\A }
\\&\quad + \la^2 \cro{ \pth{ \ka_\la  +  \L_{\ga_\la} X_\la}^{(\geq 2)} +\pth{\ffi^2_\la}^{(\geq 2)}\pth{ \ka_\la  +  \L_{\ga_\la} X_\la}  }^\osc.
\end{align*}
We use now \eqref{X 2 1 def}-\eqref{X 2 pm def} to obtain
\begin{align*}
\ka^{(1,1)}_\A  + \nab_{g_0}u_\A \tilde{\otimes} X^{(2,1)}_\A & = \bar{\PP}^{[2]}_{u_\A}\pth{\ka^{(1,1)}_\A} + \frac{N_\A \tilde{\otimes} \wc{X}^{(2,1)}_\A}{ |\nab_{g_0}u_\A|_{g_0}} ,
\\ \ka^{(1,2)}_\A -2 \nab_{g_0}u_\A \tilde{\otimes} X^{(2,2)}_\A & = \bar{\PP}^{[2]}_{u_\A}\pth{\ka^{(1,2)}_\A}  -\frac{3}{2} |\nab_{g_0} u_\A|_{g_0} F_\A^2 N_\A\tilde{\otimes} N_\A,
\\ \ka^{(1,\pm)}_{\A\B}  - \nab_{g_0}(u_\A\pm u_\B) \tilde{\otimes} X^{(2,1)}_\A & =  \bar{\PP}^{[2]}_{u_\A\pm u_\B}\pth{\ka^{(1,\pm)}_{\A\B}} - \frac{N^{(\pm)}_{\A\B}\tilde{\otimes} \wc{X}^{(2,\pm)}_{\A\B}}{|\nab_{g_0}\pth{ u_\A\pm u_\B}|_{g_0}},
\end{align*}
where we recognized the operators $\bar{\PP}^{[2]}_{u_\A}$ and $\bar{\PP}^{[2]}_{u_\A\pm u_\B}$ defined by \eqref{def P2(T)}. Defining 
\begin{align*}
\tilde{k}^{cons}_\la & \vcentcolon = \cro{ \pth{ \ka_\la  +  \L_{\ga_\la} X_\la}^{(\geq 2)} +\pth{\ffi^2_\la}^{(\geq 2)}\pth{ \ka_\la  +  \L_{\ga_\la} X_\la}  }^\osc
\end{align*}
allows us to recover the expression \eqref{k la} for $k_\la$. The estimate \eqref{estim reste cons 2} follows from \eqref{estim seed}, \eqref{estim ga kappa param}, the definition of $\ffi^{(2)}$, $\ffi^{(3)}$, $X^{(2)}$, $X^{(3)}$ and \eqref{estim tffi tX}.


\begin{thebibliography}{MTW73}

\bibitem[Bur89]{Burnett1989}
Gregory~A. Burnett.
\newblock The high-frequency limit in general relativity.
\newblock {\em J. Math. Phys.}, 30(1):90--96, 1989.

\bibitem[CB69]{ChoquetBruhat1969}
Yvonne Choquet-Bruhat.
\newblock Construction de solutions radiatives approch\'{e}es des \'{e}quations
  d'{E}instein.
\newblock {\em Comm. Math. Phys.}, 12:16--35, 1969.

\bibitem[CB00]{ChoquetBruhat2000}
Yvonne Choquet-Bruhat.
\newblock The null condition and asymptotic expansions for the {Einstein}
  equations.
\newblock {\em Ann. Phys. (8)}, 9(3-5):258--266, 2000.

\bibitem[CB09]{ChoquetBruhat2009}
Yvonne Choquet-Bruhat.
\newblock {\em General relativity and the {E}instein equations}.
\newblock Oxford Mathematical Monographs. Oxford University Press, Oxford,
  2009.

\bibitem[CBF06]{ChoquetBruhat2006}
Yvonne Choquet-Bruhat and Helmut Friedrich.
\newblock Motion of isolated bodies.
\newblock {\em Classical Quantum Gravity}, 23(20):5941--5949, 2006.

\bibitem[Chr86]{Christodoulou1986}
Demetrios Christodoulou.
\newblock Global solutions of nonlinear hyperbolic equations for small initial
  data.
\newblock {\em Comm. Pure Appl. Math.}, 39(2):267--282, 1986.

\bibitem[CK93]{Christodoulou1993}
Demetrios Christodoulou and Sergiu Klainerman.
\newblock {\em The global nonlinear stability of the {M}inkowski space},
  volume~41 of {\em Princeton Mathematical Series}.
\newblock Princeton University Press, Princeton, NJ, 1993.

\bibitem[CS06]{Corvino2006}
Justin Corvino and Richard~M. Schoen.
\newblock On the asymptotics for the vacuum {Einstein} constraint equations.
\newblock {\em J. Differ. Geom.}, 73(2):185--217, 2006.

\bibitem[GdC21]{Guerra2021}
André Guerra and Rita~Teixeira da~Costa.
\newblock Oscillations in wave map systems and homogenization of the einstein
  equations in symmetry.
\newblock {\em arXiv:2107.00942}, 2021.

\bibitem[GW11]{Green2011}
Stephen~R. Green and Robert~M. Wald.
\newblock New framework for analyzing the effects of small scale
  inhomogeneities in cosmology.
\newblock {\em Physical Review D}, 83(8), Apr 2011.

\bibitem[HL18]{Huneau2018a}
C\'{e}cile Huneau and Jonathan Luk.
\newblock High-frequency backreaction for the {E}instein equations under
  polarized {$\Bbb{U}(1)$}-symmetry.
\newblock {\em Duke Math. J.}, 167(18):3315--3402, 2018.

\bibitem[HL19]{Huneau2019}
Cécile Huneau and Jonathan Luk.
\newblock Trilinear compensated compactness and {B}urnett's conjecture in
  general relativity.
\newblock {\em arXiv:1907.10743}, 2019.

\bibitem[HL24a]{Huneau2024}
C\'{e}cile Huneau and Jonathan Luk.
\newblock High-frequency backreaction for the {E}instein equations under
  {$\Bbb{U}(1)$} symmetry: from {E}instein-dust to {E}instein-{V}lasov.
\newblock {\em In preparation}, 2024.

\bibitem[HL24b]{Huneau2024a}
C\'{e}cile Huneau and Jonathan Luk.
\newblock High-frequency solutions to the {Einstein} equations.
\newblock {\em Classical and Quantum Gravity}, 41(14)3002, 2024.

\bibitem[Jea02]{Jeanne2002}
Pierre-Yves Jeanne.
\newblock Geometric optics for gauge invariant semilinear systems.
\newblock {\em M{\'e}m. Soc. Math. Fr., Nouv. S{\'e}r.}, 90:vi + 160, 2002.

\bibitem[JMR93]{Joly1993}
Jean-Luc Joly, Guy Metivier, and Jeffrey Rauch.
\newblock Generic rigorous asymptotic expansions for weakly nonlinear
  multidimensional oscillatory waves.
\newblock {\em Duke Math. J.}, 70(2):373--404, 1993.

\bibitem[JMR00]{Joly2000}
Jean-Luc Joly, Guy Metivier, and Jeffrey Rauch.
\newblock Transparent nonlinear geometric optics and {Maxwell}-{Bloch}
  equations.
\newblock {\em J. Differ. Equations}, 166(1):175--250, 2000.

\bibitem[Kla86]{Klainerman1986}
Sergiu Klainerman.
\newblock The null condition and global existence to nonlinear wave equations.
\newblock In {\em Nonlinear systems of partial differential equations in
  applied mathematics, {P}art 1 ({S}anta {F}e, {N}.{M}., 1984)}, volume~23 of
  {\em Lectures in Appl. Math.}, pages 293--326. Amer. Math. Soc., Providence,
  RI, 1986.

\bibitem[KRS15]{Klainerman2015}
Sergiu Klainerman, Igor Rodnianski, and Jérémie Szeftel.
\newblock The bounded {$L^2$} curvature conjecture.
\newblock {\em Invent. Math.}, 202(1):91--216, 2015.

\bibitem[Lan13]{Lannes2013}
David Lannes.
\newblock Space time resonances [after {Germain,} {Masmoudi,} {Shatah]}.
\newblock In {\em S\'eminaire Bourbaki volume 2011/2012 expos\'es 1043-1058},
  number 352 in Ast\'erisque. Soci\'et\'e math\'ematique de France, 2013.

\bibitem[LR03]{Lindblad2003}
Hans Lindblad and Igor Rodnianski.
\newblock The weak null condition for {Einstein}'s equations.
\newblock {\em C. R., Math., Acad. Sci. Paris}, 336(11):901--906, 2003.

\bibitem[LR10]{Lindblad2010}
Hans Lindblad and Igor Rodnianski.
\newblock The global stability of {M}inkowski space-time in harmonic gauge.
\newblock {\em Ann. of Math. (2)}, 171(3):1401--1477, 2010.

\bibitem[LR17]{Luk2017}
Jonathan Luk and Igor Rodnianski.
\newblock Nonlinear interaction of impulsive gravitational waves for the vacuum
  {E}instein equations.
\newblock {\em Camb. J. Math.}, 5(4):435--570, 2017.

\bibitem[LR20]{Luk2020}
Jonathan Luk and Igor Rodnianski.
\newblock High-frequency limits and null dust shell solutions in general
  relativity.
\newblock {\em arXiv:2009.08968}, 2020.

\bibitem[Mé09]{Metivier2009}
Guy Métivier.
\newblock The mathematics of nonlinear optics.
\newblock In C.M. Dafermos and Milan Pokorný, editors, {\em Handbook of
  Differential Equations}, volume~5 of {\em Handbook of Differential Equations:
  Evolutionary Equations}, pages 169--313. North-Holland, 2009.

\bibitem[MTW73]{Misner1973}
Charles~W. Misner, Kip~S. Thorne, and John~Archibald Wheeler.
\newblock {\em Gravitation}.
\newblock W. H. Freeman and Co., San Francisco, Calif., 1973.

\bibitem[Rau12]{Rauch2012}
Jeffrey Rauch.
\newblock {\em Hyperbolic partial differential equations and geometric optics},
  volume 133 of {\em Graduate Studies in Mathematics}.
\newblock American Mathematical Society, Providence, RI, 2012.

\bibitem[Sze18]{Szeftel2018}
J{\'e}r{\'e}mie Szeftel.
\newblock {\em Parametrix for wave equations on a rough background. {III}:
  {Space}-time regularity of the phase}, volume 401 of {\em Ast{\'e}risque}.
\newblock Paris: Soci{\'e}t{\'e} Math{\'e}matique de France (SMF), 2018.

\bibitem[Tou23a]{Touati2023a}
Arthur Touati.
\newblock Geometric optics approximation for the {E}instein vacuum equations.
\newblock {\em Communications in Mathematical Physics}, 402(3):3109--3200,
  2023.

\bibitem[Tou23b]{Touati2023b}
Arthur Touati.
\newblock High-frequency solutions to the constraint equations.
\newblock {\em Communications in Mathematical Physics}, 402(1):97--140, 2023.

\bibitem[Tou25]{Touati2025}
Arthur Touati.
\newblock Burnett’s conjecture in general relativity.
\newblock {\em Comptes Rendus Mécanique}, 353:455--476, 2025.

\end{thebibliography}
\end{document}